\documentclass[11pt]{amsart} 
\usepackage{amsxtra}
\usepackage{smartdiagram}
\usepackage{tikz}
\usetikzlibrary{calc,matrix,arrows,shapes}
\usepackage{pgfplots}
\tikzstyle{bubble}=[align=center, draw,inner sep=0pt,shape=ellipse,minimum height=1.25cm, minimum width=3cm]
\theoremstyle{plain}

\usepackage{graphicx} 

\usepackage{booktabs} 
\usepackage{array} 
\usepackage{paralist} 
\usepackage{verbatim} 
\usepackage{subfig} 
\usepackage{hyperref}
\usepackage{tikz}

\usepackage{amsmath}
\usepackage{amssymb}
\usepackage{amsthm}
\usepackage{mathrsfs}
\usepackage{enumerate}
\usepackage[all]{xy}
\usepackage{todonotes}
\usepackage{xkvltxp}
\usepackage{fixme}
\usetikzlibrary{calc}

\newtheorem{thm}{Theorem}[section]
\newtheorem{prop}[thm]{Proposition}
\newtheorem{cor}[thm]{Corollary}
\newtheorem{cnj}[thm]{Conjecture}

\newtheorem{lem}[thm]{Lemma}
\theoremstyle{definition}

\newtheorem{defn}[thm]{Definition}
\newtheorem{ex}[thm]{Example}

\newtheorem{remk}[thm]{Remark}

\newcommand{\tn}[1]{\textnormal{#1}}

\newcommand{\wt}[1]{\widetilde{#1}}

\newcommand{\bbr}{\mathbb R}
\newcommand{\bbz}{\mathbb Z}

\newcommand{\bbn}{\mathbb N}

\newcommand{\bbc}{\mathbb C}

\newcommand{\vspan}{\tn{span}}

\newcommand{\diff}{\backslash}

\newcommand{\id}{\tn{id}}

\newcommand{\Gr}{\mathrm{Gr}}
\newcommand{\Grad}{\mathrm{Gr}^{\mathrm{ad}}}
\newcommand{\Grrat}{\mathrm{Gr}^{\mathrm{rat}}}

\DeclareMathOperator {\ad} {\mathrm{ad}}
\DeclareMathOperator {\Be} {\mathrm{Be}}
\DeclareMathOperator {\Ai} {\mathrm{Ai}}
\DeclareMathOperator {\ord} {\mathrm{ord}}
\DeclareMathOperator {\cord} {\mathrm{cord}}
\newcommand{\bisl}{\mathcal{F}_x}
\newcommand{\bisr}{\mathcal{F}_y}
\newcommand{\bisls}{\mathcal{F}_{x,\text{sym}}}
\newcommand{\bisrs}{\mathcal{F}_{y,\text{sym}}}
\newcommand{\bislf}{\mathcal{B}_x}
\newcommand{\bisrf}{\mathcal{B}_y}
\newcommand{\floor}[1]{\left\lfloor #1 \right\rfloor}

\begin{document}

\title[Integral operators, bispectrality and Fourier algebras]
{Integral operators, bispectrality and 
\\ growth of Fourier algebras}
\author[W. Riley Casper]{W. Riley Casper}
\address{
Department of Mathematics \\
Louisiana State University \\
Baton Rouge, LA 70803 \\
U.S.A.}
\email{wcasper1@lsu.edu}
\author[Milen T. Yakimov]{Milen T. Yakimov}
\address{
Department of Mathematics \\
Louisiana State University \\
Baton Rouge, LA 70803 \\
U.S.A.}
\email{yakimov@math.lsu.edu}
\date{}
\begin{abstract}
In the mid 80's it was conjectured that every bispectral meromorphic function $\psi(x,y)$ 
gives rise to an integral operator $K_{\psi}(x,y)$ which possesses a commuting differential operator. 
This has been verified by a direct computation for several families of functions $\psi(x,y)$ where the commuting differential operator is of 
order $\leq 6$. We prove a general version of this conjecture for all self-adjoint bispectral functions of rank 1
and all self-adjoint bispectral Darboux transformations of the rank 2 Bessel and Airy functions. 
The method is based on a theorem giving an exact estimate of the second and first order terms of 
the growth of the Fourier algebra of each such bispectral function. From it we obtain
a sharp upper bound on the order of the commuting differential operator for the 
integral kernel $K_{\psi}(x,y)$ and a fast algorithmic procedure 
for constructing the differential operator; unlike the previous examples its order is arbitrarily high.
We prove that the above classes of bispectral functions are parametrized by infinite dimensional 
Grassmannians which are the Lagrangian loci of the Wilson adelic Grassmannian and its analogs
in rank 2.
\end{abstract}
\maketitle
\tableofcontents
\section{Introduction}
\subsection{Main result on bispectrality and integral operators}
\label{1.1}
A meromorphic function $\psi(x,y)$ on an open subset of $\bbc^2$ is called {\em{bispectral}} 
if it is an eigenfunction of a nonzero differential operator in each of the two variables. This notion was introduced 
by Duistermaat and Gr\"unbaum in \cite{DG} in relation to computer tomography and signal processing
and since then has found relations to many other areas: soliton equations \cite{BHYcmp,KR,W},
Calogero--Moser spaces and systems \cite{BW,BGK,W2}, orthogonal polynomials \cite{GH1,Iliev, Iliev2}, $W$-algebras and Kac--Moody algebras
\cite{BHYd,BZN}, automorphisms and ideal structure of rings of differential operators \cite{BW}.
To $\psi(x,y)$ one associates the integral operator 
$$
T_{\psi}: f(x)\mapsto \int_{\Gamma_1}K_{\psi}(x,y)f(y)dy \quad \text{with kernel} \quad K_{\psi}(x,y) = \int_{\Gamma_2} \psi(x,z)\psi(y,z) dy.
$$
Here $\Gamma_1$ and $\Gamma_2$ represent paths in $\bbc$, chosen such that $\psi(x,y)\in L^2(\Gamma_1\times\Gamma_2)$, along with other 
analytic convergence conditions.
The following property was conjectured in the mid 80's and was in the heart of the formalization of the notion of
bispectrality in \cite{DG}.
\medskip
\\
{\bf{Conjecture.}} {\em{Under mild conditions on the bispectral function
$\psi(x,y)$ and the paths $\Gamma_1, \Gamma_2$, the integral operator $T_{\psi}$ posses a commuting differential 
operator.}} 
\medskip
\\
The commutativity of integral and differential operators is formalized in \S \ref{form-comm}. 
The commutativity property in the conjecture has substantial applications to the {\em{analytic properties}} of the integral operator $T_{\psi}$  and
to the {\em{numerical computation}} of its spectrum and eigenvalues.

The conjecture was proved in the following special cases:
\begin{enumerate}
\item For the sine and Bessel bispectral functions $\psi(x,y) = e^{xy}$ and $\sqrt{xy} K_{\nu+1/2}(xy)$, Landau, Pollak, Slepian \cite{SP, LP1, S} 
constructed a 2nd order differential operator commuting with $T_{\psi}$ and used it in time-band limiting in signal processing. 
Mehta proved the fact for $\psi(x,y) = e^{xy}$ independently and applied it to random matrices \cite{Mehta}.
\item Tracy and Widom \cite{TW1, TW2} proved that for the Airy function $\psi(x,y) = \Ai(x+y)$, 
$T_{\psi}$ posses a commuting 2nd order order differential operator and applied the facts for the Bessel and Airy functions to asymptotics of Fredholm determinants
and scaling limits of random matrix models.
\item For two 1-parameter, 1-step Darboux transformations from the Bessel functions $\sqrt{xy} K_{\nu+1/2}(xy)$ with $\nu =1$ and $\nu=2$, Gr\"unbaum 
constructed a 4th and 6th order commuting differential operators  \cite{Grunbaum1996}.
\item Second order commuting differential and difference operators were constructed in several 
discrete-continuous and discrete-discrete situations starting with the Hermite, Laguerre and Jacobi polynomials \cite{Grunbaum1983} and 
expanding to several other situations \cite{Per,Perli}.
\item Commuting differential and difference operators were also constructed for matrix-valued generalizations of the
examples in (4), culminating in the work of Gr\"unbaum, Pacharoni, and Zurrian \cite{GPZ} which constructs 2nd order commuting differential/difference 
operators for matrix orthogonal polynomials whose weights satisfy a functional equation. 
\end{enumerate}
We obtain the following general solution of the conjecture:
\medskip
\\
{\bf{Theorem A.}} (i) {\em{For all self-adjoint bispectral Darboux transformations $\wt{\psi}(x,y)$ of the exponential function $e^{xy}$  and 
self-adjoint bispectral Darboux transformations of the rank 2 Airy $\Ai(x+y)$ and Bessel $\sqrt{xy} K_{\nu+1/2}(xy)$ functions, the integral operator with 
kernel $K_{\wt\psi}(x,y)$ posses a commuting differential operator which is formally symmetric (i.e., equals its formal adjoint).
}}

(ii) {\em{The class of bispectral functions in part (i) of rank 1 are parametrized by the points of the infinite dimensional 
Grassmannian which is the Lagrangian locus of the Wilson's adelic Grassamannian \cite{W}. 
The class of bispectral functions in part (i) of rank 2 are parametrized by the points of the Lagrangian loci of the 
infinite dimensional Grassmannians of rank 2 bispectral functions from \cite{BHYcmp}.
}}
\medskip

The following additional results are obtained:
\begin{enumerate}
\item {\bf{Effective upper bounds}} on the order of the commuting differential
operator are given in Theorems \ref{exp theorem}, \ref{airy theorem}, and \ref{bessel theorem}. 
\item An {\bf{algorithm for fast computation of the commuting differential operator}} is described in \S \ref{6.1}. The algorithm  
only relies on solving certain linear systems of equations and compositions of differential operators. 
\item Section 6 contains a variety of examples illustrating this algorithm, including two  
examples of integral operators, the first of which commutes simultaneously with two differential operators of {\bf{orders 6 and 8}} and the second 
commutes with a differential operator of {\bf{order 22}}. Commuting differential 
operators of such high orders were never found before.
\end{enumerate}

The first part of the theorem is proved in Section \ref{main thm a}. The second part is proved in Section \ref{classif}.
Section \ref{bDarboux} contains background on the notion of self-adjoint bispectral Darboux transformations and 
its role in previous works on the classification of bispectral functions.

Theorem A has the following important special cases:
\begin{enumerate}
\item By the main result of \cite{DG}, all bispectral meromorphic functions $\wt\psi(x,y)$ that are eigenfunctions of a 2nd order differential operator are obtained 
as iterated self-adjoint bispectral Darboux transformations from the Bessel functions $\sqrt{xy} K_{\nu+1/2}(xy)$ and are covered as special cases 
by Theorem A. Even in this special situation the theorem is new and few cases of it were previously known, cf. \S \ref{1.1} (3).
\item By the main result of \cite{W} (and its interpretation in \cite{BHYcmp}), all rank 1 bispectral functions $\wt\psi(x,y)$ are  bispectral Darboux transformations
from the exponential function $e^{xy}$. Such a function is called self-adjoint if it is an eigenfunction of a formally symmetric differential operator in $x$ and $y$.
All rank 1 self-adjoint bispectral functions $\wt\psi(x,y)$ are also covered by Theorem A.
\end{enumerate}
\subsection{Main result on growth of Fourier algebras}
\label{1.2}
Our proof of Theorem A is based on the construction of Fourier algebras associated to bispectral functions and a sharp estimate on their growth. 
For a bispectral meromorphic function $\psi(x,y)$ defined on a connected open subset $U\times V$ of $\bbc^2$, 
define the {\bf{left and right Fourier algebras of differential operators}} for $\psi$ by
\begin{align*}
\bisl(\psi) = \{\mathfrak d_x \in \mathfrak D(U) : \, &\text{there exists a differential operator $\mathfrak b_y \in \mathfrak D(V)$}  \\
&\text{satisfying $\mathfrak d_x \cdot\psi(x,y) =  \mathfrak b_y  \cdot \psi(x,y)$}\}
\end{align*}
and 
\begin{align*}
\bisr(\psi) = \{\mathfrak b_y \in \mathfrak D(V) : \, &\text{there exists a differential operator $\mathfrak d_x \in \mathfrak D(U)$}  \\
&\text{satisfying $\mathfrak d_x \cdot\psi(x,y) = \mathfrak b_y  \cdot \psi(x,y)$}\}.
\end{align*}
Here and below  $\mathfrak D(U)$ denotes the algebra of differential operators with meromorphic coefficients on $U$. 
The algebras $\bisl(\psi)$ and $\bisr(\psi)$ are anti-isomorphic, via the map $b_\psi: \bisl(\psi)\rightarrow\bisr(\psi)$ defined by
$$\mathfrak d_x \cdot\psi(x,y) = b_\psi(\mathfrak d_x )\cdot\psi(x,y),$$
see Proposition \ref{b-psi}. This map will be called the {\bf{generalized Fourier map}} of $\psi$. The terminology is motivated from 
the fact that for $\psi(x,y) = e^{xy}$ one recovers the usual Fourier transformation of differential operators. The key idea is that each operator 
$\mathfrak d_x \in \bisl(\psi)$ has well defined 
\[
\mbox{\bf{order}} \; \;  \ord \mathfrak d_x \in \bbn \quad \mbox{and} \quad \mbox{\bf{co-order}} \; \; \ord b_\psi( \mathfrak d_x ) \in \bbn.
\]
They give rise to an $\bbn\times\bbn$ {\bf{filtration}} of $\bisl(\psi)$ formed by the subspaces
$$
\bisl^{\ell,m}(\psi) = \{\mathfrak d\in \bisl(\psi): \ord(\mathfrak d) \leq \ell,\ \ord(b_\psi(\mathfrak b))\leq m\} 
\quad \mbox{for} \quad \ell, m \in \bbn.
$$
The spaces $\bisl^{\ell,m}(\psi)$  are finite dimensional and one can study the {\em{growth}} of their dimensions as
$\ell, m \to \infty$.
We will denote by $\bisls^{\ell,m}(\psi)$ the symmetric subspaces of $\bisl^{\ell,m}(\psi)$ consisting of operators 
$\mathfrak d_x \in \bisl^{\ell,m}(\psi)$ such that $\mathfrak d_x$ and $b_\psi(\mathfrak d_x)$ are 
{\bf{formally symmetric}}.
 
After the breakthrough work of Wilson \cite{W} bispectral functions are classified by rank (see Definition \ref{rank})
via a realization as {\bf{bispectral Darboux transformations}} 
in the terminology of \cite{BHYcmp}. This means that new bispectral functions $\wt\psi(x,y)$ are constructed from old ones $\psi(x,y)$ 
via a representation of the form
$$
\wt\psi(x,y) = \frac{1}{q(y)p(x)}\mathfrak{u} \cdot\psi(x,y) \ \
\text{and} \ \
\psi(x,y) = \wt{\mathfrak{u}} \cdot \frac{1}{\wt{q}(y)\wt{p}(x)} \wt \psi(x,y)
$$
for some differential operators $\mathfrak{u},\wt{\mathfrak u}\in\bisl(\psi)$ and polynomials $p(x)$, $\wt{p}(x)$ and $q(y)$, 
$\wt{q}(y)$ as in Definition \ref{bisp-Darb}. We call the pair $( \ord \mathfrak{u}, \ord \wt{\mathfrak u}) \in \bbn \times \bbn$ 
the {\bf{bidegree of the transformation}}. We refer the reader to Section \ref{bDarboux} for details on the relation of this construction to the classical 
Darboux process in terms of factorizations of differential operators.

For a fixed bispectral function $\psi(x,y)$, all bispectral Darboux transformations of $\psi(x,y)$ are classified by 
the points of infinite dimensional Grassmannians generalizing Wilson's adelic Grassmannian \cite{W,BHYcmp}.
All rank 1 bispectral functions (classified by the points of the original adelic Grassmannian) 
are  bispectral Darboux transformations from the function $e^{xy}$. 
\smallskip
\\
{\bf{Theorem B.}} (i) {\em{If $\wt{\psi}(x,y)$ is a bispectral Darboux transformation from the bispectral meromorphic function $\psi(x,y)$
satisfying natural mild assumptions, then for all $\ell, m \in \bbc$, 
$$
| \dim \bisl^{\ell,m}(\wt\psi)  - \dim  \bisl^{\ell,m}(\psi)  | \leq \mathrm{const}
$$
for a constant that is independent on $\ell$ and $m$. 
}}

(ii) {\em{If $\wt\psi(x,y)$ is a self-adjoint bispectral Darboux transformation of bidegree $(d_1, d_2)$ from the exponential $e^{yx}$, the Airy $\Ai(x+y)$, or Bessel 
functions  $\sqrt{xy} K_{\nu+1/2}(xy)$  with $\nu \in \bbc \backslash \bbn$, then 
$$
| \dim\bisls^{2\ell,2m}({\wt\psi}) - (\ell m + \ell + m + 2) | \leq d_1d_2
$$
for all $\ell, m \in \bbn$.
}} 
\smallskip
\\
The precise form of the first part of the theorem is given in Theorem \ref{control thm} (see also Corollary \ref{dimension cor} and Remark \ref{easy form}). 
The second part of the theorem is a combination of Theorem \ref{control thm} and Lemmas \ref{exp lemma}, \ref{airy lemma}, and \ref{bessel lemma}. 
\smallskip
\\
{\bf{Remark.}} (i) {\em{Theorem B establishes that the growth of $\dim \bisl^{\ell,m}(\wt\psi)$ and $\dim\bisls^{2\ell,2m}({\wt\psi})$ as functions of $\ell$ and $m$ is quadratic for huge classes of bispectral functions. The theorem gives exactly the quadratic and the two linear terms of the dimension functions and a sharp upper bound on their constant terms.}}

(ii) {\em{In a typical situation we start with a simple bispectral function $\psi(x,y)$ which is an eigenfunction of low degree differential operators {\em{(}}e.g. $e^{yx}$, $\Ai(x+y)$, or 
$\sqrt{xy} K_{\nu+1/2}(xy)${\em{)}} and we built a very complicated bispectral function $\wt{\psi}(x,y)$ which is an eigenfunction of differential operators of very high degrees. 
The remarkable feature of the theorem is that it proves that the Fourier algebras $\bisl(\wt\psi)$ and $\bisl(\psi)$ have the exactly same growth up to a linear term.}}

Theorems A and B fit in the following diagram:
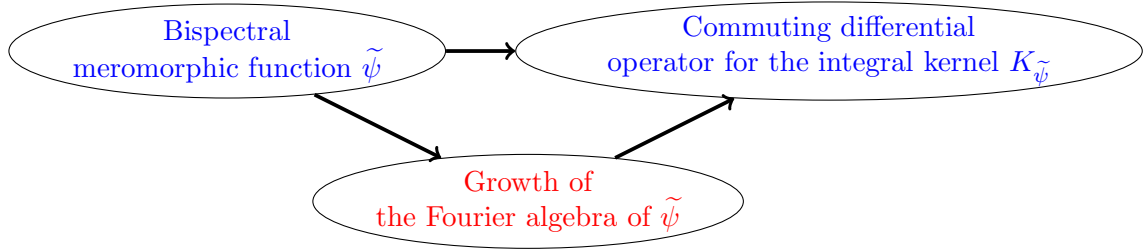
\begin{center}
\begin{tikzpicture}[x=2cm,y=-2cm]
 \path (4,0) node[bubble] (comm) {{\color{blue}{Commuting differential}} \\ {\color{blue}{operator for the integral kernel $K_{\wt{\psi}}$}}};
 \path (2,1) node[bubble] (growth) {{\color{red}{Growth of}} \\ {\color{red}{the Fourier algebra of $\wt\psi$}}};
 \path (0,0) node[bubble] (bisp) {{\color{blue}{Bispectral}} \\ {\color{blue}{meromorphic function $\wt\psi$}} };
 
 \draw[->,line width=0.5mm] (bisp) -- (comm);
 \draw[->,line width=0.5mm] (bisp) -- (growth);
 \draw[->,line width=0.5mm] (growth) -- (comm);
\end{tikzpicture}
\end{center}
\begin{enumerate}
\item {\em{Theorem A provides the horizontal implication, while Theorem B provides the left implication.}} 
\item {\em{The right implication is provided by Proposition \ref{fundamental dimension estimate}. 
The key idea of the right implication is that the 
commuting differential operator for the integral operator with kernel $K_{\wt{\psi}}$ is constructed as an element 
of the space $\dim\bisls^{2\ell,2m}({\wt\psi})$.}}
\end{enumerate}
While the horizontal implication (the conjecture in \S \ref{1.1}) has substantial analytic and numerical applications, the link that proves this relationship (in red)
is of algebraic nature.

This paper is a continuation of our collaboration with F. Alberto Gr\"unbaum \cite{GrYa}, which aimed at constructing a 
bridge between bispectrality and commuting integral and diferential operators by controlling the sizes 
of Fourier algebras under Darboux transformation. This strategy was announced in \cite{GrYa}. Among other things, the results in this paper 
fully justify the statements in \cite{GrYa} which contained no proofs. 

When bispectral functions $\wt{\psi}(x,y)$ are converted to wave functions $\wt{\Psi}(x,y)$ for the KP hierarchy via asymptotic expansions at $\infty$, 
the elements of the right Fourier algebra $\bisr(\wt\psi)$ correspond to the $W$-constraints for the wave function $\wt{\Psi}(x,y)$. The latter generalize the 
string equation for the Airy wave function \cite{Wit,Kon,vM} which played a key role in quantum gravity and intersection theory on moduli spaces
of curves. Consequently, in algebraic geometry the $W$-constraints of Gromov--Witten invariants and the total descendant potential of a simple singularity
were extensively studied by Okounkov--Pandharipande \cite{OP},  Bakalov--Milanov \cite{BM} and in many other works.

The $W$-constraints for wave functions of the Toda and multi-component KP hierarchies have played an important role in many situations.
In the theory of random matrices, Adler and van Moerbeke \cite{AvM99,AvM01,AvM07} used Virasoro constraints 
to derive a system of partial differential equations for the distributions of the spectra of coupled random matrices.
Bakalov, Horozov, and Yakimov \cite{BHYd} used $W$-constraints to prove that all tau-functions in the quasifinite representations 
of the $W_{1+\infty}$-algebra (in the sense of Frenkel, Kac, Radul, and Wang \cite{FKRW}) with highest weight vectors given by Bessel 
tau-functions are bispectral tau-functions.  
Our growth estimate theorem for the the right Fourier algebra $\bisr(\wt\psi)$ translates 
directly to a growth estimate theorem for the algebra of $W$-constraints of a wave function. In this way 
Theorem B has potential applications to the above topics in random matrices and representations of the $W_{1+ \infty}$-algebra.

More generally, Virasoro constraints are linked to the isomonodromic deformations approach to random matrices from the 
works of Palmer \cite{Pal}, Harnad--Tracy--Widom \cite{HTW}, Its--Harnad \cite{HI}, and Borodin--Deift \cite{BD}. Of particular
interest is the relation to the series of results in the literature on random matrices showing that the Fredholm determinants of various
kernels are solutions of Painlev\'e equation. The latter naturally appear in other problems, e.g. representations of $U(\infty),$
see Borodin--Olshanski \cite{B, BO}. We expect that the growth estimates from Theorem B on the algebra of $W$-constraints will 
also have applications in these respects. The main idea here is to apply the full algebra of $W$-constraints vs concrete 
Virasoro constraints.

We will use the following conventions for the notation in the paper. Differential operators will always appear in the Gothic font, e.g. $\mathfrak d = \partial_x^2 - \frac{2}{x^2}$. 
When needed, we will write $\mathfrak d_x$ in place of $\mathfrak d$ to emphasize the action of the operator in the variable $x$.  
In this case $\mathfrak d_y$ will represent the same operator, but acting in the variable $y$.  
We will denote by $\mathfrak D(\bbc[x])$ and $\mathfrak D(\bbc(x))$ the algebras of differential operators in $x$ with polynomial and rational coefficients.
For a noncommutative algebra $R$ and $r, s \in R$, we set
$$
\ad_r(s) = [r,s] = rs-sr.
$$
\section*{Acknowledgements} 
We are grateful to F. Alberto Gr\"unbaum for his insight, suggestions and support throughout the different stages of this 
project. The research of M.T.Y. was supported by NSF grant DMS-1601862 and Bulgarian Science Fund grant H02/15.
\section{Bispectrality and double filtrations of Fourier algebras}
\label{bisp-symm}
In this section we collect background material on bispectral functions and Fourier algebras associated to them.
We introduce one of the key players in the paper -- a double filtration on each such algebra.
\subsection{Bispectral meromorphic functions and associated Fourier algebras}
\label{2.1}
We first review the definition of a bispectral meromorphic function.
For an open subset $U$ of $\bbc$, denote by $\mathfrak D(U)$ the algebra of differential operators on $U$ with meromorphic coefficients.
\begin{defn}
A nonconstant meromorphic function $\psi(x,y)$ defined on a connected open subset $U\times V$ of $\bbc^2$ is said to be {\bf{bispectral}} if there exist differential operators 
$\mathfrak d\in\mathfrak D(U)$ and $\mathfrak b\in\mathfrak D(V)$ such that
$$
\mathfrak d_x\cdot\psi(x,y) = g(y)\psi(x,y)\ \ \text{and}\ \ \mathfrak b_y\cdot\psi(x,y) = f(x)\psi(x,y)
$$
for some nonconstant functions $f(x)$ and $g(y)$ meromorphic on $U$ and $V$, respectively.
\end{defn}

There are three examples of bispectral meromorphic functions that will play a fundamental role in this paper: the exponential, Airy and Bessel functions. 
We will refer to them as the {\bf{elementary bispectral functions}}.
Due to their central role, we will introduce a notation for each that will be used throughout the paper.
The notation and values of the elementary bispectral functions is given in the table in Figure \ref{elementary bispectral functions}, where for $\nu\in\bbc$ the expression $K_{\nu}(t)$ denotes the Bessel function of the second kind and $\Ai(t)$ denotes Airy function of the first kind.
\begin{center}
\begin{figure}[htp]
\begin{tabular}{|c|c|c|}\hline
name & function & operator\\\hline
bispectral exponential & $\psi_{\exp}(x,y) = e^{xy}$ & $\partial_x$\\\hline
bispectral Airy        & $\psi_{\Ai}(x,y) = \Ai(x+y)$ & $\mathfrak d_{\Ai} = \partial_x^2 - x$\\\hline
bispectral Bessel      & $\psi_{\Be(\nu)}(x,y) = \sqrt{xy} K_{\nu+1/2}(xy)$ & $\mathfrak d_{\Be(\nu)} = \partial_x^2 - \frac{\nu(\nu+1)}{x^2}$\\\hline
\end{tabular}
\caption{The elementary bispectral functions.}
\label{elementary bispectral functions}
\end{figure}
\end{center}
The bispectral Bessel functions $\psi_{\Be(\nu)}(x,y)$ are in particular a one parameter family of functions indexed by $\nu$.
For integer values of $\nu$ the Bessel functions simplify to rational functions multiplied by exponential functions.
For example $\psi_{\Be(0)}(x,y)$ is proportional to $e^{xy}$ and $\psi_{\Be(1)}(x,y)$ is proportional to $e^{xy}(1-(xy)^{-1})$.

Each of the elementary bispectral functions are, as the name suggests, bispectral and thus are families of eigenfunctions in both variables $x$ and $y$.
For example the bispectral exponential function satisfies
$$\partial_x\cdot\psi_{\exp}(x,y) = y\psi_{\exp}(x,y) \ \  \text{and} \ \ \partial_y\cdot\psi_{\exp}(x,y) = x\psi_{\exp}(x,y).$$
The Bessel and Airy bispectral functions are eigenfunctions of the Bessel and Airy operators $\mathfrak d_{\Be(\nu)}$ and $\mathfrak d_{\Ai}$ 
from the last column of the table in Figure \ref{elementary bispectral functions}.
The associated differential equations are
\begin{equation}
\label{Bessel-eq}
\mathfrak{d}_{\Be(\nu), x} \cdot\psi_{\Be(\nu)}(x,y) = y^2\psi_{\Be(\nu)}(x,y)\ \ \text{and}\ \  \mathfrak{d}_{\Be(\nu), y} \cdot\psi_{\Be(\nu)}(x,y) = x^2\psi_{\Be(\nu)}(x,y),
\end{equation}
\begin{equation}
\label{Airy-eq}
\mathfrak{d}_{\Ai,x} \cdot \psi_{\Ai} (x,y) = y \psi_{\Ai}(x,y)\ \ \text{and}\ \ \mathfrak{d}_{\Ai,y}\cdot\psi_{\Ai}(x,y) = x\psi_{\Ai}(x,y).
\end{equation}

Bispectral meromorphic functions in general satisfy a wide collection of differential equations and the associated operators form algebras.
\begin{defn}
Let $\psi(x,y)$ be a bispectral meromorphic function defined on the connected open subset $U\times V$ of $\bbc^2$. 
We define the {\bf{left}} and {\bf{right Fourier algebras}} of differential operators for $\psi$ by
\begin{align*}
\bisl(\psi) = \{\mathfrak d \in \mathfrak D(U) : \, &\text{there exists a differential operator $\mathfrak b \in \mathfrak D(V)$}  \\
&\text{satisfying $\mathfrak d\cdot\psi(x,y) =  \mathfrak b \cdot \psi(x,y)$}\}
\end{align*}
and 
\begin{align*}
\bisr(\psi) = \{\mathfrak b \in \mathfrak D(V) : \, &\text{there exists a differential operator $\mathfrak d \in \mathfrak D(U)$}  \\
&\text{satisfying $\mathfrak d\cdot\psi(x,y) = \mathfrak b \cdot \psi(x,y)$}\}.
\end{align*}
The algebras $\bisl(\psi)$ and $\bisr(\psi)$ come with distinguished subalgebras
\begin{align*}
\bislf(\psi) &= \{\mathfrak d\in\bisl(\psi): \text{for some meromorphic function $g(y)$, $\mathfrak d\cdot\psi(x,y) = g(y)\psi(x,y)$}\}. \\
\bisrf(\psi) &= \{\mathfrak b\in\bisr(\psi): \text{for some meromorphic function $f(x)$, $\mathfrak b\cdot\psi(x,y) = f(x)\psi(x,y)$}\}.
\end{align*}
The algebras $\bislf(\psi)$ and $\bisrf(\psi)$ are called the algebras of {\bf{left}} and {\bf{right bispectral differential operators}} for $\psi(x,y)$.
\end{defn}
The function $\psi(x,y)$ is an eigenfunction of all operators in these two algebras.
In a neighborhood of a sufficiently nice point $(x_0,y_0)$ of the domain of $\psi(x,y)$, the coefficients of the operators in $\bislf(\psi)$ and $\bisrf(\psi)$ will be analytic.
Then by an appropriate change of variables, we may assume that $\bislf(\psi)$ and $\bisrf(\psi)$ both contain a nonconstant differential operator whose leading coefficient is constant.
Consequently by \cite[Eq. (1.20)]{DG}, all functions $f(x)$ and $g(y)$ that appear in the definition of $\bislf(\psi)$ and $\bisrf(\psi)$ are polynomial.
Furthermore the algebras $\bislf(\psi)$ and $\bisrf(\psi)$ are necessarily commutative, so every element in each algebra will have constant leading coefficient.

\begin{remk}
\label{Fourier}
The name Fourier algebras is inspired by the case where $\psi$ is the simplest bispectral function $\psi_{\exp}(x,y) = e^{xy}$.  
In this case $\bisl(\psi_{\exp})= \mathfrak D(\bbc[x])$, the algebra of differential operators with polynomial coefficients.  
Similarly, $\bisr(\psi_{\exp}) = \mathfrak D(\bbc[y])$.  For any $\mathfrak d\in\mathfrak D(\bbc[x])$ with
$$\mathfrak d = \sum_{j=0}^\ell\sum_{k=0}^m a_{jk}x^k\partial_x^j$$
we have we have $\mathfrak d\cdot e^{xy} = \mathfrak b\cdot e^{xy}$ for $\mathfrak b\in\mathfrak D(\bbc[y])$ defined by
$$\mathfrak b = \sum_{j=0}^\ell\sum_{k=0}^m a_{jk}y^j\partial_y^k.$$
The operator $\mathfrak b$ is the Fourier transform of $\mathfrak d$.
\qed
\end{remk}

\begin{prop}
\label{b-psi}
Let $\psi(x,y)$ be a bispectral meromorphic function defined on $U\times V$.  The algebras $\bisl(\psi)$ and $\bisr(\psi)$ are anti-isomorphic via the map 
$b_\psi: \bisl(\psi)\rightarrow\bisr(\psi)$ defined by sending $\mathfrak d\in\bisl(\psi)$ to $\mathfrak b\in\bisr(\psi)$, where $\mathfrak d\cdot\psi(x,y) = \mathfrak b\cdot\psi(x,y)$.
\end{prop}
\begin{proof}
Suppose $\mathfrak d\in\bisl(\psi)$ with $\mathfrak d\cdot\psi(x,y) = g(y)\psi(x,y)$ for some $g(y)\in\bisr(\psi)$ nonconstant.  If there exists a nonzero $\mathfrak b\in\bisr(\psi)$ satisfying $\mathfrak b\cdot\psi(x,y) = 0$, then $\ad_{g(y)}^k(\mathfrak b)\cdot\psi(x,y) = 0$ for all $k$.  However, for $k$ equal to the order of $\mathfrak b$, the operator $\ad_{g(y)}^k(\mathfrak b)$ will be equal to a nonzero function $h(y)$.  This would imply $h(y)\psi(x,y) = 0$, and that $\psi(x,y)$ is therefore $0$, a contradiction.  Thus if $\mathfrak b\in\bisr(\psi)$ satisfies $\mathfrak b\cdot\psi(x,y) = 0$, then $\mathfrak b = 0$.  Similarly, if $\mathfrak a\in\bisl(\psi)$ satisfies $\mathfrak a\cdot\psi(x,y) = 0$, then $\mathfrak a = 0$.

Using this, we see that for every $\mathfrak a\in\bisl(\psi)$ there exists a unique $\mathfrak b\in\bisl(\psi)$ such that $\mathfrak a\cdot\psi(x,y) = \mathfrak b\cdot\psi(x,y)$.  This shows that the map $b_\psi$ is well-defined and bijective.  A simple argument shows that $b_\psi$ is an anti-isomorphism.  This completes the proof.
\end{proof}

\begin{defn}
We call the map $b_\psi: \bisl(\psi)\rightarrow\bisr(\psi)$ from the previous proposition the {\bf{generalized Fourier map}} or the {\bf{bispectral anti-isomorphism}} of $\psi$.
\end{defn}
\begin{ex} 
\label{Fourier-ex}
(1) By Remark \ref{Fourier},  for the exponential bispectral function $\psi_{\exp}$, the Fourier algebras are 
the algebras of differential operators with polynomial coefficients 
$$
\bisl(\psi_{\exp})= \mathfrak D(\bbc[x]), \ \ \bisr(\psi_{\exp}) = \mathfrak D(\bbc[y])
$$
and the anti-isomorphism $b_{\psi_{\exp}} : \mathfrak D(\bbc[x]) \to \mathfrak D(\bbc[y])$ is the Fourier transform given by $\partial_x  \mapsto y$ and $x \mapsto \partial_y$.
The algebras of bispectral operators $\bislf(\psi_{\exp})$ and $\bislf(\psi_{\exp})$ are the algebras of differential operators with constant coefficients in $x$ and $y$, respectively.

(2) The Bessel bispectral functions $\psi_{\Be(\nu)}(x,y)$ satisfy
$$
x \partial_x \cdot \psi_{\Be(\nu)}(x,y) = y \partial_y \cdot \psi_{\Be(\nu)}(x,y)
$$
in addition to the equations \eqref{Bessel-eq}. The left and right Fourier algebras are given by 
$$
\bisl(\psi_{\Be(\nu)}) = \langle \mathfrak d_{\Be(\nu),x}, x \partial_x, x^2 \rangle \ \ \text{and} \ \
\bisr(\psi_{\Be(\nu)}) = \langle \mathfrak d_{\Be(\nu),y}, y \partial_y, y^2 \rangle
$$
for $\nu \in \bbc \backslash \bbz$. (In the case $\nu \in \bbz$ the Fourier algebras are bigger, but this will not a play a role in the paper.)
Here $\langle \mathfrak d_1, \ldots, \mathfrak d_k \rangle$ denotes the algebra of differential operators with generators
$\mathfrak d_1, \ldots, \mathfrak d_k$.
The generalized Fourier map $b_{\psi_{\Be(\nu)}} : \bisl(\psi_{\Be(\nu)}) \to \bisl(\psi_{\Be(\nu)})$ is given by
$$
\mathfrak d_{\Be(\nu),x} \mapsto  y^2, \ \ x \partial_x \mapsto y \partial_y, \ \ x^2 \mapsto \mathfrak d_{\Be(\nu),y}.
$$
The algebras of bispectral operators $\bislf(\psi_{\Be(\nu)})$ and $\bislf(\psi_{\Be(\nu)})$ are the polynomial algebras in the Bessel operators 
$\mathfrak d_{\Be(\nu),x}$ and $\mathfrak d_{\Be(\nu),y}$, respectively.

(3) The Airy bispectral function $\psi_{\Ai}(x,y)$ satisfies
$$
\partial_x \cdot \psi_{\Ai}(x,y) = \partial_y \cdot \psi_{\Ai}(x,y)
$$
in addition to the equations \eqref{Airy-eq}. The left and right Fourier algebras associated to it are the algebras of differential operators
with polynomial coefficients
$$
\bisl(\psi_{\Ai}) = \mathfrak D(\bbc[x]) \ \ \text{and} \ \
\bisr(\psi_{\Ai}) = \mathfrak D(\bbc[y]).
$$
The generalized Fourier map $b_{\psi_{\Ai}} : \bisl(\psi_{\Ai}) \to \bisl(\psi_{\Ai})$ is given by
$$
\partial_x \mapsto \partial_y, \ \ x \mapsto \mathfrak d_{\Ai,y}.
$$
It satisfies $\mathfrak d_{\Ai,x} \mapsto  y$. 
The algebras of bispectral operators $\bislf(\psi_{\Ai})$ and $\bislf(\psi_{\Ai})$ are the polynomial algebras in the Airy operators 
$\mathfrak d_{\Ai,x}$ and $\mathfrak d_{\Ai,y}$, respectively.
\qed
\end{ex}
\subsection{Double filtrations of Fourier algebras}
Using the generalized Fourier map of a bispectral function $\psi(x,z)$, 
we define natural filtrations of the left and right Fourier algebras $\bisl(\psi)$ and $\bisr(\psi)$ as follows.
\begin{defn}\label{filtration def}
We define the {\bf{co-order}} of an operator $\mathfrak d\in\bisl(\psi)$, denoted $\cord(\mathfrak d)$, to be the order of $b_\psi(\mathfrak d)$.  
Similarly, we define the {\bf{co-order}} of $\mathfrak b\in\bisr(\psi)$, again denoted $\cord(\mathfrak b)$, 
to be the order of $b_\psi^{-1}(\mathfrak b)$.  We define {\bf{$\bbn\times\bbn$-filtrations}} of the Fourier algebras $\bisl(\psi)$ and $\bisr(\psi)$ by
\begin{align*}
\bisl^{\ell,m}(\psi) &= \{\mathfrak d\in\bisl(\psi): \text{$\ord(\mathfrak d)\leq \ell$ and $\cord(\mathfrak d)\leq m$}\},\\
\bisr^{m,\ell}(\psi)& = \{\mathfrak b\in\bisl(\psi): \text{$\ord(\mathfrak b)\leq m$ and $\cord(\mathfrak b)\leq \ell$}\}.
\end{align*}
\end{defn}

\begin{remk} By Proposition \ref{b-psi}, the generalized Fourier map $b_\psi$ restricts to an isomorphism $\bisl^{\ell,m}(\psi)\rightarrow\bisr^{m,\ell}(\psi)$. 
Moreover, under the above filtration
$$\bislf(\psi) = \vspan\bigcup_{\ell=0}^\infty\bisl^{\ell,0}(\psi)\ \ \text{and}\ \ \bisrf(\psi) = \vspan\bigcup_{m=0}^\infty\bisr^{0,m}(\psi).$$
\end{remk}

We pause now to consider the filtrations of the Fourier algebras associated to the bispectral functions $\psi_{\exp}$, $\psi_{\Be(\nu)}$ and 
$\psi_{\Ai}$, continuing Example \ref{Fourier-ex}.
\begin{ex} 
\label{Filtr-ex}
(1) For the bispectral exponential function $\psi_{\exp}$ we have 
\begin{align*}
\bisl^{\ell,m}(\psi_{\exp}) = \vspan\{x^k\partial_x^j: 0\leq j\leq \ell,\ 0\leq k\leq m\}, \\
\bisr^{m,\ell}(\psi_{\exp}) = \vspan\{y^j\partial_y^k: 0\leq j\leq \ell,\ 0\leq k\leq m\}.
\end{align*}

(2) The filtrations of the Fourier algebras associated to the bispectral Bessel functions $\psi_{\Be(\nu)}$ for $\nu \in \bbc \backslash \bbz$
are given by
\begin{align*}
\bisl^{2\ell,2m}(\psi_{\Be(\nu)}) &= \vspan\{x^{2k}\mathfrak d_{\Be(\nu),x}^j: 0\leq j\leq \ell,\ 0\leq k\leq m\} \\
&\oplus\vspan\{x^{2k}x\partial_x\mathfrak d_{\Be(\nu), x}^j: 0\leq j < \ell,\ 0\leq k <  m\}, \\
\bisl^{2m,2\ell}(\psi_{\Be(\nu)}) & = \vspan\{y^{2j}\mathfrak d_{\Be(\nu), y}^k: 0\leq j\leq \ell,\ 0\leq k\leq m\} \\
&\oplus\vspan\{y^{2j}y\partial_y\mathfrak d_{\Be(\nu), y}^k: 0\leq j < \ell,\ 0\leq k <  m\}
\end{align*}
for even indices. Similar formulas hold for odd indices, we leave the details to the reader.

(3) For the bispectral Airy function $\psi_{\Ai}$ the filtrations of the left and right Fourier algebras are
\begin{align*}
\bisl^{2\ell,2m}(\psi_{\Ai}) &= \vspan\{x^k\mathfrak d_{\Ai,x}^j: 0\leq j\leq \ell,\ 0\leq k\leq m\}
\\
&\oplus\vspan\{x^k\partial_x\mathfrak d_{\Ai,x}^j: 0\leq j < \ell,\ 0\leq k <  m\},\\
\bisr^{2m,2\ell}(\psi_{\Ai}) &= \vspan\{y^j\mathfrak d_{\Ai,y}^k: 0\leq j\leq \ell,\ 0\leq k\leq m\}
\\
& \oplus\vspan\{y^j\partial_y\mathfrak d_{\Ai,y}^k: 0\leq j < \ell,\ 0\leq k <  m\}
\end{align*}
for even indices. Similar formulas hold for odd indices.

All three facts are easily deduced from the facts for the left and right Fourier algebras and generalized Fourier maps in Example \ref{Fourier-ex}. 
\qed
\end{ex}
\section{Self-adjoint bispectral Darboux transformations}
\label{bDarboux}
This section contains background material on bispectral Darboux  transformations and the classification of bispectral functions. We 
introduce the second main player in the paper -- self-adjoint Darboux transformations and prove that they correspond bijectively to
self-adjoint bispectral functions. 
\subsection{Bispectral Darboux transformations and classification results}
\begin{defn}
\label{bisp-Darb}
Let $\psi(x,y)$ be a bispectral meromorphic function on $U\times V$.  A {\bf{bispectral Darboux transformation}} $\wt\psi(x,y)$ of $\psi(x,y)$ is a function 
satisfying 
$$
\wt\psi(x,y) = \frac{1}{q(y)p(x)}\mathfrak{u} \cdot\psi(x,y) \ \
\text{and} \ \
\psi(x,y) = \wt{\mathfrak{u}} \cdot \frac{1}{\wt{q}(y)\wt{p}(x)} \wt \psi(x,y)
$$
for some pairs of polynomials $p(x)$, $\wt{p}(x)$ and $q(y)$, $\wt{q}(y)$, and some differential operators $\mathfrak{u},\wt{\mathfrak u}\in\bisl(\psi)$.

We call the pair $(d_1,d_2)$ the {\bf{order}} of the bispectral Darboux transformation, where $d_1 = \ord(\mathfrak u)$ and $d_2 = \cord(\mathfrak u)$ 
are the order and co-order of $\mathfrak u$ (with the latter defined as in Definition \ref{filtration def} above).
\end{defn}
It follows from the definition that
$$
\wt{\mathfrak{u}}  \frac{1}{\wt{p}(x) p(x)}\mathfrak{u} \cdot\psi(x,y) = \wt{q}(y) q(y) \psi(x,y).
$$
Therefore
\begin{equation}
\label{factor}
\wt{\mathfrak{u}}  \frac{1}{\wt{p}(x) p(x)}\mathfrak{u} \in \bislf \ \ \text{and} \ \
b_{\psi} \left( \wt{\mathfrak{u}}  \frac{1}{\wt{p}(x) p(x)}\mathfrak{u} \right) = \wt{q}(y) q(y).
\end{equation}
Similar statements hold when the roles of $\psi(x,y)$ and $\wt{\psi}(x,y)$ are interchanged.
\begin{remk} 
\label{normal-rem}
If $\psi(x,y)$ is a bispectral function on $U \times V$, then $f(x) g(y) \psi(x,y)$ is also a bispectral 
function on $U \times V$ for all meromorphic functions $f(x)$ and $g(y)$ on $U$ and $V$. Because of this, the choice 
of $p(x)$ and $q(y)$ in Definition \ref{bisp-Darb} is merely a choice of normalization. The next paragraph describes the {\bf{standard normalization}}.

The new bispectral function $\wt\psi(x,y)$ can be represented in two dual ways as a transformation of $\psi(x,y)$ in the $x$ and $y$ variables:
\begin{equation}
\label{dual-poly}
\wt\psi(x,y) = \frac{1}{q(y)p(x)}\mathfrak{u} \cdot\psi(x,y) = \frac{1}{p(x)q(y)} b_{\psi} (\mathfrak{u}) \cdot\psi(x,y)
\end{equation}
The polynomials $p(x)$ and $q(y)$ are uniquely determined from the normalization that 
\begin{equation}
\label{normalization}
\mbox{the differential operators} \quad \frac{1}{p(x)}\mathfrak{u} \quad  \mbox{and}  \quad \frac{1}{q(y)} b_{\psi} (\mathfrak{u}) \quad \mbox{are monic}. 
\end{equation}
\end{remk}
\begin{remk}
In Theorem \ref{bispectral is symmetric} we prove that if $\wt p(x,y)$ is a bispectral Darboux transformation of $\psi(x,y)$ and $\mathfrak u,\wt{\mathfrak u},p(x),\wt p(x),q(y), \wt q(y)$ are as in the definition of a bispectral Darboux transformation, then necessarily $\wt p(x)\wt{\mathfrak u}\frac{1}{\wt p(x)}\in\bisl(\wt \psi)$.
In fact the definition of a bispectral Darboux transformation can be replaced with the equivalent condition that there exist differential operators $\mathfrak u\in\bisl(\psi)$, $\mathfrak v\in\bisl(\wt\psi)$ and polynomials $p(x),\wt p(x),q(y),\wt q(y)$ such that
$$
\wt\psi(x,y) = \frac{1}{q(y)p(x)}\mathfrak u\cdot\psi(x,y)\ \ \text{and}\ \ 
   \psi(x,y) = \frac{1}{\wt q(y)\wt p(x)}\mathfrak v\cdot\wt \psi(x,y).
$$
In particular, this latter definition is manifestly symmetric.
\end{remk}
The following result of \cite{BHYphl} establishes general bispectrality properties of the transformations in Definition \ref{bisp-Darb}.
\begin{thm} 
\label{bisp-Darb-thm}
[Bakalov--Horozov--Yakimov, \cite[Theorem 4.2]{BHYphl}]
Let $\psi(x,y)$ be a bispectral function on $U\times V$ and let $\wt\psi(x,y)$ be a bispectral Darboux transformation from $\psi(x,y)$ with the notation in 
Definition \ref{bisp-Darb}. Then $\wt{\psi}(x,y)$ is also a bispectral function which satisfies the spectral equations
\begin{align*}
\frac{1}{p(x)} \mathfrak{u} \wt{\mathfrak{u}}  \frac{1}{\wt{p}(x)} \cdot \wt{\psi}(x,y) &= q(y) \wt{q}(y) \wt{\psi}(x,y), \\
\frac{1}{q(y)} b_{\psi}(\mathfrak{u}) b_{\psi}( \wt{\mathfrak{u}})  \frac{1}{\wt{q}(y)} \cdot \wt{\psi}(x,y) & = p(x) \wt{p}(x) \wt{\psi}(x,y).
\end{align*} 
Dually to \eqref{factor}, we have 
\begin{equation}
\label{factor-2}
b_{\psi} (\wt{\mathfrak{u}})  \frac{1}{\wt{q}(y) q(y)} b_{\psi}(\mathfrak{u}) \in \bisrf \ \ \text{and} \ \
b_{\psi} (\wt{p}(x) p(x) ) = b_{\psi} (\wt{\mathfrak{u}})  \frac{1}{\wt{q}(y) q(y)} b_{\psi}(\mathfrak{u}).
\end{equation}
\end{thm}

The next theorem shows that being a bispectral Darboux transformation is a symmetric condition.
\begin{thm}\label{bispectral is symmetric}
Let $\psi(x,y)$ be a bispectral function and suppose that $\wt\psi(x,y)$ is a bispectral Darboux transformation of $\psi(x,y)$.
Then $\psi(x,y)$ is also a bispectral Darboux transformation of $\wt\psi(x,y)$ and $\wt p(x)\wt{\mathfrak u}\frac{1}{\wt p(x)},\frac{1}{p(x)}\mathfrak u p(x)\in \bisl(\psi)$ with
$$
b_{\wt\psi}\left(\wt p(x)\wt{\mathfrak u}\frac{1}{\wt p(x)}\right) = q(y)b_\psi(\wt{\mathfrak u}) \frac{1}{q(y)}\ \ \text{and}\ \ 
b_{\wt\psi}\left(\frac{1}{p(x)}\mathfrak up(x)\right) = \frac{1}{q(y)}b_\psi(\mathfrak u)q(y).
$$
\end{thm}
\begin{proof}
Let $\mathfrak u,\wt{\mathfrak u},p(x),\wt p(x),q(y),$ and $\wt q(y)$ be as in the definition of bispectral Darboux transformations for the bispectral Darboux transformation from $\psi(x,y)$ to $\wt\psi(x,y)$.
Then
$$\wt p(x)\wt{\mathfrak u}\frac{1}{\wt p(x)}\cdot\wt\psi(x,y)  = \wt p(x)\wt{\mathfrak u}\frac{1}{\wt p(x)p(x)q(y)}\mathfrak u\cdot\psi(x,y)  = \wt p(x)\wt q(y)\cdot\psi(x,y).$$
Therefore by Theorem \ref{bisp-Darb-thm} we find
$$p(x)\wt p(x)\wt{\mathfrak u}\frac{1}{\wt p(x)}\cdot \wt \psi(x,y) = p(x)\wt p(x)\wt q(y)\psi(x,y) = \wt q(y)b_\psi(\wt{\mathfrak u})\frac{1}{\wt q(y)q(y)}b_\psi(\mathfrak u)\cdot\psi,$$
so that
$$\wt p(x)\wt{\mathfrak u}\frac{1}{\wt p(x)}\cdot \wt \psi(x,y) = \wt q(y)b_\psi(\wt{\mathfrak u})\frac{1}{p(x)\wt q(y)q(y)}b_\psi(\mathfrak u)\cdot\psi = \wt q(y)b_\psi(\wt{\mathfrak u})\frac{1}{\wt q(y)}\cdot\wt\psi(x,y).$$
Thus $\wt p(x)\wt{\mathfrak u}\frac{1}{\wt p(x)}\in \bisl(\wt\psi)$.
Moreover
\begin{align*}
\frac{1}{p(x)}\mathfrak u p(x)\cdot\wt\psi(x,y)
  & = \frac{1}{p(x)q(y)}\mathfrak u^2\cdot\psi(x,y)\\
  & = \frac{1}{p(x)q(y)}b_\psi(\mathfrak u)^2\cdot\psi(x,y)\\
  & = \frac{1}{q(y)}b_\psi(\mathfrak u)q(y)\frac{1}{p(x)q(y)}b_\psi(\mathfrak u)\cdot\psi(x,y)\\
  & = \frac{1}{q(y)}b_\psi(\mathfrak u)q(y)\cdot \wt \psi(x,y)
\end{align*}
so that $\frac{1}{p(x)}\mathfrak u p(x)\in \bisl(\wt\psi)$.
Then since
$$
\frac{1}{\wt q(y)\wt p(x)}\left(\wt p(x)\wt{\mathfrak u}\frac{1}{\wt p(x)}\right)\cdot \wt \psi(x,y) = \psi(x,y) \ \ \text{and} \ \ 
\left(\frac{1}{p(x)}\mathfrak u p(x)\right)\frac{1}{p(x)q(y)}\cdot\psi(x,y) = \wt \psi(x,y)
$$
we see that $\psi$ is a bispectral Darboux transformation of $\wt\psi$.
\end{proof}

\begin{defn}\label{rank}
We define the {\bf{rank}} of a bispectral function $\psi(x,y)$ to be the greatest common divisor of the orders of the operators in the left bispectral algebra $\bislf(\psi)$.
\end{defn}
In the cases when classification results are available, one can show that the rank of a bispectral function $\psi(x,y)$ also equals the greatest common divisor of the orders 
of the operators in the right bispectral algebra $\bisrf(\psi)$, but there is currently no direct proof of this fact.

In \cite{W} Wilson introduced the powerful idea that the classification of bispectral functions should be performed on a per-rank basis in which case 
one sees a deep geometric picture of the moduli spaces of such functions. Bispectral Darboux transformations preserve the rank and 
are especially suited for these purposes.

\begin{thm}
\label{rank1}
[Wilson \cite{W}]
The rank $1$ bispectral functions are precisely the bispectral Darboux transformations of the exponential function $\psi_{\exp}(x,y)$.
\end{thm}

Wilson's result as stated in \cite{W} is a classification in the case that $\bislf(\psi)$ is rank $1$ and {\bf{maximally commutative}} in the sense that $\bislf(\psi)$ 
is not contained in any larger commutative subalgebra of the algebra of differential operators..
However, in the rank $1$ case the common eigenspaces of $\bislf(\psi)$ must be one dimensional, so $\psi$ is (up to normalization) equal to a wave function of the KP hierarchy \cite{vM}.
From this, one may show that $\bislf(\psi)$ is necessarily maximally commutative, so the maximality assumption is not explicitly required.
For higher rank, it is not known whether maximality follows immediately, so this assumption will required below when necessary.

Bispectral functions $\psi$ for which the bispectral algebra $\bislf(\psi)$ contains an operator of prime order $p$ were classified by Duistermaat--Gr\"unbaum 
\cite{DG} in the case $p=2$ and by Horozov \cite{H} in the case $p>2$. These results and the methods of \cite{BHYcmp,HM,KR} prompted the 
following conjecture which was has been widely circulated since the mid 90s.

\begin{cnj}
\label{rank2}
The maximally commutative rank $2$ bispectral meromorphic algebras are precisely the algebras $\bislf(\psi)$ corresponding to all 
bispectral Darboux transformations $\psi(x,y)$ of the Bessel functions $\psi_{\Be(\nu)}(x,y)$ 
for $\nu \in \bbc \backslash \bbz$ and the Airy function $\psi_{\Ai}(x,y)$ see \S \ref{2.1}.
\end{cnj}
\begin{ex} 
\label{bisp-D-ex}
Continuing Examples \ref{Fourier-ex} and \ref{Filtr-ex}, we describe more explicitly the functions in Theorem \ref{rank1} and Conjecture \ref{rank2}.

(1) The bispectral Darboux transformations of the exponential function $\psi_{\exp}$ are precisely the functions in 
the (infinite dimensional) Wilson adelic Grassmannian $\Gr^{\ad}$, \cite{W}. They are the functions of the form
$$
\wt{\psi}(x,y) = \frac{1}{q(y)} \mathfrak{v} \cdot \psi_{\exp}(x,y)
$$
where $\mathfrak{v} \in \mathfrak D( \bbc(x))$ is a monic differential operator with rational coefficients satisfying
$$
\wt{\mathfrak{v}} \mathfrak{v} = f( \partial_x)
$$
for a monic polynomial $f(t)$ and some $\wt{\mathfrak{v}} \in \mathfrak D(\bbc(x))$.
The differential operators $\mathfrak{v}$ that appear in this way are classified in terms of their kernels consisting of quasi-exponential functions (i.e., solutions of homogeneous linear ordinary differential equations with constant coefficients). 
This classification is recalled in Section \ref{classif} below.  The polynomial $q(y)$ is uniquely determined from $\mathfrak{v}$ 
as the polynomial whose roots are the support of the quasi-exponential functions in the kernel of $\mathfrak{v}$. 

(2) The bispectral Darboux transformations of the Bessel functions $\psi_{\Be(\nu)}$ (for $\nu \in \bbc \backslash \bbz$) 
are the functions of the form
$$
\wt{\psi}(x,y) = \frac{1}{q(y)} \mathfrak{v} \cdot \psi_{\Be(\nu)}(x,y)
$$
where $\mathfrak{v} \in \mathfrak D( \bbc(x))$ is a monic differential operator with rational coefficients such that
$$
\wt{\mathfrak{v}} \mathfrak{v} = f( \mathfrak{d}_{\Be(\nu),x})
$$
for a monic polynomial $f(t)$ and some $\wt{\mathfrak{v}} \in \mathfrak D(\bbc(x))$.
All differential operators $\mathfrak{v}$ satisfying these properties, and as a result, the rank 2 bispectral functions
in the Bessel class, form an infinite dimensional manifold. 
They are classified by an explicit description of the possible forms of their kernels in terms of Bessel functions and their derivatives, see Section \ref{classif} for details.

(3) The bispectral Darboux transformations of the Airy function $\psi_{\Ai}$ 
are the functions of the form
$$
\wt{\psi}(x,y) = \frac{1}{q(y)} \mathfrak{v} \cdot \psi_{\Ai}(x,y)
$$
where $\mathfrak{v} \in \mathfrak D( \bbc(x))$ is a monic differential operator with rational coefficients satisfying
$$
\wt{\mathfrak{v}} \mathfrak{v} = f( \mathfrak{d}_{\Ai,x})
$$
for a monic polynomial $f(t)$ and some $\wt{\mathfrak{v}} \in \mathfrak D(\bbc(x))$. The differential operators $\mathfrak{v}$ satisfying these properties
and thus, the rank 2 bispectral functions in the Airy class, also form an infinite dimensional manifold. 
They are classified by an explicit description of the possible of forms their kernels in terms of derivatives of the Airy functions, see Section \ref{classif}.
\end{ex}

\subsection{Self-adjoint bispectral Darboux transformations and self-adjoint bispectral functions}
\begin{defn}
Let $\mathfrak d = \sum_{j=0}^n a_j(x)\partial_x^j$.  We define the {\bf{formal adjoint}} of $\mathfrak d$ to be
$$
\mathfrak d^* = \sum_{j=0}^n (-1)^j\partial_x^ja_j(x) = \sum_{j=0}^n\sum_{k=0}^j(-1)^j\binom{j}{k}a_j^{(k)}(x)\partial_x^{j-k}.
$$
An operator $\mathfrak d$ is called {\bf{formally symmetric}} if $\mathfrak d^* = \mathfrak d$.
\end{defn}
\begin{defn}
Let $\Gamma\subseteq \bbc$ be a smooth path in $\bbc$.  Then we say $\mathfrak d$ is {\bf{adjointable}} with respect to $\Gamma$ 
if for every $f(x),g(x)\in C_c^\infty(\Gamma)$ the following integral identity holds
$$\int_\Gamma (\mathfrak d\cdot f(x))g(x) dx = \int_\Gamma f(x)(\mathfrak d^*\cdot g(x))dx.$$
If $\mathfrak d$ is formally symmetric and adjointable, then we call $\mathfrak d$ {\bf{symmetric}}.
\end{defn}
\begin{defn}
Let $\psi(x,y)$ be a bispectral meromorphic function.  An operator $\mathfrak d\in\bisl(\psi)$ is called {\bf{formally bisymmetric}} if both $\mathfrak d$ and $b_\psi(\mathfrak d)$ are
formally symmetric.  Fix smooth paths $\Gamma_1,\Gamma_2\subseteq\bbc$.  We call $\mathfrak d$ {\bf{bisymmetric}} with respect to $(\Gamma_1,\Gamma_2)$ if $\mathfrak d$ 
is symmetric with respect to $\Gamma_1$ and $b_\psi(\mathfrak d)$ is symmetric with respect to $\Gamma_2$.
\end{defn}

\begin{defn}
\label{selfadj}
Let $\psi(x,y)$ be a bispectral meromorphic function.  We call a bispectral Darboux transformation $\wt \psi(x,y)$ with the notation of Definition 
\ref{bisp-Darb} a {\bf{self-adjoint bispectral Darboux transformation}} of $\psi(x,y)$ if
$$
\wt{p}(x) = p(x), \wt{q}(y) = q(y) \ \ \text{and} \ \
\wt{\mathfrak{u}} = \mathfrak{u}^*, b_\psi( \wt{\mathfrak{u}}) = b_{\psi}(\mathfrak{u})^*.
$$
\end{defn}
\begin{remk} In all of the situations that we consider in this paper the essential conditions are 
\begin{equation}
\label{essential-cond}
\wt{\mathfrak{u}} = \mathfrak{u}^*, b_\psi( \wt{\mathfrak{u}}) = b_{\psi}(\mathfrak{u})^*
\end{equation}
and the conditions $\wt{p}(x) = p(x), \wt{q}(y) = q(y)$ follow from them after a normalization. More precisely, if $\psi$ 
is any of the exponential $\psi_{\exp}$, Bessel $\psi_{\Be(\nu)}$ or the Airy $\psi_{\Ai}$ bispectral functions, then 
the left and right bispectral algebras $\bislf(\psi)$, $\bisrf(\psi)$ consist of differential operators with constant leading terms. 
The normalization \eqref{normalization} implies that $p(x)$ and $q(y)$ equal the leading terms of the 
differential operators $\mathfrak{u}$ and $b_\psi(\mathfrak{u})$. Therefore, by \eqref{factor-2}, 
$\wt p(x)$ and $\wt q(y)$ are scalar multiples of the leading terms of the differential operators $\wt{\mathfrak{u}}$ and $b_\psi( \wt{\mathfrak{u}})$

If \eqref{essential-cond} is satisfied then one can rescale $\wt p(x)$ and $\wt q(y)$ so that  $\wt{p}(x) = p(x), \wt{q}(y) = q(y)$. This rescaling does not change 
the function $\wt{\psi}(x,y)$. 
\qed
\end{remk} 
Theorem \ref{bisp-Darb-thm} implies that every self-adjoint bispectral Darboux transformation $\wt \psi(x,y)$ of a bispectral function $\psi(x,y)$ with the above data
satisfies
\begin{align}
\frac{1}{p(x)} \mathfrak{u} \mathfrak{u}^*  \frac{1}{p(x)} \cdot \wt{\psi}(x,y) &= q(y)^2 \wt{\psi}(x,y), 
\label{spec-eq1}
\\
\frac{1}{q(y)} b_{\psi}(\mathfrak{u}) b_{\psi}(\mathfrak{u})^*  \frac{1}{q(y)} \cdot \wt{\psi}(x,y) & = p(x)^2 \wt{\psi}(x,y).
\label{spec-eq2}
\end{align} 
Moreover, (\ref{factor}) and (\ref{factor-2}) imply that the original wave function $\psi(x,z)$ satisfies
\begin{align}
\wt{\mathfrak{u}}^*  \frac{1}{p(x)^2} \mathfrak{u} \cdot \psi(x,y) &= q(y)^2 \psi(x,y), 
\label{spec-eq3}
\\
b_{\psi} (\mathfrak{u})^*  \frac{1}{q(y)^2} b_{\psi}(\mathfrak{u}) \cdot \psi(x,y) &= p(x)^2 \psi(x,y).
\label{spec-eq4}
\end{align}
\begin{defn}\label{self-adj-fn}
We call a bispectral meromorphic function $\psi(x,y)$ {\bf{self-adjoint}} if it is an eigenfunction of nonconstant, formally symmetric differential operators
in $x$ and $y$.
\end{defn}
\begin{remk}
\label{self-adjo} A rank 1 bispectral function $\psi(x,y)$ is self-adjoint if and only if $\bislf(\psi)$ and $\bisrf(\psi)$ are preserved by taking formal adjoints $*$.
This is because in the rank $1$ case $\bislf(\psi)$ and $\bisrf(\psi)$ are maximally commutative and therefore equal to the centralizer of any of its nonconstant elements.
The centralizer of the centralizer of a formally ymmetric differential operator is closed under taking formal adjoints, so the self-adjointness of $\psi$ follows immediately.
\end{remk}

Eqs. (\ref{spec-eq1})--(\ref{spec-eq4}) imply the following:
\begin{prop}
\label{prop-dir1}
Suppose that $\psi(x,y)$ is a bispectral meromorphic function and that $\wt\psi(x,y)$ is a self-adjoint bispectral Darboux transformation of $\psi(x,y)$.  
Then both $\psi(x,y)$ and $\wt\psi(x,y)$ are self-adjoint bispectral meromorphic functions.
\end{prop}

\begin{lem}
Let $\psi(x,y)$ and $\wt\psi(x,y)$ be bispectral meromorphic functions and that $\wt\psi(x,y)$ is a bispectral Darboux transformation of $\psi(x,y)$.
If $\bislf(\psi)$ (resp. $\bisrf(\psi)$) is maximally commutative then so too is $\bislf(\wt\psi)$ (resp. $\bisrf(\wt\psi)$).
\end{lem}
\begin{proof}
To be explicit, suppose
$$\wt\psi(x,y) = \frac{1}{p(x)q(y)}\mathfrak u\cdot\psi(x,y)\ \ \text{and}\ \ \psi(x,y) = \wt{\mathfrak u}\frac{1}{\wt p(x)\wt q(y)}\cdot\wt\psi(x,y)$$
for some operators $\mathfrak u,\wt{\mathfrak u}\in\bisl(\psi)$ and polynomials $p(x),\wt p(x),q(y),\wt q(y)$.
Set $\mathfrak d = \wt{\mathfrak u}\frac{1}{\wt p(x)p(x)}\mathfrak u$ and $\wt{\mathfrak d} = \frac{1}{p(x)}\mathfrak u\wt{\mathfrak u}\frac{1}{\wt p(x)}$.
Centralizers of non-constant differential operators are maximally commutative, so to prove that $\bislf(\wt\psi)$ is maximal, it suffices to show that $\bislf(\wt\psi)$ is the centralizer of $\wt{\mathfrak d}$.
Suppose that $\mathfrak b$ is a differential operator which commutes with $\wt{\mathfrak d}$.
Then the Darboux conjugate $\wt{\mathfrak u}\frac{1}{\wt p(x)}\mathfrak b\wt{\mathfrak d}\wt p(x)\wt{\mathfrak u}^{-1}$ is a differential operator which commutes with $\mathfrak d$.
By the maximality assumption on $\bislf(\psi)$, we know that $\bislf(\psi)$ is equal to the centralizer of $\mathfrak d$ and thus $\wt{\mathfrak u}\frac{1}{\wt p(x)}\mathfrak b\wt{\mathfrak d}\wt p(x)\wt{\mathfrak u}^{-1}$ is in $\bislf(\psi)$.
Consequently the Darboux conjugate $\mathfrak b\wt{\mathfrak d}$ is in $\bislf(\wt\psi)$.
Since $\wt{\mathfrak d}\in\bislf(\wt\psi)$, it follows that $\mathfrak b\in\bislf(\wt\psi)$.
Note that $\mathfrak b$ was an arbitrary element in the centralizer of $\wt{\mathfrak d}$, so this completes the proof.
\end{proof}

The next theorem shows that when we restrict our attention to a self-adjoint bispectral meromorphic functions $\psi(x,y)$, 
the self-adjoint bispectral meromorphic functions arising from bispectral Darboux transformations of $\psi(x,y)$ are precisely 
those which arise from self-adjoint bispectral Darboux transformations. This statement, which is converse to that in Proposition \ref{prop-dir1}, 
is harder and some easily verifiable conditions on the initial function $\psi(x,y)$ are imposed. 
\begin{thm}
\label{self-adj}
Let $\psi(x,y)$ and $\wt\psi(x,y)$ be a self-adjoint bispectral meromorphic functions and suppose that $\wt\psi(x,y)$ is a bispectral Darboux transformation of $\psi(x,y)$ with 
$$\wt\psi(x,y) = \frac{1}{p(x)q(y)}\mathfrak u\cdot\psi(x,y)\ \ \text{and}\ \ \psi(x,y) = \wt{\mathfrak u}\frac{1}{\wt p(x)\wt q(y)}\cdot\wt\psi(x,y)$$
for some operators $\mathfrak u,\wt{\mathfrak u}\in\bisl(\psi)$ and polynomials $p(x),\wt p(x),q(y),\wt q(y)$.
Assume that $\bislf(\psi)$ is maximally commutative, $p(x)=\wt p(x)$, $q(y)=\wt q(y)$, and that $\mathfrak u^*,\wt{\mathfrak u}^*\in\bisl(\psi)$.
Then $\wt\psi(x,y)$ is a self-adjoint bispectral Darboux transformation of $\psi(x,y)$, and in fact the above bispectral Darboux transformation is self-adjoint.
\end{thm}
\begin{remk}
(i) In the case that $\psi(x,y) = \psi_{\Ai}(x,y)$, the assumptions follow trivially since $b_\psi^{-1}(\bisrf(\psi)) = \bbc[x]$, $b_\psi(\bislf(\psi)) = \bbc[y]$, and $b_\psi(\mathfrak w^*) = b_\psi(\mathfrak w)^*$ for all $\mathfrak w\in\bisl(\psi)$.
For $\psi(x,y) = \psi_{\Be(\nu)}(x,y)$ (with $\nu\notin\bbz$) the conditions follow from the explicit characterization of the bispectral Darboux transformations by Bakalov-Horozov-Yakimov \cite[Theorem 2.7]{BHYcmp}.

(ii) The condition $p(x)=\wt p(x)$ and $q(y) = \wt q(y)$ in the theorem is equivalent to the seemingly weaker condition that
$$
p(x)/\wt p(x) = \wt{p}_1(x)/p_1(x) \quad \mbox{and} \quad 
q(y)/\wt q(y) = \wt{q}_1(y)/q_1(y)
$$
for some polynomials $p_1(x),\wt p_1(x)\in \bisl(\psi)$ and $q_1(y),\wt q_1(y)\in\bisr(\psi)$. This is the condition that is directly
verified by \cite[Theorem 2.7]{BHYcmp}. If the second condition holds, then by setting $\mathfrak d_1 = b_\psi^{-1}(q_1(y))$ and $\wt{\mathfrak d_1} = b_\psi^{-1}(\wt q_1(y))$,
and by changing
\begin{align*}
&p(x)\mapsto p_1(x)p(x),\ \ \mathfrak u\mapsto p_1(x)\mathfrak u\mathfrak d_1,\ \ q(y)\mapsto q_1(y)q(y), 
\\
&\wt p(x)\mapsto \wt p_1(x)\wt p(x),\ \ \wt{\mathfrak u}\mapsto \wt p_1(x)\wt{\mathfrak u}\wt{\mathfrak d_1},\ \ \wt q(y)\mapsto \wt q_1(y)\wt q(y),
\end{align*}
we get a new bispectral Darboux transformation from $\psi(x,y)$ to $\wt\psi(x,y)$ which satisfies the condition (a) of the theorem.
\end{remk}
\begin{proof}
Define $r(y) = q(y)\wt q(y)$ and
$$\varphi(x,y) := \frac{1}{\wt p(x)\wt q(y)}\wt{\mathfrak u}^*\cdot\psi(x,y), \ \ \text{and}\ \ \wt{\mathfrak d} = \left(\frac{1}{p(x)}\mathfrak u\wt{\mathfrak u}\frac{1}{\wt p(x)}\right)$$
so that
$$\wt{\mathfrak d}\cdot\wt\psi(x,y) = \wt\psi(x,y) r(y)\ \ \text{and}\ \ \wt{\mathfrak d}^*\cdot\varphi(x,y) = \varphi(x,y)r(y).$$
Note that $\varphi(x,y)$ is a bispectral Darboux transform of $\psi(x,y)$ and
$$\varphi(x,y) = \frac{1}{\wt q(y)^2}\mathfrak a\cdot\wt\psi(x,y)\ \text{for}\ \mathfrak a = \frac{1}{\wt p(x)}\wt{\mathfrak u}^*\wt{\mathfrak u}\frac{1}{\wt p(x)}.$$

By the previous lemma, the algebras $\bislf(\wt\psi),\bisrf(\wt\psi),\bislf(\varphi)$, and $\bisrf(\varphi)$ are all maximally commutative.
Since $\wt\psi(x,y)$ is self-adjoint and $\bislf(\wt\psi)$ is maximally commutative, the inclusion $\wt{\mathfrak d}\in\bislf(\wt\psi)$ implies $\wt{\mathfrak d}^*\in\bislf(\wt\psi)$.
Therefore both $\bislf(\wt\psi)$ and $\bislf(\varphi)$ contain the common nonconstant element $\wt{\mathfrak d}^*$.
Since both algebras are maximally commutative, they are each equal to the centralizer of $\wt{\mathfrak d}^*$ and hence $\bislf(\wt\psi) = \bislf(\varphi)$.
Similarly $\bisrf(\wt\psi) = \bisrf(\varphi)$.  In particular $\varphi$ is self-adjoint.

For any $\mathfrak d\in\bislf(\psi)$ we know that 
$$\left(\frac{1}{p(x)}\mathfrak u\mathfrak d\wt{\mathfrak u}\frac{1}{\wt p(x)}\right)\cdot\wt\psi(x,y) = \wt\psi(x,y) b_\psi(\mathfrak d)r(y),\ \ \left(\frac{1}{p(x)}\mathfrak u\mathfrak d\wt{\mathfrak u}\frac{1}{\wt p(x)}\right)^*\cdot\varphi(x,y) = \varphi(x,y)b_\psi(\mathfrak d)r(y).$$
From this, we deduce that for every $\mathfrak b\in\bislf(\wt\psi)$ we must have $b_{\wt\psi}(\mathfrak b) = b_{\varphi}(\mathfrak b^*)$.
Choose $\mathfrak b\in\bislf(\wt\psi)$ to be formally symmetric.
Then for some polynomial $s(y)$ we calculate $\mathfrak b\cdot\wt\psi(x,y) = \wt\psi(x,y)s(y)$ and $\mathfrak b\cdot\varphi(x,y) = \varphi(x,y)s(y)$ and therefore
$$\mathfrak b\mathfrak a\cdot\wt\psi(x,y) = \mathfrak b\cdot \varphi(x,y)\wt q(y)^2 = \varphi(x,y)\wt q(y)^2s(y) = \mathfrak a\cdot\wt\psi(x,y)s(y) = \mathfrak a\mathfrak b\cdot\wt\psi(x,y).$$
Hence $\mathfrak a$ and $\mathfrak b$ commute.
It follows that $\mathfrak a\in\bislf(\wt\psi)$ and therefore that $\varphi(x,y) = \wt\psi(x,y)h(y)$ for some rational function $h(y)$.
This implies that $\bisrf(\varphi) = h(y)\bisrf(\wt\psi) h(y)^{-1}$ and since $\bisrf(\varphi)=\bisrf(\psi)$ we find $\bisrf(\wt\psi) =h(y)\bisrf(\wt\psi)h(y)^{-1}$.
However, each element of $\bisrf(\wt\psi)$ has constant leading coefficient
If $\mathfrak c$ is a nonconstant differential operator with a constant leading coefficient then $\mathfrak c - h(y)\mathfrak c h(y)^{-1}$ has leading coefficient equal to some constant multiple of $h'(y)/h(y)$.
Therefore $h'(y)=0$ and $h(y)$ must be constant $c$.

From the result of the previous paragraph,
$$\frac{1}{\wt p(x)\wt q(y)}\wt{\mathfrak u}^*\cdot\psi(x,y) = \frac{c}{p(x)q(y)}\mathfrak u\cdot\psi(x,y).$$
Applying the above argument on the opposite side, we obtain the similar expression
$$\frac{1}{\wt p(x)\wt q(y)}b_\psi(\wt{\mathfrak u})^*\cdot\psi(x,y) = \frac{c'}{p(x)q(y)}b_\psi(\mathfrak u)\cdot\psi(x,y).$$
for some constant $c'$.
Finally we apply the assumption that $p(x) = \wt p(x)$ and $q(y) = \wt q(y)$ to obtain $c=c'=1$ and $\wt{\mathfrak u}^* = \mathfrak u$ and $b_\psi(\wt{\mathfrak u})^* = b_\psi(\mathfrak u)$.
\end{proof}

\begin{cor} The self-adjoint bispectral Darboux transformations of the exponential, Airy and Bessel functions are precisely those functions 
$\wt\psi(x,y)$ that may be obtained by a bispectral Darboux transformations of $\psi_{\exp}, \psi_{\Ai}$ or $\psi_{\Be(\nu)}$ and are self-adjoint as
bispectral functions in the sense of Definition \ref{self-adj-fn}.
\end{cor} 

\subsection{Interplay between adjoints and the generalized Fourier map}
\label{subFour}
In this subsection we discuss the interaction between the formal adjoint $*$ and the generalized Fourier map $b_\psi$.
Specifically, we consider in this section the value of $b_\psi(\mathfrak d^*)$, assuming that both $\mathfrak d$ and $\mathfrak d^*$ are in $\bisl(\psi)$.
It turns out that in the cases we consider, the value of $b_\psi(\mathfrak d^*)$ is related to $b_\psi(\mathfrak d)^*$ by a certain automorphism of the Weyl algebra.
As a consequence of this, the elements of $\bisls(\psi)$ are invariant under this automorphism.
This invariance property of $\bisls(\psi)$ will play an important role of our proof of Theorem A from the introduction.

For many self-adjoint bispectral meromorphic functions $\psi$ the associated algebras $\bisl(\psi)$ and $\bisr(\psi)$ are closed under the formal adjoint operation $*$.
This is true in particular of the elementary bispectral meromorphic functions: by Example \ref{Fourier-ex}, 
the Fourier algebras for $\psi_{\exp},\psi_{\Ai}$ and $\psi_{\Be(\nu)}$ are all closed under the operation of taking formal adjoints.
The same property turns out to hold for the Fourier algebras of all functions obtained by self-adjoint bispectral Darboux transformations from them. 
\begin{prop}
Let $\wt\psi$ be a self-adjoint bispectral Darboux transformation of $\psi$.
Suppose that $\bisl(\psi)$ is closed under the formal adjoint operation $*$.
Then $\bisl(\wt\psi)$ is also closed under the $*$-operation.
The same statement holds with $\bisl$ replaced with $\bisr$.
\end{prop}
\begin{proof}
Let $\mathfrak u,\wt{\mathfrak u}=\mathfrak u^*,p(x)=\wt p(x)$ and $q(x)=\wt q(x)$ be as in the definition of a self-adjoint bispectral Darboux transformation with $\wt\psi(x,y) = \frac{1}{p(x)q(y)}\mathfrak u\cdot \psi(x,y)$.
Also define $\wt{\mathfrak a} = \frac{1}{p(x)}\mathfrak u\mathfrak u^*\frac{1}{p(x)}\in\bisls(\wt\psi)$ with $\wt{\mathfrak a}\cdot\wt\psi(x,y) = q(y)^2\wt\psi(x,y)$.

Next note that for any differential operator $\mathfrak q$ with meromorphic coefficients,
\begin{equation}
\label{impl}
\mathfrak q\wt{\mathfrak a}\in\bisl(\wt \psi) \; \; \Rightarrow \; \; 
\mathfrak q\in \bisl(\wt \psi).
\end{equation}
This follows from the identity
$$\mathfrak q \cdot \wt\psi(x,y) q(y)^2 = \mathfrak q\wt{\mathfrak a}\cdot\wt \psi(x,y) = b_{\wt\psi}(\mathfrak q\wt{\mathfrak a})\wt\psi(x,y),$$
which implies that $\mathfrak q\cdot\wt\psi(x,y) = q(y)^{-2}b_{\wt\psi}(\mathfrak q\wt{\mathfrak a})\cdot\wt\psi(x,y)$.

For any $\wt{\mathfrak d}\in\bisl(\wt\psi)$, one easily checks that $\mathfrak d := \mathfrak u^* \frac{1}{p(x)}\wt{\mathfrak d}\frac{1}{p(x)}\mathfrak u\in \bisl(\psi)$.
In fact $b_\psi(\mathfrak d) =  q(y)b_{\wt\psi}(\wt{\mathfrak d})q(y)$.
Since $\bisl(\psi)$ is closed under $*$, it follows that $\mathfrak d^*  = \mathfrak u^*\frac{1}{p(x)}\wt{\mathfrak d}^*\frac{1}{p(x)}\mathfrak u\in\bisl(\psi)$.
Consequently 
$$\wt{\mathfrak a}\wt{\mathfrak d}^*\cdot\wt\psi(x,y) = \frac{1}{q(y)}b_\psi(\mathfrak d^*) \cdot q(y)\wt\psi(x,y),$$
so that in particular $\wt{\mathfrak a}\wt{\mathfrak d}^*\in\bisl(\wt\psi)$ for all $\wt{\mathfrak d}\in\bisl(\wt\psi)$.
In the remainder of the proof we show that this implies that $\wt{\mathfrak d}^*\in\bisl(\wt\psi)$. 

Let $\ell$ be the co-order of $\wt{\mathfrak d}$.
Then $\ad_{\wt{\mathfrak a}}^{\ell +1}(\wt{\mathfrak d})$ is sent under $b_{\wt\psi}$ to $(-1)^{\ell +1} \ad_{q(y)^2}^{\ell+1}(b_\psi(\wt{\mathfrak d})) = 0$.
Since the generalized Fourier map is an anti-isomorphism, this implies that $\ad_{\wt{\mathfrak a}}^{\ell +1}(\wt{\mathfrak d})$.
Using the fact that $\wt{\mathfrak a}$ is formally symmetric, we obtain that 
$$
0 = \ad_{\wt{\mathfrak a}}^{\ell +1}(\wt{\mathfrak d}^*) = \sum_{k=0}^{\ell +1} 
(-1)^k 
\begin{pmatrix}
\ell +1 \\
k
\end{pmatrix}
\wt{\mathfrak a}^{\ell +1 -k}
\wt{\mathfrak d}^*
\wt{\mathfrak a}^k.
$$
This identity and the fact that $\wt{\mathfrak a}\wt{\mathfrak d}^*, \wt{\mathfrak a} \in\bisl(\wt\psi)$ give that 
$\wt{\mathfrak a}^{\ell +1} \wt{\mathfrak d}^* \in\bisl(\wt\psi)$. By repeatedly applying \eqref{impl}, 
we obtain that $\wt{\mathfrak d}^* \in\bisl(\psi)$.
This proves that $\bisl(\wt\psi)$ is closed under $*$. In a similar way one proves that $\bisr(\wt\psi)$ is closed under $*$.
\end{proof}

Assuming that both $\bisl(\psi)$ and $\bisr(\psi)$ are closed under the $*$-operation, 
we can define automorphisms $\iota_{\psi, x}$ and $\iota_{\psi,y}$ of $\bisl(\psi)$ and $\bisr(\psi)$ by setting
\begin{equation}
\label{iota}
\iota_{\psi, x}(\mathfrak d) = b_\psi^{-1}(b_\psi(\mathfrak d^*)^*)\ \ \text{and} \ \ \iota_{\psi,y}(\mathfrak b) = b_\psi(b_\psi^{-1}(\mathfrak d^*)^*).
\end{equation}
These automorphisms encode the interaction between the generalized Fourier map $b_\psi$ and the formal adjoint.
Clearly, $\iota_{\psi, x}$ and $\iota_{\psi,y}$ restrict to the identity on $\bisls(\psi)$ and $\bisrs(\psi)$.

In the case $\psi = \psi_{\exp}$ or $\psi = \psi_{\Be(\nu)}$ the automorphisms $\iota_{\psi, x}$ and $\iota_{\psi,y}$ are exactly obtained by restricting the automorphism of the algebra of differential operators with rational coefficients induced by the affine transformation $x\mapsto -x$ and $y \mapsto - y$.
In the case $\psi = \psi_{\Ai}$ the automorphisms are just the identity.
Moreover, as a consequence of the previous proposition, for any self-adjoint bispectral Darboux transformation $\wt\psi$ of $\psi=\psi_{\exp},\psi_{\Ai}$ or $\psi_{\Be(\nu)}$ the involutive automorphisms $\iota_{\wt\psi,x}$ and $\iota_{\wt\psi,y}$ are well-defined.
The next proposition shows that $\iota_{\wt{\psi}, x}$ and $\iota_{\wt\psi,y}$ behave identically to $\iota_{\psi,x}$ and $\iota_{\psi,y}$.
\begin{prop}\label{iota prop}
Let $\psi$ be a self-adjoint bispectral meromorphic function and $\wt\psi$ be a self-adjoint bispectral Darboux transformation of $\psi$.
Suppose that $\bisl(\psi)$ and $\bisr(\psi)$ are closed under the formal adjoint operation $*$.
Assume moreover that there are involutive automorphisms $\sigma_x$ and $\sigma_y$ of the algebras of differential operators with meromorphic coefficients in $x$ and $y$, respectively, restricting to $\iota_{\psi, x}$ and $\iota_{\psi, y}$.
Then $\sigma_x$ and $\sigma_y$ also restrict to $\iota_{\wt\psi, x}$ and  $\iota_{\wt\psi, y}$, respectively.
In particular the elements of $\bisls(\wt\psi)$ are fixed by $\sigma_x$ and similarly for $\bisrs(\wt\psi)$.
\end{prop}
\begin{proof}[Proof of Proposition \ref{iota prop}]
Let $\mathfrak u,\wt{\mathfrak u}=\mathfrak u^*,p(x)=\wt p(x)$ and $q(x)=\wt q(x)$ be as in the definition of a self-adjoint bispectral Darboux transformation with $\wt\psi(x,y) = \frac{1}{p(x)q(y)}\mathfrak u\cdot \psi(x,y)$.
For sake of brevity set $\mathfrak n = \frac{1}{p(x)}\mathfrak u$.
First note that if $p(x)$ is constant, then $\wt\psi(x,y)$ is a constant multiple of $\psi(x,y)$ and the proposition is trivial.
Therefore we assume otherwise.
Suppose that $\wt{\mathfrak d}\in\bisl(\wt\psi)$.
Then $\mathfrak n^*\wt{\mathfrak d}\mathfrak n\in\bisl(\psi)$ with $b_\psi(\mathfrak n^*\wt{\mathfrak d}\mathfrak n) = q(y)b_{\wt\psi}(\wt{\mathfrak d})q(y)$.
In particular, if $\wt{\mathfrak d}\in \bisls(\wt\psi)$, then $\mathfrak n^*\wt{\mathfrak d}\mathfrak n\in\bisls(\psi)$.
Since $\iota_{\psi,x}$ acts as the identity on $\bisls(\psi)$, for all $\mathfrak d\in\bisls(\psi)$ we have
$$\mathfrak n^*\wt{\mathfrak d}\mathfrak n = \iota_{\psi,x}(\mathfrak n^*\wt{\mathfrak d}\mathfrak n) = \sigma_x(\mathfrak n^*)\sigma_x(\wt{\mathfrak d})\sigma_x(\mathfrak n).$$
Applying this to the case when $\wt{\mathfrak d} = 1$, we see that as pseudo-differential operators
$$\sigma_x(\mathfrak n^*)^{-1}\mathfrak n^* = \sigma_x(\mathfrak n)\mathfrak n^{-1}.$$
Since $\psi(x,y)$ is self-adjoint and $\wt\psi(x,y)$ is a self-adjoint bispectral Darboux transformation of $\psi(x,y)$, we know that $p(x)^2\in\bisls(\psi)\cap\bisls(\wt\psi)$.
Letting $\wt{\mathfrak d} = p(x)^2\in\bisls(\wt \psi)$ we calculate
$$\sigma_x(\mathfrak n)\mathfrak n^{-1}p(x)^2 = p(x)^2\sigma_x(\mathfrak n)\mathfrak n^{-1}.$$
Since $p(x)$ is nonconstant, this implies that $\sigma_x(\mathfrak n)\mathfrak n^{-1} = g(x)$ for some rational function $g(x)$.
Thus for all $\wt{\mathfrak d}\in\bisls(\wt\psi)$ we see
$$g(x) \wt{\mathfrak d} g(x)^{-1} =  \sigma_x(\wt{\mathfrak d}).$$
Since $\sigma_x$ was assumed to be an involution, this implies that $g(x)^2$ is in the center if $\bisls(\psi)$, and must therefore be constant.
Hence $g(x)$ is constant, ie. $g(x)=c$ for some constant $c\in\bbc$.
In particular $\sigma_x(\mathfrak n) = c\mathfrak n$ and since $\sigma_x$ is an involution $c^2=1$ so that $c = \pm 1$.
This also implies that $\sigma(\mathfrak n^*) = c\mathfrak n$.

Suppose that $\wt{\mathfrak d}\in\bisl(\wt\psi)$.
From the proof of the previous proposition, we know that
$$b_\psi(\mathfrak n^*\wt{\mathfrak d}\mathfrak n) = q(y)b_{\wt\psi}(\wt{\mathfrak d})q(y).$$
From this, we see
\begin{align*}
b_{\wt\psi}(\iota_{\wt\psi,x}(\wt{\mathfrak d})) = b_{\wt\psi}(\wt{\mathfrak d}^*)^*
 & = \frac{1}{q(y)}b_\psi(\mathfrak n^*\wt{\mathfrak d}^*\mathfrak n)^*\frac{1}{q(y)}\\
 & = \frac{1}{q(y)}b_\psi\circ\iota_{\psi,x}(\mathfrak n^*\wt{\mathfrak d}\mathfrak n)\frac{1}{q(y)}\\
 & = \frac{1}{q(y)}b_\psi(\mathfrak n^*\sigma_x(\wt{\mathfrak d})\mathfrak n)\frac{1}{q(y)} = b_{\wt\psi}(\sigma_x(\wt{\mathfrak d})).
\end{align*}
It follows that $\iota_{\wt\psi,x}(\wt{\mathfrak d}) = \sigma_x(\wt{\mathfrak d})$ for all $\wt{\mathfrak d}\in\bisl(\psi)$.
A similar proof holds for $\bisr(\psi)$.
\end{proof}

The following corollary follows immediately.
\begin{cor}\label{automatic symmetry}
If $\psi(x,y)$ is a self-adjoint bispectral Darboux transformation of $\psi_{\exp}$ or $\psi_{\Be(\nu)}$ then the elements of $\bisls(\psi)$ and $\bisrs(\psi)$ are invariant under the affine transformations $x\mapsto-x$ and $y\mapsto-y$, respectively.
\end{cor}
\subsection{Formalization of commutativity of integral and differential operators}
\label{form-comm}
In the final part of this section we formalize commutativity of integral and differential operators. 

Let $\Gamma$ be a smooth curve in $\bbc$ and $\mathfrak D(\Gamma)$ be the space of differential operators on $\Gamma$ with smooth coefficients.
Let $C^\infty(\Gamma)$ denote the space of smooth functions on $\Gamma$ and let $C_c^\infty(\Gamma)$ be the subspace of $C^\infty(\Gamma)$ consisting 
of functions with compact support.
Consider an integral operator $T: f(x)\mapsto \int_\Gamma K(x,y)f(y)dy$ with kernel $K(x,y)$ smooth on $\Gamma\times\Gamma$ and 
satisfying the property that the functions $\int_\Gamma |K(x,y)|dx$ and $\int_\Gamma |K(x,y)|dy$ both lie in $L^\infty(\Gamma)$.
Then for any $f(x)\in C_c^\infty(\Gamma)$ the integral $\int_\Gamma K(x,y)f(y)dy$ converges uniformly to a smooth function $T(f(x))\in C^\infty(\Gamma)$.
Furthermore an operator $\mathfrak d\in\mathfrak D(\Gamma)$ restricts to an endomorphism of $C^\infty(\Gamma)$ and $C_c^\infty(\Gamma)$.
Thus for any $f(x)\in C_c^\infty(\Gamma)$, the expressions $T(\mathfrak d\cdot f)$ and $\mathfrak d\cdot (T(f))$ are both elements of $C^\infty(\Gamma)$.

We say that {\bf{the integral operator $T$ and the differential operators $\mathfrak d$ commute}} if 
$$
T(\mathfrak d\cdot f) = \mathfrak d\cdot(T(f)) \quad \mbox{for all} \quad f\in C_c^\infty(\Gamma).
$$

\section{Darboux transformations and growth of Fourier algebras}
In this section we carry out two key steps of our strategy for going from bispectral meromorphic functions to integral operators possessing a commuting differential operator. 

Firstly, we prove a theorem that gives a lower bound of the growth of the double filtration of the Fourier algebra $\bisl(\wt\psi)$ of a bispectral 
Darboux transformation $\wt\psi$ of $\psi$. Modulo technical details we show that 
$$
| \dim ( \bisl^{\ell,m}(\wt\psi) ) - \dim ( \bisl^{\ell,m}(\psi) ) | \leq \mathrm{const}
$$
for a constant that is independent on $\ell$ and $m$. We also prove that 
a similar inequality holds for $\bisl^{\ell,m}$ replaced with $\bisls^{\ell,m}$.

Secondly, we prove that if the dimension of the space $\bisl^{\ell,m}(\wt\psi)$ satisfies a natural lower bound, then it contains a bisymmetric operator with respect to a pair of contours.
\subsection{Control of the growth of Fourier algebras under Darboux transformations}
\begin{thm}\label{control thm}
Suppose that $\psi(x,y)$ is a bispectral meromorphic function and that $\wt\psi(x,y)$ is a self-adjoint bispectral Darboux transformation of $\psi(x,y)$ of order $(d_1,d_2)$.  
Then the following is true for all $\ell>d_2$ and $m>d_1$
\begin{align*}
\dim(\bisl^{\ell,m}(\wt\psi))&\geq \dim(\bisl^{\ell,m-2d_2}(\psi)) + \dim(\bisl^{\ell-2d_1,2d_2-1}(\psi)) + 1, \\
\dim(\bisls^{\ell,m}(\wt\psi))&\geq \dim(\bisls^{\ell,m-2d_2}(\psi)) + \dim(\bisls^{\ell-2d_1,2d_2-1}(\psi)) + 1.
\end{align*}
\end{thm}
\begin{proof}
By the definition of a self-adjoint bispectral Darboux transformation, there exists a differential operator $\mathfrak u\in\bisl(\psi)$ 
and polynomials $p(x)\in\bbc[x]$ and $q(y)\in\bbc[y]$ such that
$$\wt\psi = \frac{1}{p(x)q(y)}\mathfrak u\cdot\psi,$$
and such that for $\mathfrak w = b_\psi(\mathfrak u)$,
$$b_\psi(p(x)^2) = \mathfrak w^*\frac{1}{q(y)^2}\mathfrak w\ \ \text{and}\ \ b_\psi^{-1}(q(y)^2) = \mathfrak u^*\frac{1}{p(x)^2}\mathfrak u.$$

Let $d_1$ be the order of $\mathfrak u$ and $d_2$ be the order of $\mathfrak w$.  If $\mathfrak d\in\bisl^{\ell,m}(\psi)$, then we have
\begin{align*}
p(x)\mathfrak d p(x)\cdot\wt\psi(x,y)
  & = \frac{p(x)}{q(y)}\mathfrak d\mathfrak v\cdot \psi(x,y)\\
  & = \frac{1}{q(y)}\mathfrak wb_\psi(\mathfrak d)\cdot p(x)\psi(x,y)\\
  & = \frac{1}{q(y)}\mathfrak wb_\psi(\mathfrak d)\cdot p(x)^2 \frac{1}{p(x)}\psi(x,y)\\
  & = \frac{1}{q(y)}\mathfrak wb_\psi(\mathfrak d)\mathfrak w^*\frac{1}{q(y)^2}\mathfrak w\cdot \frac{1}{p(x)}\psi(x,y)\\
  & = \frac{1}{q(y)}\mathfrak wb_\psi(\mathfrak d)\mathfrak w^*\frac{1}{q(y)}\cdot \wt\psi(x,y).
\end{align*}
This shows that $p(x)\mathfrak dp(x)\in\bisl^{\ell,m+2d_2}(\wt\psi)$ with
$$b_{\wt\psi}(p(x)\mathfrak dp(x)) = \frac{1}{q(y)}\mathfrak wb_\psi(\mathfrak d)\mathfrak w^*\frac{1}{q(y)}.$$
In particular, we have an inclusion 
$$p(x)\bisl^{\ell,m}(\psi)p(x)\subseteq \bisl^{\ell,m+2d_2}(\wt\psi).$$ 
Similarly, we can show that if $\mathfrak b\in \bisr^{m,\ell}(\psi)$ then $q(y)\mathfrak bq(y)\in\bisr^{m,\ell}(\wt\psi)$ with
$$b_{\wt\psi}^{-1}(q(y)\mathfrak bq(y)) = \frac{1}{p(x)}\mathfrak ub_\psi^{-1}(\mathfrak b)\mathfrak u^*\frac{1}{p(x)}.$$
Thus we also have an inclusion 
$$p(y)\bisr^{m,\ell}(\psi)p(x)\subseteq \bisr^{m,\ell+2d_1}(\wt\psi).$$
Therefore we have two inclusions
\begin{align*}
&p(x)\bisl^{\ell+2d_1,m}(\psi)p(x)\subseteq \bisl^{\ell+2d_1,m+2d_2}(\wt\psi) \quad \mbox{and} \\
&b_{\wt\psi}^{-1}(q(y)\bisr^{m+2d_2,\ell}(\psi)q(y))\subseteq \bisl^{\ell+2d_1,m+2d_2}(\wt\psi).
\end{align*}
Furthermore,
$$p(x)\bisl^{\ell+2d_1,m}(\psi)p(x)\cap b_{\wt\psi}^{-1}(q(y)\bisr^{2d_2-1,\ell}(\psi)q(y)) = 0,$$
because the nonzero elements in $p(x)\bisl^{\ell+2d_1,m}(\psi)p(x)$ are mapped under $b_{\wt\psi}$ to elements of order at least $2d_2$.
Thus we can write
$$\bbc\oplus p(x)\bisl^{\ell+2d_1,m}(\psi)p(x)\oplus b_{\wt\psi}^{-1}(q(y)\bisr^{2d_2-1,\ell}(\psi)q(y)) \subseteq \bisl^{\ell+2d_1,m+2d_2}(\wt\psi),$$
which in particular gives us the dimension estimate
$$\dim(\bisl^{\ell+2d_1,m+2d_2}(\wt\psi))\geq \dim(\bisl^{\ell+2d_1,m}(\psi)) + \dim(\bisr^{2d_2-1,\ell}) + 1.$$
The above inclusion also sends formally symmetric operators to formally symmetric operators, and therefore
\begin{equation}
\label{subspace}
\bbc\oplus p(x)\bisls^{\ell+2d_1,m}(\psi)p(x)\oplus b_{\wt\psi}^{-1}(q(y)\bisrs^{2d_2-1,\ell}(\psi)q(y)) \subseteq \bisls^{\ell+2d_1,m+2d_2}(\wt\psi),
\end{equation}
so that
$$\dim(\bisls^{\ell+2d_1,m+2d_2}(\wt\psi))\geq \dim(\bisls^{\ell+2d_1,m}(\psi)) + \dim(\bisrs^{2d_2-1,\ell}) + 1.$$
This completes the proof of the theorem.
\end{proof}
\begin{cor}\label{dimension cor}
Let $\psi(x,y)$ is a bispectral meromorphic function and $\wt\psi(x,y)$ be a self-adjoint bispectral Darboux transformation of $\psi(x,y)$ of order $(d_1,d_2)$. 
Assume that for some $r>0$ there exist constants $a_{11},a_{10},a_{01},a_{00}$ such that for all $\ell,m\geq 0$ we have
\begin{equation}
\label{ineq1}
\dim(\bisl^{r\ell,rm}(\psi)) \geq a_{11}\ell m + a_{10}\ell + a_{01}m + a_{00}.
\end{equation}
Tthen the growth of the Fourier algebra $\bisl(\wt\psi)$ is controlled by the inequality
\begin{equation}
\label{ineq2}
\dim(\bisl^{r\ell,rm}(\wt\psi)) \geq a_{11}\ell m + (2a_{10}-a_{11})\ell + a_{01}m + a_{00}-\wt c,
\end{equation}
where
$$
\wt c = a_{11}\floor{\frac{2d_1}{r}}\floor{\frac{2d_2-1}{r}} + a_{10}\floor{\frac{2d_1}{r}} + a_{01}\left(\floor{\frac{2d_2}{r}}-\floor{\frac{2d_2-1}{r}}\right) - a_{00} - 1.
$$
The statement also holds if we replace $\bisl^{r\ell,rm}$ with $\bisls^{r\ell,rm}$ in \eqref{ineq1} and \eqref{ineq2}.
\end{cor}
\begin{remk}
The precise form of $\wt c$ in the above inequalities plays an especially important role in practical implementations of the search for differential operators commuting 
with integral operators, as we shall see below.  This is because it provides an explicit bound for orders and co-orders of the operators which we must search through to find an 
operator commuting with our integral operator, and restricting our search to a finite dimensional vector space.
\end{remk}
\begin{remk}
\label{easy form}
We note that for all the cases of bispectral meromorphic functions $\psi$ which we have computed, the values of $a_{11},a_{10},$ and $a_{01}$ are the same for suitably chosen $r$.  Simply put, for all rank $1$ or rank $2$ bispectral functions $\psi(x,y)$, along with all other cases we have considered, 
when $\wt\psi$ is a self-adjoint bispectral Darboux transformation of $\psi$ then 
$$|\dim(\bisl^{r\ell,rm}(\wt \psi))-\dim(\bisl^{r\ell,rm}(\psi))| \leq \wt c$$
for $\wt c$ defined as in Corollary \ref{dimension cor}.  A similar statement holds for $\bisl^{r\ell,rm}$ replaced with $\bisls^{r\ell,rm}$.
\end{remk}
\begin{proof}[Proof of Corollary \ref{dimension cor}]
Plugging in the lower bound into the dimension estimate from the previous theorem, we see
\begin{align*}
\dim(\bisl^{r\ell,rm}(\wt \psi))
  & \geq \dim(\bisl^{r\ell,rm-2d_2}(\psi) + \dim(\bisl^{r\ell-2d_1,2d_2-1})\\
  & \geq \dim(\bisl^{r\ell,r(m-\floor{2d_2/r})}(\psi)) + \dim(\bisl^{r(\ell-\floor{2d_1/r}),r\floor{(2d_2-1)/r}}(\psi))\\
  & \geq a_{11}\ell\left(m-\floor{\frac{2d_2}{r}}\right) + a_{10}\ell + a_{01}\left(m-\floor{\frac{2d_2}{r}}\right) + a_{00}\\
  & +    a_{11}\left(\ell-\floor{\frac{2d_1}{r}}\right)\floor{\frac{2d_2-1}{r}} + a_{10}\left(\ell-\floor{\frac{2d_1}{r}}\right) + a_{01}\floor{\frac{2d_2-1}{r}} + a_{00} + 1\\
  & \geq a_{11}\ell m + \left(2a_{10}-a_{11}\right)\ell + a_{01}m + a_{00} - \wt c.
\end{align*}
\end{proof}
In the cases considered in this paper, we will find $\bisls^{2\ell,2m}(\psi) = \ell m + \ell + m + a_{00}$, so that $\bisls^{2,2}(\psi) = 3 + a_{00}$.  In this case $\wt c = d_1d_2 - a_{00}$, so that if $\wt\psi$ is a self-adjoint bispectral Darboux transformation of $\psi$ of order $(1,1)$, then
$$\dim(\bisls^{2,2}(\wt\psi))\geq 2(1+a_{00}).$$
This means that as long as $a_{00}> 0$, any self-adjoint bispectral Darboux transformation $\wt\psi$ of order $(1,1)$ will necessarily have a four-dimensional space of formally bisymmetric bispectral differential operators.  This will be shown by Proposition \ref{fundamental dimension estimate} below to in turn guarantee the existence of a differential operator of order two commuting with an integral operator whose kernel is defined using $\wt\psi$.

\subsection{Existence of bisymmetric operators}
We next use the growth rate estimates established above to prove that $\bisl(\psi)$ must contain bisymmetric operators.
\begin{lem}\label{symmetric form}
A differential operator $\mathfrak d\in\mathfrak D(\bbc(x))$ is formally symmetric if and only if it has the form
\begin{equation}
\label{symm}
\mathfrak d = \sum_{j=0}^n \partial_x^ja_j(x)\partial_x^j
\end{equation}
for some functions $a_0(x),\dots,a_j(x)\in\mathfrak D(\bbc(x))$.
\end{lem}
\begin{proof}
Clearly every differential operator of the \eqref{symm} is formally symmetric. In the opposite direction, 
if $\mathfrak d$ is formally symmetric, then necessarily $\mathfrak d$ must have even order $2n$.
If $n=0$, then the statement of the lemma is true immediately.
As an inductive assumption, suppose that the statement of the lemma is true for formally symmetric operators of order $\leq 2m$.
Let $\mathfrak d$ be an operator of order $2n$ for $n=m+1$, and let $a_n(x)\in\bbc(x)$ be the leading coefficient of $\mathfrak d$.  Then $\mathfrak d - \partial_x^na_n(x)\partial_x^n$ is a formally symmetric operator of order $\leq 2m$.  Therefore by the inductive assumption
$$\mathfrak d -\partial_x^na_n(x)\partial_x^n = \sum_{j=0}^\ell \partial_x^ja_j(x)\partial_x^j$$
for some functions $a_0(x),\dots, a_n(x) \in \bbc(x)$.  The lemma now follows by induction.
\end{proof}

\begin{lem}\label{symmetric condition}
Suppose that $\mathfrak d\in\mathfrak D(\bbc(x))$ is formally symmetric with
$$\mathfrak d = \sum_{j=0}^n \partial_x^ja_j(x)\partial_x^j,$$
and let $\Gamma$ be a smooth path in $\bbc$ with endpoints $p_0,p_1$.
Assume moreover that the poles of $a_j(x)$ are not on $\Gamma$.
If $a_j^{(k)}(p_i) = 0$ for all $0\leq k < j$ and for all $i$ with $p_i\neq \infty$, then $\mathfrak d$ is symmetric with respect to $\Gamma$.
\end{lem}
\begin{proof}
By linearity, it suffices to show that if $a^{(k)}(p_i) = 0$ for all $0\leq j < n$ then $\mathfrak d = \partial_x^na(x)\partial_x^n$ is symmetric with respect to $\Gamma$.  For $f,g\in C_c^\infty(\Gamma)$, integration by parts gives
\begin{align*}
\int_\Gamma (a(x)f(x))^{(n)}g(x) dx
  & = (-1)^n\int_\Gamma f(x) a(x)g^{(n)}(x) dx + \sum_{j=0}^{n-1} (-1)^j (a(x)f(x))^{(n-1-j)}g^{(j)}(x)|_{p_0}^{p_1}\\
  & = (-1)^n\int_\Gamma f(x) a(x)g^{(n)}(x) dx + 0.
\end{align*}
Replacing $f(x)$ with $f^{(n)}(x)$ yields
$$\int_\Gamma(a(x)f^{(n)}(x))^{(n)}g(x)dx = (-1)^n\int_\Gamma f^{(n)}(x) a(x)g^{(n)}(x) dx.$$
Swapping $f$ and $g$ in the last identity leads to
$$(-1)^n\int_\Gamma f^{(n)}(x) a(x)g^{(n)}(x) dx = \int_\Gamma f(x) (a(x)g^{(n)}(x))^{(n)}dx.$$
Combining the two identities, we obtain
\begin{multline*}
\int_\Gamma(\mathfrak d\cdot f(x))g(x)dx = \int_\Gamma(a(x)f^{(n)}(x))^{(n)}g(x)dx \\
= \int_\Gamma f(x) (a(x)g^{(n)}(x))^{(n)}dx = \int_\Gamma f(x)(\mathfrak d\cdot g(x))dx.
\end{multline*}
This shows that $\mathfrak d$ is symmetric with respect to $\Gamma$, completing the proof.
\end{proof}

\begin{prop}\label{fundamental dimension estimate}
Suppose that $\psi(x,y)$ is a bispectral meromorphic function.
Let $\Gamma_1,\Gamma_2$ be two smooth curves in $\bbc$ with the endpoints of $\Gamma_i$ equal to $p_{i0}$ and $p_{i1}$ (one of which is
allowed to be $\infty$ for each $i$).
Assume that the coefficients of the bispectral operators of $\psi(x,y)$ are holomorphic in a neighborhood of $\Gamma_1$ and $\Gamma_2$.
Assume moreover that for each $i=0,1$ either one of the following two conditions holds:
\begin{enumerate}[(i)]
\item  $p_{i0}=-p_{i1}$ and every operator of $\bisls(\psi)$ and $\bisrs(\psi)$ is invariant under the transformations $x\mapsto -x$ and $y\mapsto-y$, respectively, or
\item  one of the points $p_{i0}$ or $p_{i1}$ is $\infty$.
\end{enumerate}
If for some values of $\ell,m$ we have
$$
\dim(\bisls^{2\ell,2m}(\psi)) > \ell(\ell+1)/2 + m(m+1)/2 + 1,
$$
then $\bisl(\psi)$ must contain an operator of positive order which is bisymmetric with respect to $(\Gamma_1,\Gamma_2)$.
\end{prop}
\begin{proof}
Without loss of generality, assume that $p_{i0}$ is a finite point of $\Gamma_i$.
Suppose $\mathfrak d\in \bisls^{2\ell,2m}(\psi)$ with $b_\psi(\mathfrak d) = \mathfrak b$.
By Lemma \ref{symmetric form},
$$\mathfrak d = \sum_{j=0}^{\ell}\partial_x^j a_j(x)\partial_x^j \ \ \text{and}\ \ \mathfrak b = \sum_{j=0}^{m}\partial_y^j b_j(y)\partial_y^j$$
for some functions $a_j(x),b_j(y)$ which are holomorphic in neighborhoods of $p_{i0}$, for $i=1,2$ respectively. 
The linear map $\bisls^{\ell,m}(\psi)\rightarrow \bbc^{\ell(\ell+1)/2}\oplus \bbc^{m(m+1)/2}$
defined by
$$\mathfrak d\mapsto (a_j^{(r)}(p_{10})_{0\leq r < j \leq \ell}, b_k^{(s)}(p_{20})_{0\leq s < k\leq m})$$
has a kernel of dimension at least
$$\dim(\bisls^{\ell,m}(\psi))-\frac{\ell(\ell+1)}{2}-\frac{m(m+1)}{2} \cdot$$
Moreover if condition (i) or condition (ii) is satisfied, then the elements in the kernel are bisymmetric with respect to $(\Gamma_1,\Gamma_2)$ by Lemma \ref{symmetric condition}.
\end{proof}

\section{Differential operators commuting with integral operators}
\label{main thm a}
In this section, we prove the main theorems in the paper that self-adjoint bispectral meromorphic functions on ranks 1 and 2 give rise 
to integral operators possessing a commuting differential operator. The kernel of the integral operator corresponding to such a bispectral function $\wt\psi(x,y)$ is given by 
\begin{equation}
\label{kernel}
\wt K(x,y) = \int_{\Gamma_2}\wt  \psi(x,z)\wt \psi(y,z)dz
\end{equation}
for an appropriate smooth curve $\Gamma_2 \subseteq \bbc$.
The commuting differential operator will come from $\bisls(\wt\psi)$.
In order for the value of $\wt K(x,y)$ to exist and to define a kernel for an integral operator with with the desired domain, we have to make certain assumptions about $\wt\psi(x,y)$.
We specifically assume that $\wt\psi(x,y)$ is holomorphic in a neighborhood of $\Gamma_1\times\Gamma_2$ and for all $j,k,m,n\geq 0$
\begin{equation}\label{niceness condition}
\int_{\Gamma_1} |x^my^n\partial_x^j\partial_y^k\cdot \wt\psi(x,y)| dx\in L^\infty(\Gamma_2)\ \ \text{and}\ \ \int_{\Gamma_2} |x^my^n\partial_x^j\partial_y^k\cdot \wt\psi(x,y)| dy\in L^\infty(\Gamma_1).
\end{equation}
Note that the inclusion of the $x^my^n$ multiplier is vacuous unless $\Gamma_1$ or $\Gamma_2$ has an endpoint at infinity, in which case the condition 
imposes a lower bound on the rate of decay of the partial derivatives of $\wt\psi(x,y)$.
Under the above assumptions, the integral formula for $\wt K(x,y)$ exists and satisfies the assumptions required for differentiation under the integral for operators of arbitrary order.
Therefore $\wt K(x,y)$ is holomorphic in a neighborhood of $\Gamma_1\times\Gamma_1$ and 
$$\partial_x^j\partial_y^k\cdot \wt K(x,y) = \int_{\Gamma_2}(\partial_x^j\cdot\wt\psi(x,z))(\partial_y^k\cdot\wt\psi(y,z))dz.$$
Furthermore we have the norm estimate
$$\|x^my^n\partial_x^j\partial_y^k\cdot \wt K(x,y)\|_{1,\Gamma_1} \leq \left\|\int |x^my^n\partial_x^j\cdot \wt\psi(x,y)| dy\right\|_{\infty,\Gamma_1}
\left\|\int |\partial_x^k\cdot \wt\psi(x,y)| dy\right\|_{\infty,\Gamma_1},$$
where the $L^1$-norm on the left is taken with respect to either $x$ or $y$.
In particular \eqref{niceness condition} also holds with $\wt\psi(x,y)$ replaced with $\wt K(x,y)$ and $\Gamma_2$ replaced with $\Gamma_1$, 
and so differentiation under the integral may also be applied to the integral operator defined with kernel $\wt K(x,y)$.

\subsection{Relationship to bisymmetric operators}
The next theorem establishes that every bisymmetric differential operator in the Fourier algebra of a bispectral function $\psi(x,y)$ 
automatically commutes with the integral operator with kernel \eqref{kernel}. 
\begin{thm}\label{commuting bisymmetry}
Suppose that ${\wt\psi}(x,y)$ is a bispectral meromorphic function and that $\Gamma_1,\Gamma_2\subseteq\bbc$ are two smooth curves such that
${\wt\psi}$ is holomorphic in a neighborhood of $\Gamma_1\times\Gamma_2$ and satisfies \eqref{niceness condition}.
If $\mathfrak d \in \bisls({\wt\psi})$ is a bisymmetric differential operator with respect to $(\Gamma_1,\Gamma_2)$ whose coefficients are holomorphic in a neighborhood of $\Gamma_1$, then it commutes with the integral operator
$$T: f(x)\mapsto \int_{\Gamma_1}\wt K(x,y)f(y)dy \quad \text{with kernel} \quad \wt K(x,y) = \int_{\Gamma_2} {\wt\psi}(x,z){\wt\psi}(y,z) dz.$$
\end{thm}
\begin{proof} Denote $\mathfrak b = b_{\wt\psi}(\mathfrak d)$. 
Differentiation under the integral along with the fact that $\mathfrak b$ is symmetric with respect to $\Gamma_2$ implies
\begin{align*}
&\mathfrak d_x\cdot \wt K(x,y) =  \int_{\Gamma_2} (\mathfrak d_x \cdot  {\wt\psi}(x,z)) {\wt\psi}(y,z)dz 
\\
&= \int_{\Gamma_2} (\mathfrak b_z \cdot  {\wt\psi}(x,z)) {\wt\psi}(y,z)dz = \int_{\Gamma_2} {\wt\psi}(x,z) (\mathfrak b_z \cdot  {\wt\psi}(y,z)) dz
\\
&= \int_{\Gamma_2} {\wt\psi}(x,z) (\mathfrak d_y  \cdot  {\wt\psi}(y,z)) dz = \mathfrak d_y\cdot \wt K(x,y).
\end{align*}
Moreover for any $f\in C_c^\infty(\bbc)$, Fubini's theorem, differentiation under the integral and the fact that $\mathfrak d$ is symmetric with respect to $\Gamma_1$ imply
\begin{align*}
& \mathfrak d\cdot T(f) = \int_{\Gamma_1}(\mathfrak d_x\cdot \wt K(x,y)) f(y)dy =  \int_{\Gamma_1}(\mathfrak d_y \cdot \wt K(x,y)) f(y)dy
\\
& = \int_{\Gamma_1}\wt K(x,y) ((\mathfrak d_y \cdot  f(y) ) dy = T(\mathfrak d\cdot f).
\end{align*}
It follows that $\mathfrak d$ and $T$ commute.
\end{proof}
\subsection{The rank $1$ case}
\begin{lem}\label{exp lemma}
The exponential bispectral function $\psi_{\exp}(x,y) = e^{xy}$ satisfies
$$\bisls^{2\ell,2m}(\psi_{\exp}) = \left\lbrace\sum_{j=0}^\ell \partial_x^ja_j(x^2)\partial_x^j: \deg a_j(x^2) \leq 2m\right\rbrace.$$
In particular,
$$\dim\bisls^{2\ell,2m}(\psi_{\exp}) = \ell m + \ell + m + 1.$$
\end{lem}
\begin{proof} From Example \ref{Fourier-ex}(1) we have that 
$\bisl(\psi_{\exp}) = \mathfrak D(\bbc[x])$ and $\bisr(\psi_{\exp}) = \mathfrak D(\bbc[y])$, 
so that both $\bisl(\psi_{\exp})$ and $\bisr(\psi_{\exp})$ are closed under $*$.  
Recall that $\bisl(\psi_{\exp})$ is closed under the adjoint $*$ so that the automorphism $\iota = \iota_{\psi_{\exp},x}$ defined by \eqref{iota} is well-defined.
One may readily check that the automorphism $\iota$ agrees with the endomorphism $\sigma$ of the Weyl algebra induced by the affine transformation $x\mapsto -x$, ie.
$$\sigma: \sum_{j=0}^\ell a_j(x)\partial_x^j \mapsto \sum_{j=0}^\ell a_j(-x)(-1)^j\partial_x^j.$$
Therefore the formally bisymmetric operators in $\bisls(\psi_{\exp})$ are exactly those fixed by $\sigma$, which are those operators of the form
$$\mathfrak d = \sum_{j=0}^\ell \partial_x^ja_j(x^2)\partial_x^j$$
for some integer $\ell\geq 0$ and some functions $a_j(x^2)\in\bbc[x]$.
The corresponding element of $\bisrs(\psi_{\exp})$ is
$$b_{\psi_{\exp}}(\mathfrak d) = \sum_{j=0}^\ell y^ja_j(\partial_y^2)y^j.$$
From this the statement of the lemma follows immediately.
\end{proof}
\begin{thm}\label{exp theorem}
Let ${\wt\psi}(x,y)$ be a self-adjoint bispectral meromorphic function of rank $1$, and let $\Gamma_1$ and $\Gamma_2$ be two finite, smooth curves in $\bbc$ whose endpoints are $\pm p_1$ and $\pm p_2$, respectively.
Assume moreover that the coefficients of the operators in $\bisl({\wt\psi})$ and $\bisr({\wt\psi})$ are holomorphic in a neighborhood of $\Gamma_1$ and $\Gamma_2$, respectively, and that ${\wt\psi}(x,y)$ is holomorphic in a neighborhood of $\Gamma_1\times\Gamma_2$ and satisfies \eqref{niceness condition}.
Then there exists a differential operator $\mathfrak d$ of positive order commuting with the integral operator
$$T: f(x)\mapsto \int_{\Gamma_1}\wt K(x,y)f(y)dy \quad \text{with kernel} \quad \wt K(x,y) = \int_{\Gamma_2} {\wt\psi}(x,z){\wt\psi}(y,z) dy.$$
Moreover, the operator $\mathfrak d$ may be taken to be in $\bisls^{2d_1d_2,2d_1d_2}(\wt\psi)$, for $(d_1,d_2)$ the bidegree of the self-adjoint bispectral Darboux 
transformation from $\psi_{\exp}(x,y)=e^{xy}$ to $\wt\psi(x,y)$.
\end{thm}
\begin{remk}
The rank one bispectral functions ${\wt\psi}(x,y)$ are all of the form $e^{xy}\frac{h(x,y)}{p(x)q(y)}$ for some polynomials $p(x),q(y),h(x,y)$.  The poles of the coefficients of operators in $\bisl({\wt\psi})$ and $\bisr({\wt\psi})$ occur at the zeros of $p(x)$ and $q(y)$, respectively.
In particular, as long as $\Gamma_1$ avoids the zeros of $p(x)$ and $\Gamma_2$ avoids the zeros of $q(y)$, and the endpoints of $\Gamma_1$ and $\Gamma_2$ satisfy the desired symmetry, the assumptions of the above theorem will be satisfied.
\end{remk}
\begin{proof}
Suppose that ${\wt\psi}(x,y)$ is a self-adjoint bispectral meromorphic function of rank $1$.  Then by Lemma \ref{exp lemma} and Corollary \ref{dimension cor}, we have
$$\dim\bisls^{2\ell,2m}({\wt\psi})\geq \ell m + \ell + m + 2 - d_1d_2.$$
Thus for $\ell=m$, we have
$$\dim\bisls^{2m,2m}({\wt\psi})\geq m^2 + 2m + 2 - d_1d_2 > m^2 + m + 1,$$
for $m\geq d_1d_2$.
Proposition \ref{fundamental dimension estimate} combined with Corollary \ref{automatic symmetry} implies that $\bisls^{2d_1d_2,2d_1d_2}({\wt\psi})$ contains a differential operator bisymmetric 
with respect to $(\Gamma_1,\Gamma_2)$. This differential operator commutes with the integral operator $T$ by Theorem \ref{commuting bisymmetry}.
\end{proof}

\subsection{The rank $2$ Airy case}
We next deal with the bispectral Darboux transformations in the rank $2$ Airy case.
The relevant integral operator differs in this case from the integral operators in the rank $1$ case and in the rank $2$ Bessel case (discussed below) 
in that the kernel is not compactly supported. For the resultant kernel to satisfy \eqref{niceness condition}, the support must be contained in a certain subdomain of the complex plane.
For this reason, for all $\epsilon>0$ we consider the domain
$$\Sigma_\epsilon = \{re^{i\theta}\in\bbc: r> 0,\ |\theta| < \pi/6-\epsilon\}.$$
The Airy function is holomorphic on this domain and has the asymptotic expansion
$$\psi_{\Ai}(x+y) = e^{-2(x+y)^{3/2}/3}\left(\sum_{j=1}^\infty c_j (x+y)^{-j/4}\right),$$
for some real constants $c_j\in\bbr$ where $(x+y)^{1/4}$ is interpreted as the principal $4$th root of $(x+y)$.
Furthermore, any bispectral Darboux transformation of $\psi_{\Ai}(x+y)$ will be equal to $\wt\psi(x,y) = \frac{1}{p(x)q(y)}\mathfrak u\cdot\psi_{\Ai}(x+y)$ for some rational functions $p(x),q(y)$ and some differential operator $\mathfrak u$ with rational coefficients.
Thus for any bispectral Darboux transformation of $\psi_{\Ai}(x+y)$ we have the asyptotic estimate
$$\|\partial_x^j\partial_y^k\cdot\psi_{\Ai}(x+y)\| = e^{-2(x+y)^{3/2}/3}\mathcal O((|x|+|y|)^{(j+k)/2+m})$$
for some integer $m$.

Note that $z\mapsto 2z^{3/2}/3$ sends $\Gamma$ to $\{re^{i\theta}\in\bbc: r>0,\ |\theta| < \pi/4-3\epsilon/2\}$.
Therefore if $\Gamma_1,\Gamma_2\subseteq\Sigma_\epsilon$ are smooth, semi-infinite curves inside this domain with parametrizations $\gamma_i(t): [0,\infty)\rightarrow\bbc$ then the real part of $-2(\gamma_1(t)+\gamma_2(s))^{3/2}/3$ must go to $-\infty$ as $t\rightarrow\infty$ or $s\rightarrow\infty$.
Therefore the above asymptotic estimate shows that $\wt\psi(x,y)$ will satisfy \eqref{niceness condition} for any pair of curves $\Gamma_1,\Gamma_2\subseteq\Sigma_\epsilon$.
\begin{lem}\label{airy lemma}
The Airy bispectral function $\psi_{\Ai}(x,y) = \Ai(x+y)$ satisfies
$$\dim\bisls^{2\ell,2m}(\psi_{\Ai}) = \ell m + \ell + m + 1.$$
\end{lem}
\begin{proof} Recall from Example \ref{Fourier-ex}(3) that 
$$
\bisl(\psi_{\Ai}) = \mathfrak D(\bbc[x]) \ \ \text{and} \ \
\bisr(\psi_{\Ai}) = \mathfrak D(\bbc[y]).
$$
In particular, $\bisl(\psi_{\Ai})$ is closed under $*$. From the description of $b_{\psi_{\Ai}}$ in Example \ref{Fourier-ex}(3) one easily sees that 
$b_{\psi_{\Ai}}$ and $*$ commute. Therefore $\mathfrak d \in \bisl(\psi_{\Ai})$ is formally bisymmetric if and only if $\mathfrak{d}^* = \mathfrak{d}$.

Recall that $\bislf(\psi_{\Ai})$ contains all polynomials in $x$ as well as the Airy operator $\mathfrak d_{\Ai,x} = \partial_x^2-x$.
Moreover for $\mathfrak d\in \bisl^{\ell,m}(\psi_{\Ai})$ the Ad-condition $\ad_{\mathfrak d_{\Ai_x}}^{m+1}(\mathfrak d) = 0$ implies that the leading coefficient of $\mathfrak d$ must be a polynomial.
For each $j,k\geq 0$ consider the formally symmetric differential operator in $\bisl(\psi_{\Ai})$ defined by
$$\mathfrak a_{jkx} = \mathfrak d_{\Ai,x}^jx^k + x^k\mathfrak d_{\Ai,x}^j.$$
which in particular has order $2j$ and leading coefficient $2x^k$.
The anti-isomorphism $b_{\psi_{\Ai}}$ sends $\mathfrak a_{jkx}$ to $\mathfrak a_{jky}$, so $\mathfrak a_{jkx}\in\bisls(\psi_{\Ai})$.

If $\mathfrak d$ is self-adjoint, it must have even order with polynomial leading coefficient.
Thus by comparing leading coefficients, we see that $\bisls(\psi_{\Ai})$ has basis $\{\mathfrak a_{jkx}: j,k\geq 0\}$.
Therefore comparing orders we see that $\bisls^{2\ell,2m}(\psi_{\Ai})$ has basis
$$\{\mathfrak a_{jix}: 0\leq j\leq l,\ 0\leq k\leq m\}.$$
In particular it has dimension $(\ell+1)(m+1)$.
\end{proof}

\begin{thm}\label{airy theorem}
Let $\wt\psi(x,y)$ be a self-adjoint bispectral Darboux transformation of the Airy bipectral function $\psi_{\Ai}(x,y)= \Ai(x+y)$, and let $\Gamma_1$ 
and $\Gamma_2$ be two semi-infinite, smooth curves in $\Sigma_\epsilon$ for some $\epsilon > 0$ whose finite endpoints are $p_1$ and $p_2$, respectively.
Assume moreover that $\wt \psi(x,y)$ is holomorphic in a neighborhood of $\Gamma_1\times\Gamma_2$ and satisfies \eqref{niceness condition} and that the operators $\bisl(\wt\psi)$ and $\bisr(\wt\psi)$ have holomorphic coefficients in a neighborhood of $\Gamma_1$ and $\Gamma_2$, respectively.
Then there exists a differential operator $\mathfrak d$ of positive order commuting with the integral operator
$$T: f(x)\mapsto \int_{\Gamma_1}\wt K(x,y)f(y)dy \quad \text{with kernel} \quad \wt K(x,y) = \int_{\Gamma_2} \wt\psi(x,z)\wt\psi(y,z) dy.$$
Moreover, the operator $\mathfrak d$ may be taken to be in $\bisls^{2d_1d_2,2d_1d_2}(\wt\psi)$, for $(d_1,d_2)$ the bidegree the self-adjoint bispectral Darboux 
transformation from $\psi_{\Ai}(x,y)$ to $\wt\psi(x,y)$.
\end{thm}
\begin{remk}
The bispectral Darboux transformations of the Airy bispectral function will be holomorphic away from the roots of a polynomial $p(x)$ and a polynomial $q(y)$, which are 
the polynomials entering in the definition of the concrete bispectral Darboux transformation as in \eqref{dual-poly}. 
Furthermore, the poles of the coefficients of the operators in $\bisl(\wt\psi)$ and $\bisr(\wt\psi)$ occur at the roots of these polynomials also.
Therefore the assumptions of the theorem will be automatically satisfied as long as $\Gamma_1$ and $\Gamma_2$ are both semi-infinite paths in $\Sigma_\epsilon$ which avoid the zero sets of $p(x)$ and $q(y)$.
\end{remk}
\begin{proof}
By Lemma \ref{airy lemma} and Corollary \ref{dimension cor}, we have
$$\dim\bisls^{2\ell,2m}(\wt\psi)\geq \ell m + \ell + m + 2 - d_1d_2$$
for some constant $c$.  Thus for $\ell=m$, we have
$$\dim\bisls^{2m,2m}(\wt\psi)\geq m^2 + 2m + 2 - d_1d_2 > m^2 + m + 1$$
for $m\geq d_1d_2$. By Proposition \ref{fundamental dimension estimate}, $\bisls^{2d_1d_2,2d_1d_2}(\wt\psi)$ contains a differential operator 
which is bisymmetric with respect to $(\Gamma_1,\Gamma_2)$. It follows from Theorem \ref{commuting bisymmetry} that
this differential operator commutes with the integral operator $T$. 
\end{proof}
\subsection{The rank $2$ Bessel case}
\begin{lem}\label{bessel lemma}
Let $\nu\in\bbr\diff\bbz$.  The Bessel bispectral function $\psi_{\Be(\nu)}(x,y) = \sqrt{xy}K_{\nu+1/2}(xy)$ satisfies
$$\dim\bisls^{2\ell,2m}(\psi_{\Be(\nu)}) = \ell m + \ell + m + 1.$$
\end{lem}
\begin{proof} It follows from Example \ref{Fourier-ex}(2) that the algebra $\bisl(\psi_{\Be(\nu)})$ is generated 
over $\bbc$ by the Bessel operator $\mathfrak d_{\Be(\nu)x} = \partial_x^2-\frac{\nu(\nu+1)}{x^2}$ the operator $\mathfrak s_x= x \partial_x$ and $x^2$.  
In particular, all of the elements of $\bisl(\psi_{\Be(\nu)})$ are invariant under the change of coordinates $x\mapsto -x$.
Furthermore, for all $j,k\geq 0$,
$$\mathfrak a_{jk} = \mathfrak s_x^k\mathfrak d_{\Be(\nu)x}^j(\mathfrak s_x^*)^k \quad \text{and} \quad \mathfrak b_{jk} = \mathfrak s_x^k x^{2j}(\mathfrak s_x^*)^k$$
are formally symmetric differential operators with $\ord(\mathfrak a_{jk}) = 2j+2k$, $\ord(\mathfrak b_{jk}) = 2k$, $\cord(\mathfrak a_{jk}) = 2k$ and $\cord(\mathfrak b_{jk}) = 2j+2k$.
In particular, $\ord(\mathfrak a_{jk}) \geq \cord(\mathfrak a_{jk})$ and $\ord(\mathfrak b_{jk}) \leq \cord(\mathfrak b_{jk})$ with equality if and only if $j$ is zero.  Thus the set
$$\{\mathfrak a_{jk}: j+k\leq \ell,k\leq m\}\cup\{\mathfrak b_{jk}: j+k\leq m, k\leq \ell, j\neq 0\}$$
is a linearly independent collection of $(\ell+1)(m+1)$ elements of $\bisls^{2\ell,2m}(\psi_{\Be(\nu)})$.

If $\mathfrak d\in\bisls(\psi^{2\ell,2m}_{\Be(\nu)})$ is an arbitrary operator, then the Ad-condition $\ad_{\mathfrak d_{\Be(\nu),x}}^{2m+1}(\mathfrak d) = 0$ implies that the leading coefficient of $\mathfrak d$ must be a polynomial.
Furthermore, since $\mathfrak d$ is formally symmetric it must have even order.
Finally, since $\mathfrak d$ must be invariant under the transformation $x\mapsto-x$, the leading coefficient of $\mathfrak d$ must be a polynomial in $x^2$.
Therefore by comparing leading coefficients, we see that $\mathfrak d$ must lie in the span of the $\mathfrak a_{jk}$ and $\mathfrak b_{jk}$.
Thus the set $\{\mathfrak a_{jk},\mathfrak b_{jk}: j,k\geq 0\}$ forms a basis for $\bisls(\psi_{\Be(\nu)})$ and by noting the orders and co-orders, we see that the linearly independent collection noted in the previous paragraph is actually a basis for $\bisls^{2\ell,2m}(\psi_{\Be(\nu)})$.
\end{proof}

\begin{thm}\label{bessel theorem}
Let $\wt\psi(x,y)$ be a self-adjoint bispectral Darboux transformation of the bispectral Bessel function $\psi_{\Be(\nu)}$, and let $\Gamma_1$ and $\Gamma_2$ be two finite, smooth curves in $\bbc$ whose endpoints are $\pm p_1$ and $\pm p_2$, respectively.
Assume moreover that the coefficients of the operators in $\bisl(\wt\psi)$ and $\bisr(\wt\psi)$ are holomorphic in a neighborhood of $\Gamma_1$ and $\Gamma_2$, respectively, and that $\wt\psi(x,y)$ is holomorphic in a neighborhood of $\Gamma_1\times\Gamma_2$ and satisfies \eqref{niceness condition}.
Then there exists a differential operator $\mathfrak d$ of positive order commuting with the integral operator
$$T: f(x)\mapsto \int_{\Gamma_1}\wt K(x,y)f(y)dy \quad \text{with kernel} \quad \wt K(x,y) = \int_{\Gamma_2} \wt\psi(x,z)\wt\psi(y,z) dy.$$
Moreover, the operator $\mathfrak d$ may be taken to be in $\bisls^{2d_1d_2,2d_1d_2}(\wt\psi)$, for $(d_1,d_2)$ the bidegree the self-adjoint bispectral Darboux 
transformation from $\psi_{\Be(\nu)}$ to $\wt\psi(x,y)$.
\end{thm}
\begin{remk}
As in the rank $1$ and the Airy case, the assumptions of Theorem \ref{bessel theorem} will be automatically satisfied as long as the end points of $\Gamma_1$ and $\Gamma_2$
satisfy the symmetry condition stated above and the curves avoid the branching point 0 of the Bessel functions and the roots of the polynomials $p(x)$ and $q(y)$, 
entering in the definition of the concrete bispectral Darboux transformation as in \eqref{dual-poly}. 
\end{remk}
\begin{proof}
By Lemma \ref{bessel lemma} and Corollary \ref{dimension cor}, we have
$$\dim\bisls^{2\ell,2m}(\wt\psi)\geq \ell m + \ell + m + 2 - d_1d_2$$
for some constant $c$.  Thus for $\ell=m$, we have
$$\dim\bisls^{2m,2m}(\wt\psi)\geq m^2 + 2m + 2 - d_1d_2 > m^2 + m + 1,$$
for $m\geq d_1d_2$.  Proposition \ref{fundamental dimension estimate} and Corollary \ref{automatic symmetry} imply that $\bisls^{2d_1d_2,2d_1d_2}(\wt\psi)$ 
contains a differential operator bisymmetric with respect to $(\Gamma_1,\Gamma_2)$. This differential operator commutes with the integral operator specified in the statement 
of the proposition by Theorem \ref{commuting bisymmetry}.
\end{proof}
\section{Examples}
In this section we illustrate how Theorems \ref{exp theorem}, \ref{airy theorem}, and \ref{bessel theorem} are applied to construct differential operators commuting with the 
the integral operators associated to all self-adjoint bispectral Darboux transformations $\wt\psi(x,y)$ from the elementary bispectral functions 
$\psi_{\exp}$, $\psi_{\Ai}$, and $\psi_{\Be(\nu)}$. Among the examples are an integral operator which commutes simultaneously with two differential operators of orders 6 and 8, and 
an integral operator which commutes with a differential operator of order 22.

{\em{The proofs of Theorems \ref{exp theorem}, \ref{airy theorem}, and \ref{bessel theorem}  are constructive and give rise to an effective algorithm 
for the construction of a commuting differential operator. This algorithm only involves solving linear systems of equations.}}
\subsection{The algorithm}
\label{6.1}
To describe the constructive algorithm mentioned above, we first review the definition of the bilinear concomitant.
\begin{defn}
\label{conco}
Let $\mathfrak d$ be a differential operator
$$\mathfrak d = \sum_{j=0}^m d_j(x)\partial_x^j.$$
The {\bf{bilinear concomitant}} of $\mathfrak d$ is the bilinear form $\mathcal{C}_{\mathfrak d}(\cdot,\cdot; p)$ defined on pairs of sufficiently smooth functions $f(x),g(x)$ by
\begin{align*}
\mathcal{C}_{\mathfrak d}(f,g;p)
  & = \sum_{j=1}^m \sum_{k=0}^{j-1} (-1)^k f^{(j-1-k)}(x)(d_j(x)g(x))^{(k)}|_{x=p}\\
  & = \sum_{j=1}^m \sum_{k=0}^{j-1}\sum_{\ell=0}^k\binom{k}{\ell} (-1)^k f^{(j-1-k)}(x)d_j(x)^{(k-\ell)}g(x)^{(\ell)}|_{x=p}.
\end{align*}
Equivalently, for $C_{\mathfrak d}(x)$ the $m\times m$ matrix whose $n,\ell$-th entry is given by
\begin{equation}
\label{C-matr}
C_{\mathfrak d}(x)_{n,\ell} = \sum_{j=\ell+n+1}^{m}\binom{j-n}{\ell-1} (-1)^{j-n} d_j(x)^{(j+1-n-\ell)},
\end{equation}
the bilinear concomitant may be expressed as
$$\mathcal{C}_{\mathfrak d}(f,g;p) = [f(x)\ f'(x)\ \dots\ f^{(m)}(x)] C_{\mathfrak d}(x)[g(x)\ g'(x)\ \dots\ g^{(m)}(x)]^T|_{x=p}.$$
\end{defn}
Note also that via integration by parts, we find that
\begin{equation}\label{integral biconcomitant}
\int_{x_0}^{x_1}[(\mathfrak d\cdot f(x))g(x)-f(x)(\mathfrak d^*\cdot g(x))] dx = \mathcal{C}_{\mathfrak d}(f,g;x_1)-\mathcal{C}_{\mathfrak d}(f,g;x_0).
\end{equation}
In this way the bilinear concomitant may be seen to act as a means of comparison between the formal adjoint $\mathfrak d^*$ of a differential operator $\mathfrak d$ and the adjoint of $\mathfrak d$ as an unbounded linear operator on a sufficiently nice space of functions on a path connecting $x_0$ to $x_1$.

{\bf{Algorithm:}} Let $\wt{\psi}(x,y)$ be a self-adjoint bispectral Darboux transformation from one of the elementary bispectral functions 
$\psi_{\exp}$, $\psi_{\Ai}$ or $\psi_{\Be(\nu)}$ of order $(d_1,d_2)$. Fix contours $\Gamma_1$ and $\Gamma_2$ as in 
Theorems \ref{exp theorem}, \ref{airy theorem}, and \ref{bessel theorem}. The following algorithm constructs a 
differential operator commuting with the integral operator
\begin{equation}
\label{oper}
T: f(x)\mapsto \int_{\Gamma_1}\wt{K}(x,y)f(y)dy \quad \text{with kernel} \quad \wt{K}(x,y) = \int_{\Gamma_2} \wt{\psi}(x,z)\wt{\psi}(y,z) dy.
\end{equation}
\begin{enumerate}
\item[(Sketch)]  The dimension estimates from Corollary \ref{dimension cor} and Lemmas \ref{exp lemma}, \ref{airy lemma}, and \ref{bessel lemma} 
imply that
$$\dim(\bisls^{2\ell,2m}(\wt{\psi}))\geq \ell m + \ell+m-d_1d_2+1.$$
By Proposition \ref{fundamental dimension estimate} for $\ell :=d_1d_2+1$ the space $\bisls^{2\ell,2\ell}(\wt{\psi})$
 will contain a nonconstant differential operator which is bisymmetric with respect to $\Gamma_1\times\Gamma_2$.
Theorem \ref{commuting bisymmetry} implies that this operator will commute with the integral operator in \eqref{oper}.

\item[(Step 1)] Fix $\ell = d_1d_2+1$ and construct explicitly a $(\ell^2 + \ell + 2)$-dimensional subspace of $\bisls^{2\ell,2\ell}(\wt{\psi})$  as given by the LHS of \eqref{subspace}
in the proof of Theorem \ref{control thm}.
In \eqref{subspace} one plugs the explicit description of $\bisls^{2\ell,2\ell}(\psi)$ for $\psi = \psi_{\exp}$, $\psi_{\Ai}$ or $\psi_{\Be(\nu)}$ 
from Lemmas \ref{exp lemma}, \ref{airy lemma}, and \ref{bessel lemma}.

\item[(Step 2)]
A differential operator $\mathfrak d$ in this $(\ell^2+\ell+2)$-dimensional space will be bisymmetric with respect to $\Gamma_1\times\Gamma_2$
if and only if it is in the kernel of the $\bbc$-linear map
$$B: \bisls^{2\ell,2\ell}(\wt{\psi})\rightarrow \bbc^{(\ell+1)^2}\times \bbc^{(\ell+1)^2},
\ \ \mathfrak d\mapsto ((\mathcal{C}_{\mathfrak d}(x^j,x^k;p_1)_{j,k=0}^{\ell},(\mathcal{C}_{b_{\wt\psi}(\mathfrak d)}(y^j,y^k;p_2))_{j,k=0}^\ell),$$
where $\mathcal{C}_{\mathfrak d}$ is the bilinear concomitant of $\mathfrak d$, $p_1$ is one of the finite endpoints of $\Gamma_1$ and $p_2$ is 
one of the final endpoints of $\Gamma_2$. We solve this homogeneous linear system of equations 
and find a vector space of operators $\mathfrak d$ which will commute with the integral operator in \eqref{oper} by Theorems 
\ref{exp theorem}, \ref{airy theorem}, and \ref{bessel theorem}.
Note that it follows from Lemma \ref{symmetric form} that the rank of this system of linear equations is necessarily $\leq \ell(\ell+1)$, so that the space of solutions will be at least two dimensional.
In particular, it must contain at least one non-constant operator.
\end{enumerate}

\label{ex}
\subsection{Kernels of integral operators}
Some of our examples will involve bispectral functions $\wt\psi(x,y)$ which are families of eigenfunctions of a fixed Schr\"{o}dinger operator $\partial_y^2-v(y)\in \bisr(\wt\psi)$.
Before moving into those examples, we derive a second expression for for these types of kernels $\wt K(x,y)$ which does not involve integration. 
\begin{lem}
\label{2nd-kernel}
Suppose that $\wt\psi(x,y)$ is a bispectral meromorphic function satisfying the equation
$$\wt\psi_{yy}(x,y) - v(y)\wt\psi(x,y) = f(x)\wt\psi(x,y).$$
Let $\Gamma$ be a smooth curve in the extended complex plain with endpoints $p_1,p_2$ such that $\wt\psi(x,y)$ is holomorphic in a neighborhood of $\Gamma$ and for some fixed $x_0,y_0\in\bbc$ with $x_0\neq y_0$ we have $f(x_0)\neq f(y_0)$, and $v(y),\wt K(x,y),K_y(x,y),K_{yy}(x,y),\in L^2(\Gamma)$ for $x=x_0,y_0$.
Then $\wt K(x,y) = \int_\Gamma \wt\psi(x,z)\wt\psi(y,z)dz$ satisfies
\begin{equation}
\label{kernel equation}
\wt K(x_0,y_0) = \frac{1}{f(x_0)-f(y_0)}(\wt\psi_z(x_0,z)\wt\psi(y_0,z)-\wt\psi(x_0,z)\wt\psi_z(y_0,z))|_{z=p_1}^{p_2}.
\end{equation}
\end{lem}
\begin{remk}
The assumptions that we make regarding the kernel $\wt K(x,y)$ are weaker than \eqref{niceness condition}, and in particular if we identify $\Gamma=\Gamma_2$ then \eqref{niceness condition} implies the assumptions for $\wt K(x,y)$ in the lemma for a.e. $x$ on the curve $\Gamma_1$.
\end{remk}
\begin{proof}
We see that
$$\int_\Gamma \wt\psi_{zz}(x_0,z)\wt\psi(y_0,z)dz = \int_\Gamma v(z)\wt\psi(x_0,z)\wt\psi(y_0,z)dz + f(x_0)\wt K(x_0,y_0),$$
$$\int_\Gamma \wt\psi(x_0,z)\wt\psi_{zz}(y_0,z)dz = \int_\Gamma v(z)\wt\psi(x_0,z)\wt\psi(y_0,z)dz + f(y_0)\wt K(x_0,y_0),$$
and therefore
\begin{align*}
(f(y_0)-f(x_0))\wt K(x_0,y_0)
  & = \int_{\Gamma} (\wt\psi(x_0,z)\wt\psi_{zz}(y_0,z)-\wt\psi_{zz}(x_0,z)\wt\psi(y_0,z))dz\\
  & = (\wt\psi_z(x_0,z)\wt\psi(y_0,z)-\wt\psi(x_0,z)\wt\psi_z(y_0,z))|_{z=p_1}^{p_2}.
\end{align*}
Dividing by $f(y_0)-f(x_0)$ we obtain the expression stated in the lemma.
\end{proof}
In the case that $\wt\psi(x,y)$ is a self-adjoint bispectral Darboux transformation of $\psi(x,y)$ we can also relate the kernel $\wt K(x,y)$ obtained from $\wt\psi(x,y)$ to the kernel $K(x,y)$ obtained from $\psi(x,y)$. This coupled with Lemma \ref{2nd-kernel}, gives a second expression for the kernel $\wt K(x,y)$ which does not involve integration. 
\begin{lem}
Suppose that $\psi(x,y)$ is a bispectral meromorphic function and that $\wt\psi(x,z) = \frac{1}{p(x)q(z)}\mathfrak w_z\cdot\psi(x,z)$ is a self-adjoint 
bispectral Darboux transformation of $\psi(x,z)$.
Let $\Gamma$ be a smooth curve in the extended complex plain with endpoints $p_1,p_2$ such that $\psi(x,y)$ is holomorphic in a neighborhood of $\Gamma$ and for some fixed $x_0,y_0\in\bbc$ which are not poles of $p(x)$ the partial derivatives $\partial_y^j\psi(x,y)\in L^2(\Gamma)$ for all $j\geq 0$ and $x=x_0,y_0$.
Then 
\begin{multline}
\label{Darboux kernel}
\wt K(x,y) = \frac{p(y)}{p(x)} K(x,y) + \\ \sum_{k=0}^{d_2} \sum_{j=0}^{k-1} b_j(y) (-1)^j \left[ \big( \partial^{k-1-j}_z \psi(x,z) \big) \, 
\partial_z^j \left(\frac{b_k(z)}{p(x)q(z)}\wt \psi(y,z) \right) \right]_{z=p_1}^{p_2}.
\end{multline}
where $b_k(z)$ denote the coefficients of $\mathfrak w_z$, i.e., $\mathfrak w_z = \sum_{k=0}^{d_2} b_j(z)\partial_z^j$ and $p_1$ and $p_2$ are the endpoints of $\Gamma$.
\end{lem}
\begin{remk}
The assumptions that we make regarding the kernel $K(x,y)$ are again implied by \eqref{niceness condition}.
\end{remk}
\begin{proof}
Note that
$$\wt K(x,y) = \int_\Gamma \wt\psi(x,z)\wt\psi(y,z) dz = \frac{1}{p(x)p(y)}\int_\Gamma \frac{1}{q(z)^2}(\mathfrak w_z\cdot\psi(x,z))(\mathfrak w_z\cdot\psi(y,z))dz.$$
By integration by parts we obtain that the far right term equals
$$\frac{1}{p(x)p(y)}\int_\Gamma\psi(x,z)(\mathfrak w_z^* q(z)^{-2} \mathfrak w_z)\cdot\psi(y,z) + B(x,y)$$
where 
$$B(x,y) = \sum_{k=0}^{d_2} \sum_{j=0}^{k-1} b_j(y) (-1)^j \left[ \big( \partial^{k-1-j}_z \psi(x,z) \big) \, 
\partial_z^j  \left(\frac{b_k(z)}{p(x)q(z)}\wt \psi(y,z)\right)\right]_{z=p_1}^{p_2}.$$
Applying the fact that $\mathfrak w_z^*q(z)^{-2}\mathfrak w_z = p(y)^2$, the conclusion of the lemma follows immediately.
\end{proof}
\subsection{Examples arising from classical bispectral functions}
\begin{ex}
Let $s,t\in\bbr$, and consider the elementary bispectral function  $\psi_{\exp}(x,y)= e^{xy}$ with $\Gamma_1 = [-t,t]$ and $\Gamma_2 = [-s,s]i$.  We see
$$
K(x,y) = i\int_{-s}^{s} e^{ixz}e^{iyz} dz = \frac{2i}{x+y}\sin(s(x+y)).
$$
Thus the integral operator
$$T(f(x)) := \int_{-t}^t \frac{\sin(s(x+y))}{x+y}f(y)dy$$
commutes with some differential operator in $\bisls(\psi_{\exp})$.  Since $\dim\bisls^{2\ell,2m}(\psi_{\exp}) = (\ell+1)(m+1)+1$, 
Theorem \ref{exp theorem} implies that $T$ commutes with a nonconstant operator in $\bisls^{2,2}(\psi_{\exp})$.  
This space is described in Lemma \ref{exp lemma}.
Thus, $T$ commutes with an operator of the form
$$ c_0 + c_1x^2 + \partial_x(c_2x^2 + c_3)\partial_x$$
for some $c_0, c_1, c_2, c_3 \in \bbc$.
For this to be symmetric with respect to $\Gamma_1$, we must have $c_2t^2+c_3 = 0$.  The image of $\mathfrak d$ under $b_{\psi_{\exp}}$ is 
$$ c_0 + c_3y^2 + \partial_y(c_2y^2 + c_1)\partial_y.$$
For this to be symmetric with respect to $\Gamma_2$, we must have $-c_2s^2+c_1 = 0$.  Taking $c_2=1$ and $c_0=0$, gives that $T$ commutes with the differential operator
$$(x^2-t^2)\partial_x^2 + 2x\partial_x + s^2x^2.$$
\end{ex}

\begin{ex}
Let $\Ai(x)$ be the Airy function of the first kind and let $s,t\in\bbr$.  Consider the curves $\Gamma_2 = [s,\infty)$ and $\Gamma_1 = [t,\infty)$.  Then by (\ref{kernel equation}) we have
$$
K(x,y) = \frac{\Ai'(x+s)\Ai(y+s) - \Ai(x+s)\Ai'(y+s)}{x-y}.
$$
Therefore the integral operator 
$$T: f(x)\mapsto \int_t^\infty K(x,y)f(y)dy$$ 
commutes with a differential operator in $\bisls(\psi_{\Ai})$.  By Lemma \ref{airy lemma}, we have that $\dim(\bisls^{2\ell,2m}(\psi_{\Ai})) = (\ell+1)(m+1)$.
Theorem \ref{airy theorem} implies that $T$ commutes with a differential operator in $\bisls^{2,2}(\psi_{\Ai})$. 
 Explicitly, $T$ commutes with a differential operator of the form
$$c_0 + c_1x - c_3x^2 + \partial_x(c_2 + c_3x)\partial_x$$
for some $c_0, c_1, c_2, c_3 \in \bbc$.
For this to be symmetric with respect to $\Gamma_1$, we must have $c_2 +c_3t=0$.
The image of this operator under $b_{\psi_{\Ai}}$ is
$$c_0 - c_1y - c_3y^2 + \partial_x(c_1+c_2 + c_3y)\partial_y.$$
For this to be symmetric with respect to $\Gamma_2$, we must have $c_1+c_2+c_3s=0$.
Taking $c_0=0$ and $c_3=1$, this shows that $T$ commutes with the operator
$$(x-t)\partial_x^2 + \partial_x - (x^2 + (s-t)x).$$
\end{ex}
\begin{ex}
Consider the elementary bispectral Bessel function 
$$\psi_{\Be(\nu)}(x,y) = \sqrt{xy}K_{\nu+1/2}(xy),$$ 
along with the curves $\Gamma_2 = \{se^{i\theta}: 0\leq \theta\leq \pi\}$ and $\Gamma_1 = \{te^{i\theta}: 0\leq \theta\leq \pi\}$.  By \eqref{kernel equation} we have that
\begin{equation}
K(x,y) = \frac{\sqrt{xy}}{x^2-y^2}(xzK_{\nu+1/2}'(xz)K_{\nu+1/2}(yz)-yzK_{\nu+1/2}(xz)K_{\nu+1/2}'(yz))|_{z=-s}^s
\end{equation}
and for this kernel, the integral operator $T: f(x)\mapsto \int_{\Gamma_1} K(x,y)f(y)dy$ commutes with a differential operator in $\bisls({\psi_{\Be(\nu)}})$.  
To determine which one, we should try to find a bisymmetric operator with respect to $(\Gamma_1,\Gamma_2)$.  
Again $\dim\bisls^{2\ell,2m}({\psi_{\Be(\nu)}}) = (\ell+1)(m+1)$ and Theorem \ref{bessel theorem} implies that $T$ commutes with an operator in $\bisls^{2,2}({\psi_{\Be(\nu)}})$.  
By Lemma \ref{bessel lemma}, the operators in $\bisls^{2,2}({\psi_{\Be(\nu)}})$ are of the form
$$\left(c_0 + c_1x^2-\frac{c_3\nu(\nu+1)}{x^2}\right) + \partial_x(c_2x^2 + c_3)\partial_x$$
for some constants $c_0,c_1,c_2,c_3$. The image under $b_{\psi_{\Be(\nu)}}$ of this operator is
$$\left(c_0 + c_3y^2-\frac{c_1\nu(\nu+1)}{y^2}\right) + \partial_y(c_2y^2 + c_1)\partial_y$$
Thus for this to be bisymmetric with respect to $(\Gamma_1,\Gamma_2)$, we need $c_2t^2 + c_3=0$ and $c_2s^2 + c_1 = 0$.  Taking $c_0 = 0$ and $c_2 = 1$, this says that $T$ commutes with the second-order operator
$$\partial_x(x^2 - t^2)\partial_x-\left(s^2x^2+t^2\frac{\nu(\nu+1)}{x^2}\right).$$
\end{ex}
\subsection{Examples arising from bispectral Darboux transformations}
In this subsection, we present in increasing difficulty three examples of integral operators commuting with differential operators
whose minimal orders are 2, 6, and 22, respectively.
\begin{ex}
Consider the bispectral meromorphic function
$$\wt\psi(x,y) = e^{xy}\left(1-\frac{1}{xy}\right),$$
along with the curves $\Gamma_1 = \{te^{i\theta}: 0\leq\theta\leq \pi\}$ and $\Gamma_2 = \{se^{i\theta}: |\theta|\leq \pi/2\}$.  By \eqref{kernel equation}, we have that
$$
\wt K(x,y) = \frac{2i}{x+y}\sin((x+y)s) + \frac{2i}{xys}\cos((x+y)s).
$$
Note that
$$\wt\psi(x,y) = \frac{1}{xy}\mathfrak u\cdot \psi_{\exp}(x,y)$$
for $\mathfrak u = x\partial_x - 1$ with
$$\partial_x^2 = \left(\partial_x + \frac{1}{x}\right)\left(\partial_x - \frac{1}{x}\right)$$
so that $\wt\psi(x,y)$ is a self-adjoint bispectral Darboux transformation of $\psi_{\exp}(x,y)$ of order $(1,1)$.  Thus by 
Theorem \ref{exp theorem}, the integral operator
$$T: f(x)\mapsto \int_{\Gamma_1}\left( \frac{1}{x+y}\sin((x+y)s) + \frac{1}{xys}\cos((x+y)s)\right)f(y)dy$$
commutes with a differential operator in $\bisls^{2,2}(\wt\psi)$.  That is, $T$ commutes with a differential operator of the form
$$c_0 + c_1x^2 + \frac{-2c_3}{x^2} + \partial_x(c_2x^2 + c_3)\partial_x$$
(which are the operators obtained from the space in the LHS of \eqref{subspace}). The image of such an operator under 
$b_{\wt\psi}$ is
$$c_0 + c_3y^2 + \frac{-2c_1}{y^2} + \partial_y(c_2y^2 + c_1)\partial_y.$$
Therefore for bisymmetry with respect to $(\Gamma_1,\Gamma_2)$, we need $-c_2t^2+c_3=0$ and $c_2s^2+c_1=0$. 
Taking $c_2=1$ and $c_0=0$, we obtain that $T$ commutes with the operator
$$t^2x^2 + \frac{-2s^2}{x^2} + \partial_x(x^2 - s^2)\partial_x.$$
\end{ex}

\begin{ex}
In this example we consider self-adjoint bispectral Darboux transformations of the Bessel function $\psi_{\Be(\nu)}(x,y)$ of bidegree $(2,2)$ corresponding to subspaces of the kernel of $\mathfrak d_{\Be(\nu),x}^2$.
To do so, we consider rational factorizations of the square of the Bessel operator $\mathfrak d_{\Be(\nu),x}^2$
$$\left(\partial_x^2-\frac{\nu(\nu+1)}{x^2}\right)^2 = \wt{\mathfrak a}\mathfrak a,\ \text{where}\ \wt{\mathfrak a} = \left(\partial_x^2-\frac{a}{x}\partial_x + \frac{b}{x^2}\right),\ \mathfrak a = \left(\partial_x^2+\frac{a}{x}\partial_x + \frac{c}{x^2}\right)$$
Multiplying out the product, and solving the associated system of equations in three variables, we see that $(a,b,c)$ is either
$$(2\nu-1,(\nu-1)(\nu+1),\nu(\nu-2))\ \ \text{or}\ \ (-2\nu-3,\nu(\nu+2),(\nu+1)(\nu+3)).$$
Furthermore, since $b=a+c$ we see that $\wt{\mathfrak a}^*=\mathfrak a$ and $\mathfrak u := x^2\mathfrak a$ and $\mathfrak u^*$ are in $\bisl(\psi)$ with $b_{\psi_{\Be(\nu)}}(\mathfrak u_x) = \mathfrak u_y$ and $b_{\psi(\nu)}(\mathfrak u_x^*) = \mathfrak u_y^*$.
For sake of concreteness, we will take
$$\mathfrak u = x^2\partial_x^2 + (2\nu-1)x\partial_x + \nu(\nu-2).$$
Thus $\mathfrak u$ defines a symmetric bispectral Darboux transformation (for $K(xy) = K_{\nu+1/2}(xy)$)
$$\wt\psi(x,y) := \frac{1}{xy}\mathfrak u\cdot\psi_{\Be(\nu)}(x,y) = (\nu^2-\nu-\frac{3}{4})\frac{1}{(xy)^{3/2}}K(xy)+\frac{2\nu}{\sqrt{xy}}K'(xy)+\sqrt{xy}K''(xy).$$

Now consider paths $\Gamma_1,\Gamma_2\subseteq\bbc$ consisting of semicircles of radius $t$ and $s$, respectively, in the upper half plane centered at the origin and oriented counter-clockwise.
Theorem \ref{bessel theorem} implies that the integral operator 
\begin{equation}
\label{last oper}
T: f(x)\mapsto \int_{-t}^t \wt K(x,y)f(y)dy \quad 
\mbox{with kernel} \quad \wt K(x,y) = \int_{-s}^s \wt\psi(x,z)\wt\psi(y,z)dz
\end{equation}
commutes with a differential operator in $\bisls^{8,8}(\wt\psi)$.
Furthermore, this operator may be taken to be an element in the subspace of dimension $21$ consisting of operators of the form
$$\wt{\mathfrak d} := \frac{1}{x^2}\mathfrak u\mathfrak d_1\mathfrak u^*\frac{1}{x^2} + x^2\mathfrak d_2 x^2,$$
for $\mathfrak d_1\in \bisls^{4,8}(\psi_{\Be(\nu)})$ and $\mathfrak d_2\in\bisls^{4,2}(\psi_{\Be(\nu)})$.
Using our explicit expression for $\bisls^{m,n}(\psi_{\Be(\nu)})$ we may write (for $\mathfrak s_x = x\partial_x$)
\begin{align*}
\wt{\mathfrak d}
  & = x^2\left( \sum_{k=0}^2 \sum_{j=0}^{4-k}a_{jk}\mathfrak s_x^k\mathfrak d_{\Be(\nu)}^j\mathfrak s_x^k + \sum_{k=0}^4\sum_{j=1}^2b_{jk}\mathfrak s_x^k x^{2j}\mathfrak s_x^k \right)x^2\\
  & + \frac{1}{x^2} {\mathfrak u}
  \left( \sum_{k=0}^1 \sum_{j=0}^{2-k}c_{jk}\mathfrak s_x^k\mathfrak d_{\Be(\nu)}^j\mathfrak s_x^k + \sum_{k=0}^2\sum_{j=1}^{1-k}d_{jk}\mathfrak s_x^k x^{2j}\mathfrak s_x^k \right)
  {\mathfrak u}^* \frac{1}{x^2}
\end{align*}

We use a computer to determine the values of the constants $a_{ij},b_{ij},c_{ij},d_{ij}$ that make it bisymmetric with respect to a pair of curves $(\Gamma_1,\Gamma_2)$.
For simplicity, we take $s=t=1$.
We reduce the problem of checking adjointability to a linear algebra problem on a finite dimensional vector space
by calculating the value of the bilinear concomitant on each element of a basis of the space of linear operators above for pairs of polynomials of degree up to one more than the degree of the operator.
This gives us a linear system of $242$ linear equations which we compute with python using a code written with the SymPy symbolic computation library \cite{10.7717/peerj-cs.103}.
The solution of the system of equations is then obtained using FLINT \cite{flint} for precise integer arithmetic.
We found that the space above has a two dimensional subspace of operators commuting with $T$, with basis given by
\begin{center}
\begin{tabular}{c|c|c}
& $a_{00}$ & $a_{10} $  \\\hline
solution 1 & $6\nu^4+60\nu^3+206\nu^2+288\nu+136$ & $-4\nu^4-56\nu^3-256\nu^2-444\nu-240$ \\\hline 
solution 2 & $3\nu^2+9\nu+6$ & $-3\nu^2-15\nu-9$ \\\hline
\end{tabular}
\begin{tabular}{c|c|c|c|c}
& $a_{20} $ & $a_{30} $ & $a_{40}$ & $a_{01}$ \\\hline
solution 1 & $\nu^4+18\nu^3+95\nu^2+150\nu+78$ & $4\nu^2+36\nu+20$ & $6$ & $12\nu^2+60\nu+80$ \\\hline
solution 2 &  $\nu^2+7\nu$ & $3$ & $0$ & $3$\\\hline
\end{tabular}
\begin{tabular}{c|c|c|c|c|c|c}
& $a_{11}$ & $a_{21} $ & $a_{31} $ & $a_{02} $ & $a_{12}$ & $a_{22}$ \\\hline
solution 1 & $-8\nu^2-56\nu-76$ & $2\nu^2+18\nu+18$ & $4$ & $6$ & $-4$ & $1$ \\\hline
solution 2 & $-3$ & $1$ & $0$ & $0$ & $0$ & $0$ \\\hline
\end{tabular}
\begin{tabular}{c|c|c|c}
& $b_{10}$ & $b_{20} $ & $b_{11} $   \\\hline
solution 1 & $4\nu^2+20\nu+24$ & $1$ & $4$ \\\hline
solution 2 & $1$ & $0$ & $0$ \\\hline
\end{tabular}
\begin{tabular}{c|c|c|c|c|c|c}
& $c_{00}$ & $c_{10} $ & $c_{20} $ & $c_{01} $ & $c_{11}$ & $d_{10}$ \\\hline
solution 1 & $-4\nu^2-12\nu-8$ & $4\nu^2+20\nu+16$ & $1$ & $-4$ & $4$ & $0$ \\\hline
solution 2 & $3\nu^2+9\nu+6$ & $1$ & $0$ & $3$ & $0$ & $3$\\\hline
\end{tabular}
\end{center}
\medskip
Each of the above defines a differential operator commuting with the integral operator $T$.


\end{ex}

\begin{ex}
Let $a,b\in\bbc$ be fixed.
Recall the Airy operator $\mathfrak d_{\Ai,x} = \partial_x^2 -x$.
Consider the bispectral meromorphic function $\wt\psi(x,y)$ defined by
$$\wt\psi(x,y) = \frac{1}{p(x)q(y)}\mathfrak u\cdot\psi_{\Ai}(x,y)$$
in the notation of Figure \ref{elementary bispectral functions} for $p(x) = (1-a^2(x+b^2))^2$, $q(y)=(y-b)^2$, and $\mathfrak u = \mathfrak w^2-a^2$ with
$$\mathfrak w = (1-a^2b^2-a^2x)\partial_x^2 + a^2\partial_x + a^2b^4 + 2a^2b^2x + a^2x^2 - b^2 - x.$$
Note that
$$b_{\Ai}(\mathfrak w) = a^2(b^2 - y)\partial_y^2 + a^2\partial_y + a^2b^4 - 2a^2b^2y + a^2y^2 - b^2 + y$$
and also that  $b_{\Ai}(\mathfrak u^*) = b_{\Ai}(\mathfrak u)^*$ because $\iota_{\Ai,x}: \bisl(\psi_{\Ai})\rightarrow\bisl(\psi_{\Ai})$ is the identity map.
Finally, a direct calculation shows
$$\mathfrak u^*\frac{1}{p(x)^2}\mathfrak u = (\mathfrak d_{\Ai,x}-b)^4.$$
Thus $\wt\psi(x,y)$ is a self-adjoint bispectral Darboux transformation of $\psi(x,y)$ of bidegree $(4,4)$.

From Theorem \ref{airy theorem}, we know that the integral operator with kernel $\wt K(x,y)$ defined by
$$T: f(x)\mapsto \int_t^\infty \wt K(x,y)f(y)dy,\ \ \text{with}\ \ \wt K(x,y) = \int_s^\infty \wt\psi(x,z)\wt\psi(y,z)dz$$
will commute with a nonconstant, bisymmetric ordinary differential operator of order at most $32$ and co-order at most $32$.
Here $s,t$ are chosen such that $b\notin [s,\infty)$ and $a^{-2}-b^2\notin [t,\infty) $.
Moreover, we know that this differential operator will be of the form $\frac{1}{p(x)}\mathfrak u\mathfrak d\mathfrak u^*\frac{1}{p(x)}$ for some formally bisymmetric differential $\mathfrak d\in\bisls^{24,24}(\psi)$.
Using the explicit basis for $\bisls^{\ell,m}(\wt\psi)$ defined above, we can write down a generic element $\wt{\mathfrak d}$ of the subspace of Theorem \ref{control thm} of 
$\bisls^{32,32}(\wt\psi)$ with $273$ free variables $c_{jk}$ and $\wt c_{jk}$:
\begin{align*}
\wt{\mathfrak d}
   := \frac{1}{p(x)}\mathfrak u\mathfrak d\mathfrak u^*\frac{1}{p(x)}  & = \frac{1}{p(x)}\mathfrak u\left(\sum_{j=0}^{12}\sum_{k=0}^{16} c_{jk} (\mathfrak d_{\Ai,x}^jx^k + x^k\mathfrak d_{\Ai,x}^j)\right)\mathfrak u^*\frac{1}{p(x)}\\
  & + p(x)\left(\sum_{j=0}^{3}\sum_{k=0}^{12} \wt c_{jk} (\mathfrak d_{\Ai,x}^jx^k + x^k\mathfrak d_{\Ai,x}^j)\right)p(x)
\end{align*}
We evaluate the associated bilinear concomitant on each pair of polynomials with degree up to $24$ in order to obtain a homogeneous linear system of equations in the unknowns $c_{jk},\wt c_{jk}$.

We solve this numerically using exact integer calculations with the specific values $a=1$, $b=0$ and $s=t=2$.
As before, the solution amounts to solving a certain homogeneous linear system of equations associated to the bilinear concomitant.
However, due to the order of the operators involved, the coefficients in the linear system naturally exceed the accuracy of $8$-byte long integer or double variables.
For this reason, we must rely on exact integer arithmetic, even when setting up the linear system that we must solve.
By carefully writing our python code with the SymPy computational library \cite{10.7717/peerj-cs.103}, we are able to obtain this linear system with exact integer values.
The result is passed to a C code written with the FLINT \cite{flint}, we obtain the kernel of the linear system exactly.
We find that the subspace of differential operators of order and co-order at most $32$ that the integral operator $T$ commutes with  has dimension at least $35$! An example of an operator in this space is the following differential operator of order $22$ and co-order $14$
\begin{align*}
\wt{\mathfrak d} &= \frac{1}{p(x)}\mathfrak u(35\mathfrak d_{\Ai,x}^4 + 84\mathfrak d_{\Ai,x}^5 + 70\mathfrak d_{\Ai,x}^6 + 20\mathfrak d_{\Ai,x}^7) {\mathfrak u}^* \frac{1}{p(x)} 
\\
&+ p(x)(1-4x+10x^2-20x^3)p(x).
\end{align*}
Unlike the previous example, the expanded form of the differential operator is too long to be presented here.
\end{ex}
\section{Classification of self-adjoint bispectral meromorphic functions}
\label{classif}
In this section we describe a classification of the self-adjoint bispectral meromorphic functions that appear in 
Theorems \ref{exp theorem}, \ref{airy theorem}, and \ref{bessel theorem}. This classification 
is given in terms of (infinite-dimensional) lagrangian versions of Wilson's adelic Grassmannian \cite{W}. 
We use the construction of the latter in terms of bispectral Darboux transformations as in \cite{BHYcmp}.

\subsection{The adelic Grassmannian}
\label{7.1}
By Theorem \ref{rank1}, the class of rank $1$  bispectral meromorphic functions is precisely the class of 
bispectral Darboux transformations of the exponential function
$\psi_{\exp}(x,y) = e^{xy}$, normalized as in Remark \ref{normal-rem}. A function $\wt\psi(x,y)$ in this class has the form
\begin{equation}
\label{polyn-form}
\wt\psi(x,y) = \frac{h(x,y)}{p(x)q(y)}e^{xy}
\end{equation}
for some polynomials $p(x)\in\bbc[x],q(y)\in\bbc[y]$
and a polynomial $h(x,y)\in\bbc[x,y]$. (Note however that not all functions of the form \eqref{polyn-form} 
are bispectral Darboux transformations of $\psi_{\exp}(x,y) = e^{xy}$; 
the polynomials $p(x), q(y), h(x,y)$ need to satisfy some conditions.)
We can recover the bispectral transformation data from this form.
Writing $h(x,y) = \sum_{i=0}^m\sum_{j=0}^n a_{ij}x^iy^j$, we define the operators
$$\mathfrak v_L = \sum_{i=0}^m\sum_{j=0}^n a_{ij}x^i\partial_x^j \in \bisls(\psi_{\exp})
\ \ \text{and}\ \ \mathfrak v_R = \sum_{i=0}^m\sum_{j=0}^n a_{ij}y^j\partial_y^i \in \bisrs(\psi_{\exp}).$$
Then $\mathfrak v_R$ is the Fourier transform $b_{\psi_{\exp}}$
of $\mathfrak v_L$ and by virtue of their definition,
$$
\wt\psi(x,y) = \frac{1}{p(x)q(y)}\mathfrak v_L\cdot \psi_{\exp}(x,y) \ \ \text{and} \ \ \wt\psi(x,y) = \frac{1}{p(x)q(y)}\mathfrak v_R\cdot \psi_{\exp}(x,y).
$$

Let $\mathscr C$ be the span of all linear functionals of the form $\delta^{(k)}(y-a)$ for $k \in \bbn$ and $a\in\bbc$, defined on smooth functions by
$$\langle f(y),\delta^{(k)}(y-a)\rangle = f^{(k)}(a).$$
Note that $\mathscr C$ comes with a natural left action of $\bbc[y]$, defined by
$$f(y)\cdot \delta^{(k)}(y-a) = \sum_{j=0}^k \frac{k!}{j!} \delta^{(j)}(y-a)f^{(k-j)}(a).$$
For a given $\wt\psi(x,y)$, we define the set of linear functionals
$$\mathscr C_L(\wt\psi) = \{\chi\in\mathscr C: \langle e^{xy}h(x,y),\chi(y)\rangle = 0\}.$$
The vector space $\mathscr C_L(\wt\psi)$ is finite dimensional and naturally isomorphic to the kernel of the differential operator $\mathfrak v_L$.
Specifically, for any $\chi\in\mathscr C_L(\wt\psi)$ we have
$$0 = \langle e^{xy}h(x,y),\chi(y)\rangle = \langle \mathfrak v_L\cdot e^{xy},\chi(y)\rangle = \mathfrak v_L\cdot\langle e^{xy},\chi(y)\rangle,$$
so that $\chi\mapsto \langle e^{xy},\chi(y)\rangle$ defines a linear map of $\mathscr C_L(\wt\psi)$ into $\ker(\mathfrak v_L)$.
With a reverse argument one shows that this is an isomorphism.

Conjugation by $\frac{1}{p(x)}\mathfrak v_L$ sends $\bislf(\wt\psi)$ into $\bbc[\partial_x]$.
In fact
$$\bislf(\wt\psi) = \left\lbrace\frac{1}{p(x)}\mathfrak v_Lf(\partial_x)\mathfrak v_L^{-1}p(x): f(\partial_x)\in\bbc[\partial_x],\ \text{and $\mathfrak v_Lf(\partial_x)\mathfrak v_L^{-1}$ 
is a diff. operator}\right\rbrace.$$
From the study of kernels of differential operators, we know that $\mathfrak v_Lf(\partial_x)\mathfrak v_L^{-1}$ is a differential operator if and only if $f(\partial_x)\cdot\ker(\mathfrak v_L)\subseteq\ker(\mathfrak v_L)$.
Each element of the kernel of $\mathfrak v_L$ is of the form $\langle e^{xy},\chi(y)\rangle$ for some $\chi\in \mathscr C_L(\wt\psi)$.
Therefore
$$f(\partial_x)\cdot\langle e^{xy},\chi(y)\rangle = \langle e^{xy},f(y)\cdot\chi(y)\rangle.$$
Thus, if $f(\partial_x)$ preserves the kernel, then 
$$ \langle e^{xy},f(y)\cdot\chi(y)\rangle = \sum_{\lambda\in\mathscr C_L(\wt\psi)} \langle e^{xy},c_\lambda \lambda\rangle,$$
which in turn implies that
$$f(y)\cdot\chi(y) = \sum_{\lambda\in\mathscr C_L(\wt\psi)}c_\lambda\lambda.$$
Thus,
$$\bislf(\wt\psi) = \left\lbrace\frac{1}{p(x)}\mathfrak v_Lf(\partial_x)\mathfrak v_L^{-1}p(x): f(\partial_x)\in\bbc[\partial_x],\ f(y)\cdot \mathscr C_L(\wt\psi)\subseteq\mathscr C_L(\wt\psi)\right\rbrace.$$

For a finite dimensional subspace $C$ of $\mathscr C$, define
$$V_C = \{f(y): \langle f(y),\chi(y)\rangle = 0, \ \forall \chi\in C\}.$$
By a direct argument, one shows that $f(y)\cdot C\subseteq C$ if and only if $f(y)V_C\subseteq V_C$.
This gives the following characterization of $\bislf(\wt\psi)$:
$$\bislf(\wt\psi) = \left\lbrace\frac{1}{p(x)}\mathfrak v_Lf(\partial_x)\mathfrak v_L^{-1}p(x): f(\partial_x)\in\bbc[\partial_x],\ f(y)V_{\mathscr C_L(\wt\psi)}\subseteq V_{\mathscr C_L(\wt\psi)}\right\rbrace.$$
It motivates the following definition of the rational Grassmannian $\Grrat$ of Wilson \cite{W}.
\begin{defn}
We define the {\bf{rational Grassmannian}} $\Grrat$ to be the set of all subspaces $W$ of $\bbc(y)$ of the form
$$W = \frac{1}{q(y)}V_C$$
for some $C\subseteq \mathscr C$ and $q(y)\in\bbc[y]$ with $\dim(C) = \deg(q(y))$.
The subspace $C\subseteq \mathscr C$ is called the {\bf{space of conditions}} of $W$.
\end{defn}
To each point $W\in\Grrat$, we can associate an algebra
$$A_W = \{f(y) \in \bbc [y] : f(y)W\subseteq W\}.$$
The pair $(W,A_W)$ is called a Schur pair.
From the construction above, we have shown that every bispectral meromorphic function $\wt\psi(x,y)$ of rank $1$ corresponds 
to a point $W$ in $\Grrat$ whose associated algebra $A_W$ is isomorphic to $\bislf(\wt\psi)$.

Not every point of $\Grrat$ corresponds to a bispectral meromorphic function.
In order for this to be true, the space of conditions $C$ of $W$ must be homogeneous.
\begin{defn}
For any point $c\in\bbc$, let $\mathscr C_c$ denote the subspace of $\mathscr C$ spanned by linear functionals of the form $\delta^{(k)}(x-c)$ for $k\geq 0$ an integer.
Put another way, $\mathscr C_c$ is the subspace of linear functionals of $\mathscr C$ supported at the point $c$.
A linear functional $\chi\in\mathscr C$ is called homogeneous if $\chi\in\mathscr C_c$ for some value $c\in\bbc$.
A finite dimensional subspace $C$ of $\mathscr C$ is called {\bf{homogeneous}} if it is spanned by homogeneous elements.
Equivalently, $C$ is homogeneous if
$$C = \bigoplus_{c\in\bbc} C\cap \mathscr C_c.$$
\end{defn}

It was shown by Wilson that bispectral meromorphic functions give rise to points $W\in\Grrat$ with $\mathscr C_L(\wt\psi)$ homogeneous.
He also proved that the converse is true: a point $W = \frac{1}{q(y)}V_C$ in $\Grrat$ with $C$ homogeneous gives rise to a bispectral meromorphic function $\wt\psi(x,y)$ with $\mathscr C_L(\wt\psi) = C$.
\begin{defn}
\label{ad-Gr}
We define the {\bf{adelic Grassmannian}} $\Grad$ to be the subset of $\Grrat$ consisting of the subspaces $W$ of $\bbc(y)$ of the form
$$W = \frac{1}{q(y)}V_C,$$
where $C\subseteq\mathscr C$ is homogeneous and
\begin{equation}
\label{q}
q(y) = \prod_{c\in\bbc} (y-c)^{n(c)}, \ \ n(c) = \dim(C\cap \mathscr C_c).
\end{equation}
\end{defn}
The above correspondence between the points of $\Grad$ and the normalized (as in Remark \ref{normal-rem})
bispectral Darboux transformations $\wt\psi(x,y)$ of $\psi_{\exp}(x,y)$ is a bijection.

\subsection{A classification in rank 1 in terms of fixed points of involutions} 
\label{7.2}
In \cite{W}
Wilson defined two involutions of the adelic Grassmannian $\Grad$: the adjoint and sign involutions. 

The {\bf{adjoint involution}} on $\Grad$ sends a point $W \in \Grad$ to the point $aW\in\Grad$, 
given by
$$aW = \left\lbrace f(y): \frac{1}{2\pi i}\oint_{|y|=1}f(y)g(-y)dz = 0, \ \forall g(y)\in W \right\rbrace,$$
see \cite[Sect. 7]{W} for details. 
Letting $\wt\psi(x,y)$ and $a\wt\psi(x,y)$ be the associated bispectral meromorphic functions, the bispectral algebras of $\wt\psi$ and $a\wt\psi$ 
are related by the formal adjoint
\begin{equation}
\label{Ba}
\bislf(a\wt\psi) = \{\mathfrak d^*: \mathfrak d\in\bislf(\wt\psi)\}, \ \ 
\bisrf(a\wt\psi) = \{\mathfrak d^*: \mathfrak d\in\bisrf(\wt\psi)\},
\end{equation}
see \cite[Corollary 7.7]{W}. If $\wt\psi(x,y)$ is given as a bispectral Darboux transformation of $\psi_{\exp}(x,y)$ as
\begin{equation}
\label{transf}
\wt\psi(x,y) = \frac{1}{q(y)p(x)}\mathfrak{u} \cdot\psi_{\exp}(x,y) \ \
\text{and} \ \
\psi_{\exp}(x,y) = \wt{\mathfrak{u}} \cdot \frac{1}{\wt{q}(y)\wt{p}(x)} \wt \psi(x,y),
\end{equation}
then $a \wt\psi(x,y)$ is given by
\begin{equation}
\label{apsi}
a \wt\psi(x,y) = \frac{1}{\wt{q}(-y)\wt{p}(x)} \wt{\mathfrak{u}}^* \cdot \psi_{\exp}(x,y) \ \
\text{and} \ \
\psi_{\exp}(x,y) = \mathfrak{u}^* \frac{1}{q(-y)p(x)} \cdot  a \wt \psi(x,y),
\end{equation}
see \cite[Proposition 1.7(i)]{BHYcmp}.

The {\bf{sign involution}} $s$ of $\Grad$, defined in \cite[Sect. 8]{W}, is given by
\[
W \mapsto s W := \left\lbrace f(-y): f(y) \in W\right\rbrace.
\]
On the level of bispectral functions, the sign involution is given by
\begin{equation}
\label{sign}
s \wt\psi (x,y):= \wt\psi (-x, -y).
\end{equation}
The associated bispectral algebras are related by 
\begin{equation}
\label{Bs}
\bislf(s \wt\psi) = \sigma \bislf(\wt\psi), 
\end{equation}
where $\sigma$ is the automorhism of the algebra of differential operators with rational coefficients
\begin{equation}
\label{sigma}
\sigma(x) = - x, \sigma(\partial_x) = - \partial_x.
\end{equation}
\begin{thm}
\label{fixed}
Let $\wt\psi(x,y)$ be a bispectral meromorphic function of rank $1$.
Then the following are equivalent:
\begin{enumerate}
\item  $\wt\psi(x,y)$ is self-adjoint.
\item  The algebra $\bislf(\wt\psi)$ is closed under the formal adjoint $*$ and the transformation $\sigma$.
\item The plane $W \in \Grad$ corresponding to $\wt\psi(x,y)$ is invariant under the adjoint $a$ and sign $s$ involutions of $\Grad$.
\end{enumerate}
\end{thm}
\begin{proof}
$(1) \Rightarrow (3)$ Let $\wt \psi(x,y)$ be a self-adjoint bispectral Darboux transformation of $\psi_{\exp}(x,y)$ with the notation of 
Definition \ref{bisp-Darb} and the normalization of Remark \ref{normal-rem}. By \eqref {apsi}, 
\[
a \wt\psi(x,y) = \frac{1}{\wt{q}(-y)\wt{p}(x)} \wt{\mathfrak{u}}^* \cdot \psi(x,y) = 
\frac{1}{q(y)\wt{p}(x)} \mathfrak{u} \cdot \psi(x,y) = \wt\psi(x,y).
\]
Since
\[
\iota_{\psi_{\exp}, x} = \sigma|_{\bisl(\psi_{\exp})}
\]
(see \S \ref{subFour}), 
it follows from \eqref{sign} that
\[
s \wt\psi(x,y) = \frac{1}{q(-y)p(-x)}(\sigma(\mathfrak{u})) \cdot\psi(x,y) = \frac{1}{q(y)p(x)}\mathfrak{u} \cdot\psi(x,y) = \wt\psi(x,y).
\]
Here the equalities $q(-y) = q(y)$ and $p(-x) = p(x)$ follow from the facts that $q(y)$ and $p(x)$ are the leading coefficients 
of the differential operators $\mathfrak u$ and $b_{\psi_{\exp}}( \mathfrak u )$ (see Remark \ref{normal-rem}) and the 
facts that these operators are fixed under $\iota_{\psi_{\exp}, x}$ and $\iota_{\psi_{\exp}, y}$, respectively.

$(3) \Rightarrow (1)$  If the plane $W \in \Grad$ corresponding to $\wt\psi(x,y)$ 
is invariant under the adjoint and sign involutions of $\Grad$, then 
it follows from \eqref{q} that $q(-y) = q(y)$. Applying Wilson's bispectral involution $b$ of $\Grad$ \cite[Sect. 8]{W} to $\wt\psi$, 
and using that $bs = sb$ and $ab = bsa$ \cite[Sect. 8]{W},
gives that $p(-x) = p(x)$.  
The rest of this implication is proved with a reverse argument to the one used for the previous implication.

The equivalence $(2) \Leftrightarrow (3)$ follows from \eqref{Ba} and \eqref{Bs}.
\end{proof}

\subsection{The Lagrangian adelic Grassmannian}
\label{7.3}
In this subsection we provide a geometric classification of the self-adjoint bispectralfunctions of rank 1.
Specifically, we relate the points of $\Grad$ fixed by the adjoint $a$ and sign $s$ involutions
to certain Lagrangian subspaces of kernels of formally symmetric differential operators with constant coefficients.
Recall the definition of the bilinear concomitant $\mathcal{C}_\mathfrak d$ of a differential operator $\mathfrak d$ (Definition \ref{conco}).

First we prove a general classification result for the factorizations of a formally symmetric differential operator $\mathfrak d$ 
(i.e., $\mathfrak d= \mathfrak d^*$) with analytic coefficients into a 
product $\mathfrak u^* \mathfrak u$ for differential operators $\mathfrak u$ with analytic coefficients.

Start with a (not necessarily formally symmetric) differential operator $\mathfrak d$ with analytic coefficients on an open subset of $\bbc$.
Choose a sufficiently small connected open subset $\mathcal{O}$ of $\bbc$ such that the kernels $\ker \mathfrak d$ and $\ker \mathfrak d^*$ in 
the space of holomorphic functions on $\mathcal{O}$ 
have dimension equal to the order of $\mathfrak d$.
Note from \eqref{integral biconcomitant} that if $f(x)\in\ker(\mathfrak d)$ and $g(x)\in\ker(\mathfrak d^*)$, then for all pairs of points $x_0,x_1$ we have 
$\mathcal{C}_{\mathfrak d}(f,g;x_1)=\mathcal{C}_{\mathfrak d}(f,g;x_0)$.
Thus the bilinear concomitant $\mathcal{C}_{\mathfrak d}(\cdot,\cdot;p)$ restricts to a pairing
$$\mathcal{C}_{\mathfrak d}(\cdot,\cdot): \ker(\mathfrak d)\times\ker(\mathfrak d^*)\rightarrow\bbc$$
which is independent of the value of $p$, where $m = \ord \mathfrak d$.
Let $f_1(x),\dots, f_m(x)$ be a basis for the kernel of $\mathfrak d$ and $g_1(x),\dots, g_m(z)$ be a basis for the kernel of $\mathfrak d^*$. Consider the matrix
$$M_{\mathfrak d} := \text{Wr}[f_1(x),\dots, f_m(x)]^T C_{\mathfrak d}(x)\text{Wr}[g_1(x),\dots, g_m(x)]$$
where $\text{Wr}[\cdot]$ denotes the Wronskian matrix and $C_{\mathfrak d}(x)$ is the matrix defined in \eqref{C-matr}.
Then $M_{\mathfrak d}$ is a constant matrix representing the pairing with respect to the chosen bases.
If $M_{\mathfrak d}$ is singular, then there exists a constant vector $\vec c$ with $\vec c^T M_{\mathfrak d}= 0$.
Note that since the Wronskian is nonsingular, the matrix $M_{\mathfrak d}$ is singular if and only if there exists an element 
$f(x)\in\ker(\mathfrak d)$ with $\mathcal{C}_{\mathfrak d}(f,g) = 0$ for all $g$.  
Using the integral formula in \eqref{integral biconcomitant}, this would in fact imply  $\int_{x_0}^{x_1} f(x)(\mathfrak d^*\cdot g(x))dx$ for all $g(x)$.
That is, $f(x)$ would have to lie orthogonal to the closure of the image of $\mathfrak d^*$ on the Hilbert space of square integrable function of a suitably chosen path from $x_0$ to $x_1$.
However, the image of $\mathfrak d^*$ will be dense, so this in turn implies that $f(x)$ is zero.
Hence $\mathcal{C}_{\mathfrak d}( \cdot, \cdot)$ is nonsingular and defines a non-degenerate pairing.

The bilinear concomitant of $\mathfrak d^*$ is related to that of $\mathfrak d$ by
$$\mathcal{B}_{\mathfrak d^*}(f,g;p) = - \mathcal{B}_{\mathfrak d}(g,f;p).$$
This in particular, when $\mathfrak d$ is formally symmetric the bilinear concomitant defines a symplectic form on $\ker \mathfrak d$.

Now suppose that $\mathfrak d=\mathfrak d^*$.
Recall that factorizations of $\mathfrak d$ correspond to choices of subspaces of $\ker \mathfrak d$.
Specifically given a subspace $V\subseteq \ker \mathfrak d$, there exists a differential operator $\mathfrak b$ 
(unique up to multiplication by a function on the left) with $\ker \mathfrak b = V$ and $\mathfrak d= \mathfrak a\mathfrak b$ for some differential operator $\mathfrak a$.
Since $\mathfrak d = \mathfrak d^*$, this also implies that $\mathfrak d = \mathfrak b^*\mathfrak a^*$ so that $\mathfrak a^*$ also corresponds to 
a certain subspace of $\ker \mathfrak d$, which turns out to be exactly the orthogonal subspace $V^\perp$ of $V$ under the symplectic form defined by the bilinear concomitant.
In particular, in the special case of a factorization of the form $\mathfrak d= \mathfrak b^*\mathfrak b$, the subspace $V$ of $\mathfrak b$ satisfies $V^\perp = V$, 
i.e., is a Lagrangian subspace of $\ker(\mathfrak d)$.
This is the content of the next theorem.
\begin{thm}
\label{Lagr}
Let $\mathfrak d$ be a monic formally symmetric operator on an open subset $\mathcal{O}$ of $\bbc$. Choose $\mathcal{O}$ sufficiently small so that the kernel 
$\ker \mathfrak d$ in the space of holomorphic functions on $\mathcal{O}$ has dimension equal to the order of $\mathfrak d$. Let $a,b\in \mathcal{O}$ with $a\neq b$.

If $\mathfrak d$ has a factorization $\mathfrak d= \mathfrak a\mathfrak b$ into differential operators $\mathfrak a, \mathfrak b$ with analytic coefficients on $\mathcal{O}$,
then $\ker \mathfrak b $ and $\ker \mathfrak a^*$ are complementary subspaces of 
$\ker \mathfrak d $ relative to the symplectic form $\mathcal{C}_{\mathfrak d}$.
Furthermore, given a subspace $V\subseteq \ker \mathfrak d $ there exist monic differential operators 
$\mathfrak a,\mathfrak b$ satisfying $\ker \mathfrak a^*  = V^\perp$ and $\ker \mathfrak b = V$.

In particular, factorizations of the form $\mathfrak d = \mathfrak b^*\mathfrak b$ for monic  differential operators $\mathfrak b$ with analytic coefficients on $\mathcal{O}$
are in bijective correspondence with Lagrangian subspaces of $\ker \mathfrak d$.
\end{thm}
\begin{proof}
Suppose that $\mathfrak d$ is an operator of order $m$ and has a factorization of the form $\mathfrak d= \mathfrak a\mathfrak b$.
If $\ord \mathfrak b = r$, then $\ord \mathfrak a = m-r$.
Since $\mathfrak d= \mathfrak d^*$, we have $\mathfrak d^* = \mathfrak b^*\mathfrak a^*$.
Therefore $\ker \mathfrak a^*$ and $\ker \mathfrak b $ are subspaces of $\ker \mathfrak d $ of dimension $m-r$ and $r$, respectively.
Following Wilson in \cite{Wilson-n}, for all $f,g$ we have that
\begin{equation}\label{biconcomitant decomposition}
\mathcal{C}_{\mathfrak d}(f,g;p) = \mathcal{C}_{\mathfrak a}(\mathfrak b\cdot f,g;p) + \mathcal{C}_{\mathfrak b}(f,\mathfrak a^*\cdot g;p).
\end{equation}
This implies that for all $f\in\ker \mathfrak b$ and $g\in \ker \mathfrak a^*$ we must have $\mathcal{C}_{\mathfrak d}(f,g;p) = 0$.
Hence $\ker \mathfrak a^* \subseteq (\ker \mathfrak b)^\perp$.
Since the dimension of $( \ker \mathfrak b)^\perp$ must be $m-r$, this implies that $\ker \mathfrak a^*  = (\ker \mathfrak b)^\perp$.
This proves the first claim of the theorem.

To prove the second claim, we start with an arbitrary subspace $V\subseteq \ker \mathfrak d$.
Then the usual construction may be used to produce a differential operator $\mathfrak b$ whose kernel is $V$.
Then since the kernel of $\mathfrak b$ is contained in the kernel of $\mathfrak d$, there must exist a differential operator $\mathfrak a$ with $\mathfrak d= \mathfrak a\mathfrak b$.
By the previous argument then $\ker \mathfrak a^* = V^\perp$.
This proves the second claim of the theorem.

In the special case that $\mathfrak d=\mathfrak b^*\mathfrak b$, we must have that $\ker \mathfrak b = (\ker \mathfrak b)^\perp$.
Thus in this case, the kernel of $\mathfrak b$ defines a Lagrangian subspace of $\ker \mathfrak d$.
Conversely, a lagrangian subspace $V\subseteq\ker(\mathfrak d)$ defines a factorization $\mathfrak d = \mathfrak a\mathfrak b$ where $\ker \mathfrak a^* = \ker \mathfrak b$.
This implies that $\mathfrak a^* = h(x)\mathfrak b$ for some function $h(x)$, and therefore that $\mathfrak d= \mathfrak b^*h(x)\mathfrak b$.
Setting $\mathfrak c= \pm \sqrt{h(x)}\mathfrak b$, this gives us a factorization $\mathfrak d=\mathfrak c^*\mathfrak c$ for a monic differential operator $\mathfrak c$.
\end{proof}

By the discussion in \S \ref{7.1}, the bispectral functions in the adelic Grassmannian are obtained as follows. Start with a differential 
operator with constant coefficients
\[
\mathfrak d = \prod_j ( \partial_x - c_j)^{m_j}.
\]
Its kernel in the space of entire functions consists of the quasipolynomials 
\[
\ker \mathfrak d = \big\{ \sum_j p_j(x) e^{ c_j x} \mid \deg p_j(x) \leq n_j -1 \big\}.
\]
Let $V$ be a subspace of $\ker \mathfrak d$ having a basis consisting of quasiexponential functions each of which contains a
single exponent $e^{c_j x}$; this is the adelic condition.
Let $\frac{1}{p(x)}  \mathfrak{u}$ be the unique monic differential operator such that
\[
\ker \mathfrak{u} = V \ \ \mbox{and} \ \
\mathfrak{u} \in \bisl(\psi_{\exp}) \ \mbox{(i.e., $\mathfrak{u}$ has polynomial coefficients).}
\]
Let 
\[
q(y) = \prod ( y - c_j)^{n_j},
\]
where $n_j$ equals the number of basis elements of $V$ whose exponent is $e^{c_j x}$. 
The bispectral functions in the adelic Grassmannnian are the functions of the form
\begin{equation}
\label{psiwt}
\wt\psi(x,y) = \frac{1}{q(y) p(x)} \mathfrak u \cdot e^{xy}.
\end{equation}
The corresponding $W \in \Grad$ (recall Definition \ref{ad-Gr}) is obtained as follows: it corresponds to the unique space of conditions $C\subseteq\mathscr C$ such that
$$
V = C (e^{xy}),
$$
keeping in mind  that the delta function in $C$ act in the variable $y$.

With this in mind, we define the Lagrangian adelic Grassmannian.
\begin{defn} The {\bf{Lagrangian adelic Grassmannian}} is the sub-Grassmannian $\Grad_{\mathrm{Lagr}}$ of $\Grad$ consisting of those points for which 
$\mathfrak d^* = \mathfrak d$ and the corresponding subspaces $V$ as above are 
\begin{enumerate}
\item Lagrangian subspaces of $\ker \mathfrak d$ with respect to the symplectic form $\mathcal{C}_{\mathfrak d}( \cdot, \cdot)$ and 
\item are preserved under the transformations $x \mapsto -x$.
\end{enumerate}
\end{defn}
Combining Theorems \ref{fixed} and \ref{Lagr}, we obtain the following:
\begin{thm}
The self-adjoint bispectral functions of rank $1$ are in bijective correspondence with the points of the Lagrangian adelic Grassmannian
$\Grad_{\mathrm{Lagr}}$.
\end{thm}
This theorem gives an explicit algorithmic way to construct all self-adjoint bispectral functions of rank 1.
\subsection{Classification of the self-adjoint tranformations in the Airy and Bessel cases} 
In this subsection we describe extensions of the results in \S \ref{7.2}-\ref{7.3} to the rank 2 Airy and Bessel cases.

The Sato's Grassmannian classifies the 
solutions of the KP hierarchy. We refer the reader to \cite{vM} for details. The adelic Grassmannian is naturally embedded in it.
In \cite[Sect. 2--4]{BHYcmp} it was proved that the families of (normalized) bispectral Darboux transformations of the Bessel functions $\psi_{\Be(\nu)}$ 
and Airy functions $\psi_{\Ai}$ are canonically embedded in it, forming sub-Grassmannians which we will denote by $\Gr^{\mathrm{ad}, \Be(\nu)}$ and  
$\Gr^{\mathrm{ad}, \Ai}$, respectively. (The Grassmannians $\Gr^{\mathrm{ad}, \Be(\nu)}$ and $\Gr^{\mathrm{ad}, \Ai}$ are disjoint from $\Grad$
which is also naturally embedded in the Sato's Grassmannian.)

In \cite[\S 1.4]{BHYcmp} Wilson's adjoint $a$ and sign $s$ involutions were extended to Sato's Grassamannian and it was shown that 
\eqref{Ba}--\eqref{Bs} are satisfied whenever $\wt \psi(x,y)$ and $\psi_{\exp}(x,y)$ are replaced with any pair of wave functions 
of the KP hierarchy that satisfy the transformation property \eqref{transf}. Moreover, in \cite[Sections 2 and 4]{BHYcmp}
it was proved that $a$ preserves $\Gr^{\mathrm{ad}, \Be(\nu)}$ and  $\Gr^{\mathrm{ad}, \Ai}$, 
while $s$ preserves $\Gr^{\mathrm{ad}, \Be(\nu)}$. 
Similarly to Theorem \ref{as-fixed}, using the fact that $\iota_{ \psi_{\Ai} , x} = \id$ and $\iota_{ \psi_{\Be(\nu)}, x} = \sigma$ (cf. \S \ref{subFour} and eq. \eqref{sigma}), 
one proves the following:
\begin{thm}
\label{as-fixed2}
\label{as-fixed} (a) Let $\wt\psi(x,y)$ be a (normalized) bispectral Darboux transformation from the Airy bispectral 
function $\psi_{\Ai}(x,y)$. Then the following are equivalent:
\begin{enumerate}
\item  $\wt\psi(x,y)$ is self-adjoint.
\item  The algebra $\bislf(\wt\psi)$ is closed under the formal adjoint $*$.
\item The plane $W \in \Gr^{\mathrm{ad}, \Ai}$ corresponding to $\wt\psi(x,y)$ is invariant under the adjoint involution $a$ of $\Gr^{\mathrm{ad}, \Ai}$.
\end{enumerate}

(b) Let $\wt\psi(x,y)$ be a (normalized) bispectral Darboux transformation from the Bessel bispectral 
function $\psi_{\Be(\nu)}(x,y)$ for $\nu \in \bbc \backslash \bbz$. Then the following are equivalent:
\begin{enumerate}
\item  $\wt\psi(x,y)$ is self-adjoint.
\item  The algebra $\bislf(\wt\psi)$ is closed under the formal adjoint $*$ and the transformation $\sigma$ from \eqref{sigma}.
\item The plane $W \in \Gr^{\mathrm{ad}, \Be(\nu)}$ corresponding to $\wt\psi(x,y)$ is invariant under the adjoint $a$ and sign $s$ involutions of $\Gr^{\mathrm{ad}, \Be(\nu)}$.
\end{enumerate}
\end{thm}
The (normalized) bispectral Darboux transformations from $\psi_{\Ai}(x,y)$ and  $\psi_{\Be(\nu)}(x,y)$ (for $\nu \in \bbc \backslash \bbz$) 
are constructed as follows, see \cite[Sect. 2 and 4]{BHYcmp} for details. Let $\mathfrak d$ be a differential operator which is a polynomial with constant coefficients 
in the differential operators $\mathfrak d_{\Ai, x}$ (resp. $\mathfrak d_{\Be(\nu),x}$) from Figure \ref{elementary bispectral functions}. 
(As a consequence of this, $\mathfrak d^* = \mathfrak d$.) Let 
$\mathcal{O}$ be an open connected subset of $\bbc$ such that the dimension of the kernel of $\mathfrak d$ 
in the space of holomorphic functions on $\mathcal{O}$ equals the order of $\mathfrak d$. The kernel of the operator 
$\mathfrak d$  is given in terms of derivatives of the Airy function (resp. derivatives of the Bessel functions and 
products of powers of x and logarithmic functions) by Proposition 4.9 and Lemma 2.1 of \cite{BHYcmp}.

The planes in $\Gr^{\mathrm{ad}, \Ai}$ and $\Gr^{\mathrm{ad}, \Be(\nu)}$ correspond to bispectral functions of the form 
\[
\wt\psi(x,y) = \frac{1}{q(y) p(x)} \mathfrak u \cdot \psi(x,y), \; \; 
\mbox{where} \; \; \psi(x,y) = \psi_{\Ai}(x,y), \; \; \mbox{resp.} \; \; 
\psi(x,y) = \psi_{\Be(\nu)}(x,y),
\]
$\ker  \mathfrak u$ is a subspace of  $\mathfrak d$ satisfying certain adelic type conditions \cite[Definition 2.5 and Proposition 4.9]{BHYcmp}, and 
$q(y)$, $p(x)$ are appropriate normalization polynomials.
 
We define the {\bf{Airy Lagrangian adelic Grassmannian}} $\Gr_{\mathrm{Lagr}}^{\mathrm{ad}, \Ai}$ to be the sub-Grassmannian of $\Gr^{\mathrm{ad}, \Ai}$ 
consisting of those points for which $\ker  \mathfrak u$ is a Lagrangian subspace of $\ker \mathfrak d$ with respect to the symplectic form $\mathcal{C}_{\mathfrak d}( \cdot, \cdot)$.
Similarly, define the {\bf{Bessel Lagrangian adelic Grassmannian}} $\Gr_{\mathrm{Lagr}}^{\mathrm{ad}, \Be(\nu)}$ to be the sub-Grassmannian of $\Gr^{\mathrm{ad}, \Be(\nu)}$ 
consisting of those points for which $\ker  \mathfrak u$ is a Lagrangian subspace of $\ker \mathfrak d$ with respect to the symplectic form $\mathcal{C}_{\mathfrak d}( \cdot, \cdot)$.
and is preserved under the transformations $x \mapsto -x$. 

From Theorems \ref{Lagr} and \ref{as-fixed2}, we get:

\begin{thm}
The self-adjoint bispectral Darboux transformations from the Airy bispectral function $\psi_{\Ai}(x,y)$ (resp. the Bessel bispectral functions $\psi_{\Be(\nu)}(x,y)$
for $\nu \in \bbc \backslash \bbz$) are in bijective correspondence with the points of the Airy Lagrangian adelic Grassmannian $\Gr^{\mathrm{ad}, \Ai}_{\mathrm{Lagr}}$
(resp. the Bessel Lagrangian adelic Grassmannian $\Gr_{\mathrm{Lagr}}^{\mathrm{ad}, \Be(\nu)}$).
\end{thm}
This theorem gives an explicit algorithmic way to construct all self-adjoint bispectral Darboux transformations from the Airy and Bessel bispectral functions.
\bibliographystyle{plain}

\end{document}